\documentclass[11pt]{article}
\usepackage[a4paper,margin=25mm]{geometry}
\usepackage{cmap}
\usepackage[T1]{fontenc}
\usepackage{lmodern,amsmath,amssymb,graphicx,booktabs,microtype,placeins}
\usepackage{amsthm}
\usepackage[labelfont=bf,labelsep=quad,font=small]{caption}
\usepackage[colorlinks=true,allcolors=blue]{hyperref}
\newcommand{\dd}{\mathrm{d}}
\graphicspath{{figures/}}
\title{Designing collective behaviour within a fixed coarse-grained description}
\author{Chuling Wen$^{1}$ and Jian Lu$^{1,2,*}$\\[4pt]
\small $^1$School of Mathematical Sciences, Shenzhen University, Shenzhen 518060, China\\
\small $^2$National Center for Applied Mathematics Shenzhen, Shenzhen 518055, China\\
\small $^*$Corresponding author: \href{mailto:jianlu@szu.edu.cn}{jianlu@szu.edu.cn}}
\date{}
\newtheorem{theorem}{Theorem}
\newtheorem{proposition}[theorem]{Proposition}
\newtheorem{corollary}[theorem]{Corollary}
\newtheorem{lemma}[theorem]{Lemma}

\theoremstyle{remark}
\newtheorem{remark}[theorem]{Remark}
\theoremstyle{plain}

\newcommand{\diag}{\operatorname{diag}}

\newcommand{\Real}{\operatorname{Re}}
\newcommand{\Imag}{\operatorname{Im}}
\newcounter{suppnote}
\newcommand{\suppnote}[1]{\refstepcounter{suppnote}\section*{Supplementary Note~\thesuppnote: #1}\addcontentsline{toc}{section}{Supplementary Note \thesuppnote: #1}\setcounter{subsection}{0}}
\newcommand{\TabRank}{%
\begin{tabular}{p{4.1cm}rrrrrp{3.1cm}}
\toprule
hidden coefficients $\theta$ & $\dim\theta$ & points & rank 3 & $\sigma_2/\sigma_1$ & $\sigma_3/\sigma_1$ & null directions\\
\midrule
$(\eta_2,\eta_{11},\beta_3,\beta_5)$ & 4 & 200 & 200/200 & $9.0\times10^{-2}$ & $1.3\times10^{-2}$ & $\beta_5$\\
degrees 2, 3 in both equations, $U^5$, $V^5$ & 16 & 200 & 200/200 & $4.9\times10^{-2}$ & $4.7\times10^{-3}$ & $U^5$, $V^5$ and 11 combinations\\
degrees 2, 3, and $U^4$, $V^4$, $U^5$, $V^5$, in both equations & 22 & 500 & 500/500 & $5.0\times10^{-2}$ & $6.2\times10^{-3}$ & 8 quartic/quintic and 11 combinations\\
\bottomrule
\end{tabular}}

\newcommand{\TabSilent}{%
\begin{tabular}{rlll}
\toprule
$\beta_5$ & $\eta_2=0$, stripe-seeded & $\eta_2=0.09$, stripe-seeded & $\eta_2=0.09$, hexagon-seeded\\
\midrule
0 & stripe 0.0744 & stripe 0.0763 & hexagon 0.0402\\
0.5 & stripe 0.0740 & stripe 0.0759 & hexagon 0.0398\\
1 & stripe 0.0737 & stripe 0.0755 & hexagon 0.0395\\
2 & stripe 0.0731 & stripe 0.0748 & hexagon 0.0388\\
3 & stripe 0.0725 & stripe 0.0742 & hexagon 0.0382\\
5 & stripe 0.0714 & stripe 0.0729 & hexagon 0.0372\\
8 & stripe 0.0700 & stripe 0.0713 & hexagon 0.0359\\
12 & stripe 0.0683 & stripe 0.0695 & hexagon 0.0346\\
\midrule
leading order & stripe 0.0744 & stripe 0.0770 & hexagon 0.0410\\
\bottomrule
\end{tabular}}

\newcommand{\TabTwins}{%
\begin{tabular}{lrrrllrrr}
\toprule
target & $A_{\rm s}$ & $x$ & $h/g$ & seed & pattern & in main state & s.d.\ (\%) & offset (\%)\\
\midrule
T1 & 0.070 & 0.30 & 2.4 & hexagon & stripe & 4/4 & 0.58 & $-$5.9\\
T1 & 0.070 & 0.30 & 2.4 & stripe & stripe & 5/5 & 0.74 & $-$6.1\\
T2 & 0.060 & 0.70 & 2.0 & hexagon & hexagon & 5/5 & 1.83 & $-$0.6\\
T2 & 0.060 & 0.70 & 2.0 & stripe & stripe & 5/5 & 0.80 & $-$1.2\\
T3 & 0.080 & 0.90 & 2.6 & hexagon & hexagon & 3/5 & 0.11 & $-$6.1\\
T3 & 0.080 & 0.90 & 2.6 & stripe & stripe & 5/5 & 5.16 & $-$10.9\\
T4 & 0.060 & 2.00 & 2.2 & hexagon & hexagon & 5/5 & 0.70 & $-$15.5\\
T4 & 0.060 & 2.00 & 2.2 & stripe & hexagon & 5/5 & 0.70 & $-$15.5\\
T5 & 0.060 & 0.00 & 2.2 & hexagon & stripe & 4/4 & 1.38 & $-$2.1\\
T5 & 0.060 & 0.00 & 2.2 & stripe & stripe & 5/5 & 3.63 & $-$3.8\\
T6 & 0.050 & 1.20 & 2.4 & hexagon & hexagon & 5/5 & 0.85 & $-$6.9\\
T6 & 0.050 & 1.20 & 2.4 & stripe & stripe & 5/5 & 1.05 & $-$5.6\\
\bottomrule
\end{tabular}}

\newcommand{\TabNRQS}{%
\begin{tabular}{rrrrrrrrr}
\toprule
 & & & & & & \multicolumn{3}{c}{minimum budget}\\
\cmidrule(l){7-9}
$M$ & hidden dim. & $\Gamma_b$ & $D_M^{(33)}$ & area & reachable & B & P & F\\
\midrule
2 & 30 & 2 & $5.31\times10^{-2}$ & $2.52\times10^{-3}$ & 1.000 & 0.0258 & 0.0337 & 0.0288\\
3 & 29 & 2 & $4.28\times10^{-2}$ & $1.28\times10^{-3}$ & 1.000 & 0.0482 & 0.0463 & 0.0367\\
4 & 28 & 2 & $2.71\times10^{-2}$ & $6.00\times10^{-4}$ & 1.000 & 0.0983 & 0.101 & 0.0812\\
5 & 27 & 2 & $1.35\times10^{-2}$ & $2.32\times10^{-4}$ & 1.000 & 0.214 & 0.233 & 0.19\\
6 & 26 & 2 & $5.39\times10^{-3}$ & $7.36\times10^{-5}$ & 1.000 & 0.52 & 0.59 & 0.486\\
7 & 25 & 2 & $1.95\times10^{-3}$ & $2.40\times10^{-5}$ & 0.245 & 1.07 & 1.24 & 1.03\\
8 & 24 & 2 & $1.12\times10^{-3}$ & $1.23\times10^{-5}$ & 0.000 & 1.32 & 1.56 & 1.29\\
9 & 23 & 2 & $1.07\times10^{-3}$ & $6.87\times10^{-6}$ & 0.000 & 1.33 & 1.56 & 1.29\\
10 & 22 & 2 & $1.06\times10^{-3}$ & $3.14\times10^{-6}$ & 0.000 & 1.92 & 2.14 & 1.75\\
\bottomrule
\end{tabular}}

\newcommand{\TabTrunc}{%
\begin{tabular}{r*{9}{r}}
\toprule
harmonics $H$ & \multicolumn{9}{c}{fraction of the target box reachable within $\|h\|_\infty<1$, $M=2,\ldots,10$}\\
\midrule
9 & 1.00 & 1.00 & 0.90 & 0.01 & 0.00 & -- & -- & -- & --\\
13 & 1.00 & 1.00 & 1.00 & 1.00 & 0.08 & 0.00 & 0.00 & 0.00 & 0.00\\
17 & 1.00 & 1.00 & 1.00 & 1.00 & 1.00 & 0.00 & 0.00 & 0.00 & 0.00\\
25 & 1.00 & 1.00 & 1.00 & 1.00 & 1.00 & 0.07 & 0.00 & 0.00 & 0.00\\
33 & 1.00 & 1.00 & 1.00 & 1.00 & 1.00 & 0.24 & 0.00 & 0.00 & 0.00\\
\bottomrule
\end{tabular}}

\newcommand{\TabSeeds}{%
\begin{tabular}{lrrr}
\toprule
initial condition & kernel--IC pairs & same pattern, three seeds & max.\ relative amplitude range\\
\midrule
stripe seed & 16 & 16 & $1.1\times10^{-5}$\\
hexagon seed & 16 & 16 & $1.8\times10^{-8}$\\
noise & 16 & 11 & $4.2\times10^{-2}$\\
\bottomrule
\end{tabular}}

\newcommand{\TabKinetic}{%
\begin{tabular}{lrrrrrr}
\toprule
target & $M$ & $I^*\ (\times10^{-3})$ & $\Omega^*$ & budget & power error (\%) & $\Delta\Omega\ (\times10^{-6})$\\
\midrule
B & 2 & 0.50000 & 0.29740 & 0.035 & $+$0.293 & $-$5.7\\
P & 2 & 0.60000 & 0.29740 & 0.064 & $+$0.299 & $-$7.3\\
F & 2 & 0.50000 & 0.29760 & 0.062 & $+$0.242 & $-$4.9\\
B & 3 & 0.50000 & 0.29740 & 0.148 & $+$0.335 & $-$5.9\\
P & 3 & 0.60000 & 0.29740 & 0.074 & $+$0.309 & $-$7.3\\
F & 3 & 0.50000 & 0.29760 & 0.063 & $+$0.245 & $-$4.9\\
B & 4 & 0.50000 & 0.29740 & 0.805 & $+$0.435 & $-$6.8\\
P & 4 & 0.60000 & 0.29740 & 0.556 & $+$0.362 & $-$8.0\\
F & 4 & 0.50000 & 0.29760 & 0.389 & $+$0.280 & $-$5.3\\
\midrule
B$^{33}$ & 5 & 0.50000 & 0.29740 & 0.214 & $+$0.927 & $-$28.8\\
P$^{33}$ & 5 & 0.60000 & 0.29740 & 0.233 & $+$0.673 & $-$29.1\\
F$^{33}$ & 5 & 0.50000 & 0.29760 & 0.190 & $+$0.481 & $-$18.3\\
B$^{33}$ & 6 & 0.50000 & 0.29740 & 0.520 & $-$0.273 & $-$37.3\\
P$^{33}$ & 6 & 0.60000 & 0.29740 & 0.590 & $-$0.765 & $-$34.0\\
F$^{33}$ & 6 & 0.50000 & 0.29760 & 0.486 & $-$0.765 & $-$20.1\\
\midrule
H0 & 3 & 0.69755 & 0.29772 & 0.344 & $+$0.301 & $-$8.3\\
H0 & 4 & 0.69755 & 0.29772 & 0.398 & $+$0.289 & $-$8.0\\
H1 & 3 & 0.50300 & 0.29759 & 0.063 & $+$0.247 & $-$5.0\\
H1 & 4 & 0.50300 & 0.29759 & 0.395 & $+$0.283 & $-$5.4\\
H2 & 3 & 0.58771 & 0.29772 & 0.207 & $+$0.247 & $-$6.0\\
H2 & 4 & 0.58771 & 0.29772 & 0.208 & $+$0.246 & $-$6.0\\
H3 & 3 & 0.62047 & 0.29748 & 0.098 & $+$0.282 & $-$7.2\\
H3 & 4 & 0.62047 & 0.29748 & 0.333 & $+$0.312 & $-$7.7\\
H4 & 3 & 0.57417 & 0.29737 & 0.089 & $+$0.327 & $-$7.1\\
H4 & 4 & 0.57417 & 0.29737 & 0.690 & $+$0.400 & $-$7.9\\
H5 & 3 & 0.64795 & 0.29731 & 0.081 & $+$0.338 & $-$8.6\\
H5 & 4 & 0.64795 & 0.29731 & 0.625 & $+$0.399 & $-$9.4\\
H6 & 3 & 0.64890 & 0.29740 & 0.090 & $+$0.305 & $-$8.1\\
H6 & 4 & 0.64890 & 0.29740 & 0.427 & $+$0.344 & $-$8.7\\
H7 & 3 & 0.51563 & 0.29746 & 0.085 & $+$0.298 & $-$5.8\\
H7 & 4 & 0.51563 & 0.29746 & 0.641 & $+$0.367 & $-$6.5\\
\bottomrule
\end{tabular}}

\renewcommand{\figurename}{Fig.}
\begin{document}
\maketitle

\begin{abstract}
\noindent
A fixed coarse-grained description need not uniquely determine the underlying microscopic
rules\cite{wilson,marchetti,bechinger,noid}. Variations in these rules can alter collective behaviour
while leaving the retained observables unchanged. Here we develop a method to identify and use such
variations for collective-state design. We construct families of microscopic models with exactly
matched coarse observables and use the response derivative within each family to identify
independent first-order output changes. In a nonreciprocal active mixture\cite{fruchart,you,duan}, the largest first-order change in
interfacial growth rate under a pointwise kernel budget also determines the leading minimax error of
predictions based only on the matched data. Near a Hopf bifurcation, this response map gives an
initial turning-kernel design, which is refined by jointly solving for the periodic kinetic state
and the kernel. Fixed-kernel integrations realize prescribed oscillation powers and frequencies while
preserving the selected angular relaxation rates. In a reaction--diffusion
family\cite{turing,crosshohenberg}, cubic amplitude equations guide the design of stripe and hexagon
amplitudes while the full linear operator remains fixed. Additional constraints reduce the available
response even when its rank is unchanged, whereas the matched angular moments retain their exact
relaxation rates. These results provide a constructive route to collective-state design within a
prescribed coarse-grained description.
\end{abstract}

\noindent
Complex systems are understood through coarse variables that keep collective properties and
discard microscopic detail\cite{wilson,marchetti,bechinger,noid}. This compression makes
effective theories possible, and renormalization and projection-operator methods have made it
precise for prediction: they identify which details matter for the retained variables and turn
the rest into memory and noise\cite{machta,transtrum,zwanzig,mori}. Many microscopic realizations
thereby become one coarse description.

In such a theory, two realizations with the same coarse description are treated as the same
system. A self-organized interface, a pattern that grows from an instability or an oscillation that
saturates at finite amplitude can, however, depend on microscopic rules to which the coarse
variables are insensitive, so that two systems with the same coarse description reach different
collective states. The microscopic freedom that the coarse description discards then acts as a set
of dials for the collective state.

Formally, a coarse description is a map $\mathcal C:\mathcal M\to\mathcal B$
from microscopic models to selected observables, and the realizations it cannot tell apart are the
preimage of one value, the fibre $\mathcal F_b=\{m:\mathcal C(m)=b\}$. At a regular point,
$\ker D\mathcal C$ is the tangent space of the fibre and consists of perturbations that preserve the
coarse observables to first order. The fibre is a hidden design space whenever a collective
response $R$ varies along it. Using it raises three questions: which collective functions vary
along the fibre, and how many of them independently; how large a change of the microscopic rule
each requires; and which are fixed by the coarse description itself.

We address these questions in two systems with different retained descriptions. In a
nonreciprocal active mixture the coarse description is a finite list of angular relaxation rates; in
a reaction--diffusion system it is the entire linear operator. Both examples use exactly matched
families and local response maps for inverse design\cite{torquato,sherman,molesky}. In the active
mixture, a pointwise kernel budget quantifies the available response, and the kernel is refined
together with the periodic state. In the reaction--diffusion family, the amplitude-equation
coefficients guide the choice of nonlinear reaction terms.

\section*{Bulk matching and interfacial response}
Run-and-tumble particles in a nonreciprocal quorum-sensing
mixture\cite{fruchart,saha,you,brauns,duanprl,duan,dinelliNR,tailleur08}, of the kind realised
with light-controlled colloids\cite{baeuerle,lefranc}, move at a speed set by the sensed densities
and reorient by jumps drawn from an even turning kernel $q(\varphi)$. The kernel enters the
kinetic equation only through its angular relaxation rates $\gamma_m$, and the bulk depends on these
rates in a fixed order. The angular harmonics form a ladder coupled by spatial transport, a path
from the density mode to harmonic $m$ and back needs $2m$ spatial derivatives, and the density
branch of the uniform state has the expansion
\begin{equation}
 \lambda(k)=\sum_{j\geq1}\ell_j k^{2j},\qquad \ell_j=\ell_j(\gamma_1,\ldots,\gamma_j).
 \label{eq:bulkblind}
\end{equation}
Kernels that share $\gamma_1,\ldots,\gamma_M$ are therefore indistinguishable in the bulk through
order $k^{2M}$ and generically first differ at $k^{2M+2}$ (Fig.~\ref{fig:one}b; Methods). The same
matching fixes the coefficients of the formal nonlinear long-wave closure through gradient order
$2M$ (Supplementary Note~2). The known
equivalence of active Brownian and run-and-tumble particles\cite{cates,solon,dinelli} is the case
$M=1$.

An interface carries order-one gradients across its width, so the gradient expansion that orders the bulk has no small parameter there, and every angular sector is
in play. The sensitivity of the band's leading eigenvalue to the rates is dominated by the second
harmonic, and most of it arises from the self-consistent reshaping of the band that a change of
$\gamma_2$ induces rather than from the damping itself (Extended Data Fig.~5a). Bulk diffusion, the
$k^2$ term, does not depend on $\gamma_2$, whereas the interface is most sensitive to it. We
therefore compare strictly positive kernels with identical $\gamma_1$ and $\gamma_2$
(Fig.~\ref{fig:one}a) in the same mixture, at the same composition and with the same speed rule.
Under one kernel the band travels steadily; under another it breathes about a fixed mean shape
(Fig.~\ref{fig:one}c). The coarse
description does not determine the collective state.

\begin{figure}[t]
\centering
\includegraphics[width=\linewidth]{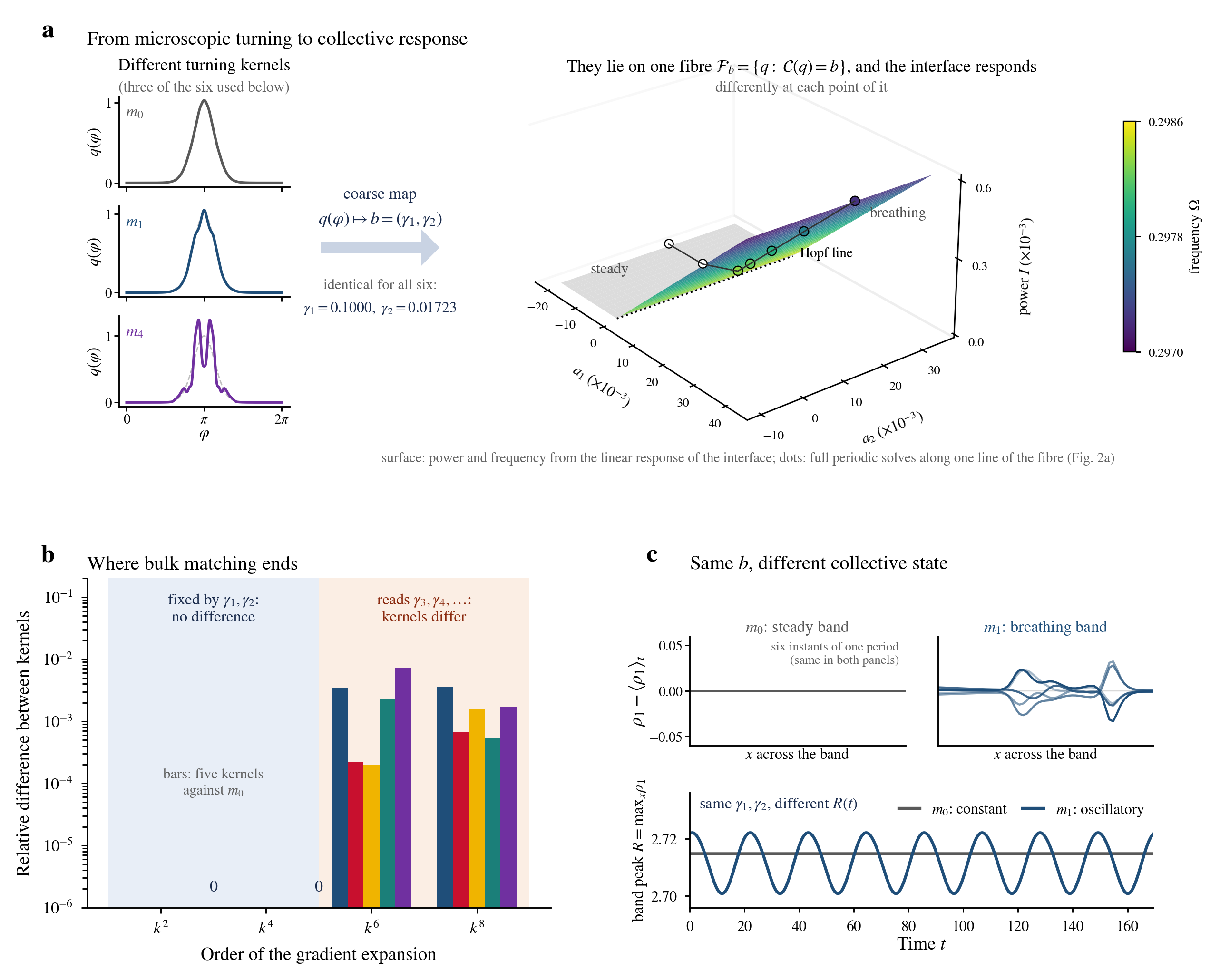}
\caption{\textbf{One coarse description, many collective states.}
\textbf{a}, Strictly positive turning kernels with identical normalization, $\gamma_1$ and
$\gamma_2$ (three shown against the base kernel $q_0$, dashed) are sent by the coarse map to one
value $b$ and lie on its fibre. Right, the interface's response over two hidden coordinates
$(a_1,a_2)$ on the fibre: height is the oscillation power and colour the frequency predicted from
the linear response of the band; dots are full periodic solves of the kinetic equations along one
line of the fibre (open, steady side; filled, oscillatory side).
\textbf{b}, Where bulk blindness ends. Relative difference between the coefficients of the
gradient expansion of the bulk density dispersion (Methods) for five kernels and the reference: the
coefficients of $k^2$ and $k^4$ are fixed by $\gamma_1$ and $\gamma_2$ and coincide exactly; the
first unmatched rate enters at $k^6$, where the kernels differ.
\textbf{c}, Different interfacial dynamics. Top, density of species 1 across the band minus its
time average, at six instants of one period, under the first two kernels of \textbf{a}: flat under
$m_0$, a standing wave under $m_1$. Bottom, the band peak against time: constant under $m_0$,
oscillatory under $m_1$ (forward integrations of the kinetic equations).}
\label{fig:one}
\end{figure}

\section*{Response capacity and prediction loss}
We write kernels as $q=q_0(1+h)$, with $q_0$ the reference kernel, and measure a change by the
pointwise budget $\varepsilon=\|h\|_\infty$; a budget below one keeps the kernel positive. The
outputs are the oscillation power $I$, the variance of the spatially aligned density about its time
average, and the frequency $\Omega$ of the breathing band. We quantify the local response by its rank
and by its magnitude under this budget. The rank is the local response capacity
\begin{equation}
 \Gamma_b=\operatorname{rank}\big(D\Phi|_{\ker D\mathcal C}\big),\qquad \Phi=(I,\Omega),
 \label{eq:capacity}
\end{equation}
the number of independent first-order output changes within the fibre; for the breathing band it is
two. Along a line of kernels inside the fibre that crosses the Hopf point of the band, the power and
frequency follow the linear response map smoothly, and the periodic orbits on the oscillatory side
are stable within the one-dimensional periodic cell (Fig.~\ref{fig:two}a). The two directions of the linear
response act on growth rate and frequency separately, and the responses reachable within a
kernel budget form a two-dimensional region whose size is set by that budget; three prescribed
targets require kernel budgets of 2.6--3.4\% (Fig.~\ref{fig:two}b).

For the interfacial growth rate, the largest first-order response that a unit budget can still
produce after $M$ rates are matched is
\begin{equation}
 D_M=\inf_a\int q_0\,|\Psi-a\cdot C| ,
 \label{eq:DM}
\end{equation}
with $\Psi$ the sensitivity density of the interfacial growth rate and $C=(1,\cos\varphi,\ldots,
\cos M\varphi)$ the matched moments. At the working point $D_M$ falls from 0.17 to 0.04 as $M$
increases from 0 to 4 (Fig.~\ref{fig:two}c): each matched rate lowers $D_M$, which stays finite. Any predictor that uses only the matched rates returns one value on the whole
fibre, and its smallest worst-case error on kernels within a budget $\varepsilon$ is
$\varepsilon D_M$ to leading order\cite{micchelli} (Supplementary Note~4). $D_M$ thus bounds both the
growth-rate response available within a budget and the error of coarse prediction; for power and
frequency together the corresponding object is the reachable set of Fig.~\ref{fig:two}b
(Supplementary Note~5).

\begin{figure}[t]
\centering
\includegraphics[width=\linewidth]{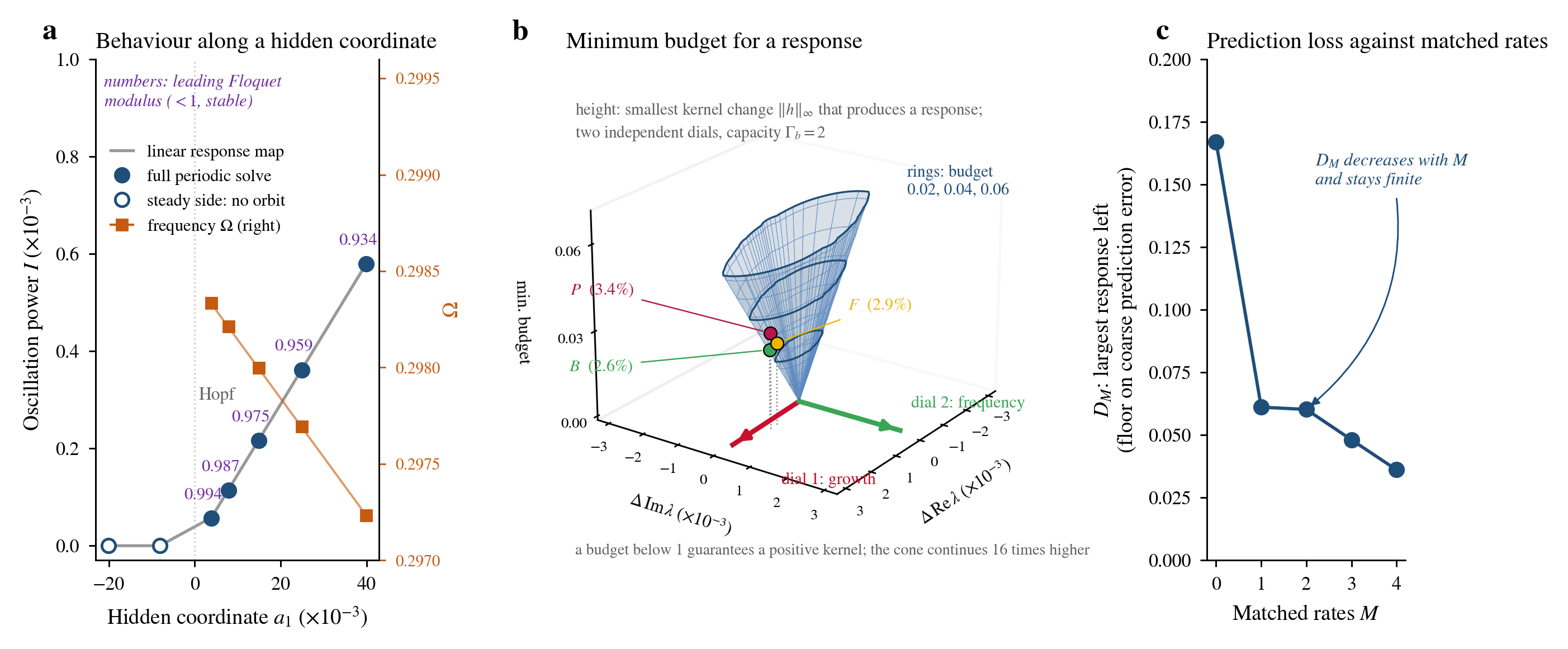}
\caption{\textbf{Response capacity and prediction loss.}
\textbf{a}, Oscillation power (left axis) and frequency (right axis) of the band along a line of
kernels inside the fibre of $(\gamma_1,\gamma_2)$ that crosses the Hopf point: linear response
map (lines), full periodic solves of the kinetic equations (symbols) and the leading Floquet
modulus of each orbit (labels); on the steady side no orbit is found on the tracked branch.
\textbf{b}, Minimum kernel budget. Height is the smallest kernel change $\|h\|_\infty$ that
produces a given change of growth rate and frequency, with two rates matched, in a 33-harmonic
family at the Hopf reference; the two arrows in the base plane are the responses of the two hidden
directions, one acting on growth rate and one on frequency alone ($\Gamma_b=2$). Three prescribed
targets of power and frequency (B, P, F) sit at the budgets they require, 2.6--3.4\%; a budget
below one guarantees a positive kernel.
\textbf{c}, Prediction loss $D_M$ of Eq.~\eqref{eq:DM} at the working point against the number of
matched rates.}
\label{fig:two}
\end{figure}

\section*{A reaction--diffusion fibre}
In the active mixture the coarse description stops at a finite order, and a finer description
would resolve part of the fibre. The reaction--diffusion family
\begin{equation}
 \partial_tw=Jw+D\nabla^2w+N_\theta(w),\qquad N_\theta(0)=0,\quad DN_\theta(0)=0,
 \label{eq:rd}
\end{equation}
with $w=(U,V)$, fixed $J$ and $D$, and any polynomial nonlinearity vanishing to second order at
the uniform state, has no such truncation\cite{turing,crosshohenberg}. Every member has the same
linear operator and hence the same dispersion relation at every wavenumber
(Fig.~\ref{fig:three}a), and even a complete linear-stability measurement cannot tell the members
apart. A reaction scheme expanded about its homogeneous state has this structure, with
$J$ and $D$ the linearized kinetics and transport and the coefficients of $N_\theta$ the feedbacks
of the autocatalytic step and its saturation\cite{grayscott,epstein}.

The fibre of $(J,D)$ is infinite-dimensional. Through cubic order, the three-mode amplitude
equations\cite{crosshohenberg,hoyle} depend on the reaction law only through the resonant quadratic
coupling $a$ and the cubic saturation coefficients $g$ and $h$; quartic and quintic terms do not
enter these coefficients, and the map to $(a,g,h)$ has rank three at each of 500 random reaction laws
(Fig.~\ref{fig:three}b; Extended Data Fig.~4). Reaction laws matched in $(a,g,h)$ therefore share
the same reduced equations. Simulations from stripe and hexagon initial conditions test how far this
carries to the full system. Three reaction laws of one target, with nonlinearity in the activator
equation only, in the inhibitor equation only or at random in both, reach the same pattern type with
amplitudes that differ by at most 4\% (Fig.~\ref{fig:three}c), and over six targets 88 of 90 runs end
in the predicted pattern type. The equivalence holds at leading order only: quintic terms shift the
amplitude by up to 14\%, and two reaction laws of one target also reach a hexagonal branch of larger
amplitude (Supplementary Note~11). Hexagonal and striped Turing patterns are both observed in
chemical\cite{ouyang} and biological\cite{kondo} systems; in this family the choice between them and
their amplitude are set, at leading order, by directions that linear characterization cannot
resolve.

\begin{figure}[t]
\centering
\includegraphics[width=\linewidth]{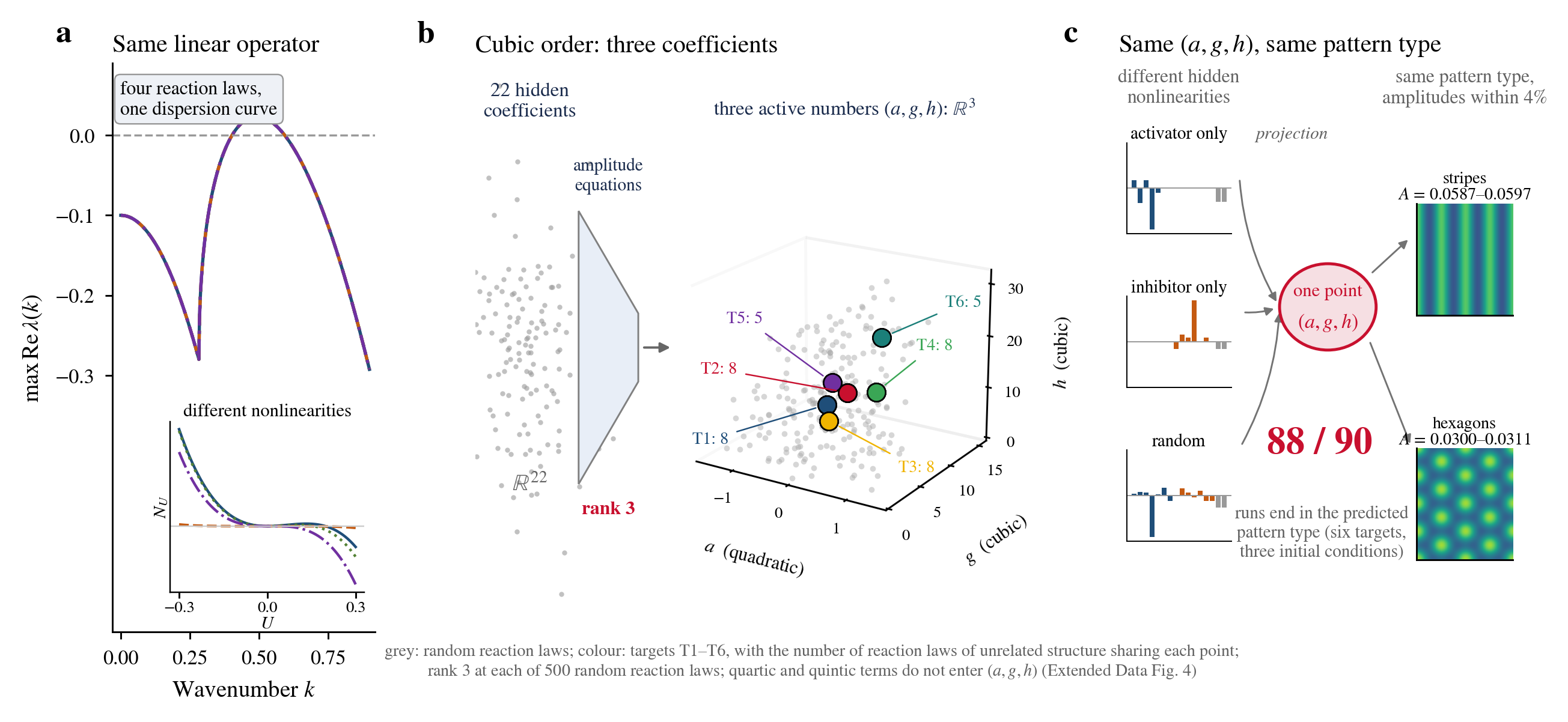}
\caption{\textbf{Reaction laws with a common linear operator.}
\textbf{a}, Four reaction laws of the family of Eq.~\eqref{eq:rd} have the same linear operator
$(J,D)$ and hence one dispersion curve; inset, their nonlinearities differ.
\textbf{b}, Through cubic order the amplitude equations depend on three coefficients. Grey, random
reaction laws with 22 nonlinear coefficients, mapped to the coefficients $(a,g,h)$ of the three-mode
amplitude equations; colour, six target points, each shared by five to eight reaction laws of
unrelated structure that were projected onto it.
\textbf{c}, Matched coefficients, same reduced equations. Three reaction laws of one target with
nonlinearity in the activator equation only, in the inhibitor equation only, or random in both
(bars, their sixteen coefficients) share one point $(a,g,h)$ and produce the same pattern types,
stripes from a stripe seed and hexagons from a hexagon seed, with amplitudes that differ by at most
4\%. Over the six targets, 88 of 90 runs from stripe, hexagon and noise initial conditions end in the
predicted pattern type; the two exceptions are hexagon-seeded runs that stay hexagonal where only
stripes are predicted to be stable.}
\label{fig:three}
\end{figure}

\section*{Designed oscillations and their limits}
For prescribed oscillation power and frequency $(I^*,\Omega^*)$, the Hopf response map first gives
the required changes in interfacial growth rate and linear frequency. We minimize the pointwise
relative kernel change over the moment-preserving subspace, so that the normalization, $\gamma_1$
and $\gamma_2$ of the reference are kept exactly. The resulting kernel initializes a joint solve of
the periodic kinetic state and two kernel coordinates, with the power and frequency imposed as
constraints; the solve moves the coordinates by at most 1.6\% (Methods). We then hold the kernel
fixed and run the kinetic equations for about 380 oscillation periods from the steady interface of
the reference kernel plus a small perturbation (Fig.~\ref{fig:four}a). Under the reference kernel
the perturbation decays. Under each designed kernel it grows, saturates and settles on a breathing
band whose frequency equals the target to $2\times10^{-10}$ and whose power equals it to 0.03\% at the
discretization used for the design (Fig.~\ref{fig:four}b). A fourth design reaches the second target
while preserving $\gamma_1$ through $\gamma_4$, and the four orbits are stable to perturbations
within the one-dimensional periodic cell (Extended Data Fig.~2). In the reaction--diffusion family
the amplitude equations are inverted instead: a prescribed pair of stripe and hexagon amplitudes is
reached in closed form to within a few per cent and, after one correction of the targets by the
measured offset, to within 1\% for the first pair (Extended Data Fig.~1).

We examine how additional moment constraints and larger target amplitudes restrict the designs.
Each further matched rate removes one hidden direction; the capacity stays $\Gamma_b=2$ from $M=2$ to
$M=10$, but the region reachable to first order at unit kernel budget shrinks by
a factor of two to three per rate, so that a fixed box of targets that lies within unit budget up to
$M=6$ lies outside it from $M=8$ (Fig.~\ref{fig:four}c--e); a budget below one guarantees a positive
kernel, and beyond it positivity has to be checked kernel by kernel. Raising the target power eightfold pushes the
kernel budget from 7\% to 46\%, and the linear design degrades gradually before the budget
reaches one (Extended Data Fig.~2). In the reaction--diffusion family the cubic amplitude equations
lose accuracy near the edge where stripes are lost, and a second hexagonal branch that they do not
describe appears there (Extended Data Fig.~1).

\begin{figure}[t]
\centering
\includegraphics[width=\linewidth]{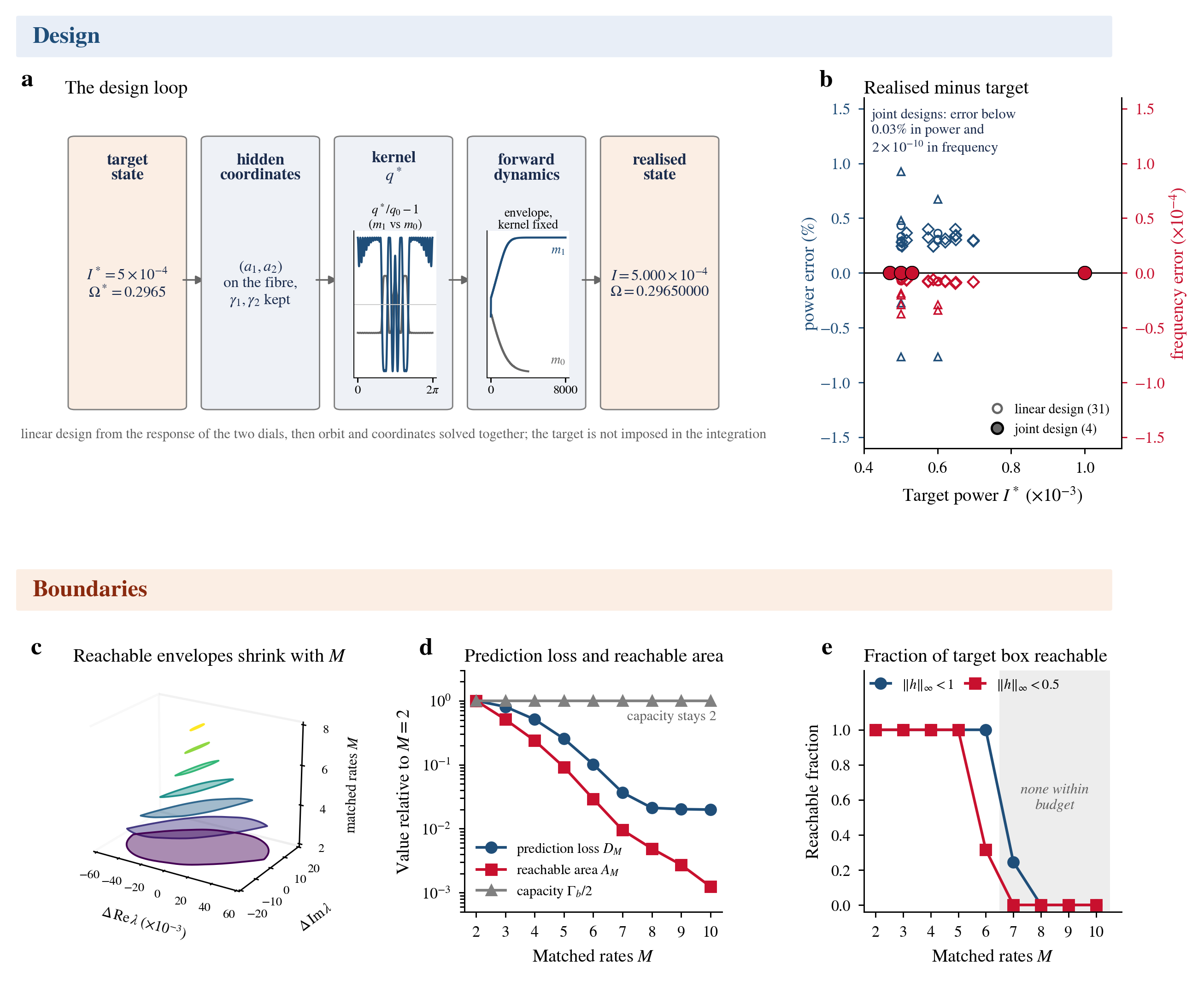}
\caption{\textbf{Designed oscillations and their limits.}
\textbf{a}, The design loop for the first target. A prescribed power and frequency are converted
into two hidden coordinates on the fibre, hence into a kernel $q^*$ with the same $\gamma_1$ and
$\gamma_2$ as the reference $m_0$ (inset, deviation of both from $q_0$); the kernel is then held
fixed in a forward integration of the kinetic equations in which the target is not imposed
(inset, envelope of the critical-mode amplitude under $m_0$ and $q^*$ from the same initial state),
and the realised power and frequency are read from the saturated state.
\textbf{b}, Realised minus target against target power: relative power error in per cent (blue) and
frequency error in units of $10^{-4}$ (red). Filled, the four joint designs after independent
integration; open, 31 linear designs tested by full periodic solves with the kernel fixed.
\textbf{c}, Regions of growth-rate and frequency responses reachable at unit kernel budget, stacked
by the number $M=2,\ldots,8$ of matched rates.
\textbf{d}, Prediction loss and reachable area relative to $M=2$; the capacity stays two.
\textbf{e}, Fraction of 400 targets in a fixed box of power and frequency reachable, to first order,
within kernel budgets $\|h\|_\infty$ below one (which guarantees positivity) and below one half.}
\label{fig:four}
\end{figure}

\section*{Conclusion}
For a coarse description given by matched angular moments, the matching constraints are linear in
the turning kernel, so the family of kernels that keeps them can be parametrized exactly. Within this
family, the rank of the restricted response gives the number of independent first-order output
changes, and the prediction loss gives the largest first-order change of the interfacial growth rate
that a pointwise budget allows, together with the leading error of any predictor that uses the
matched data alone. Near the Hopf point these quantities give an initial design, the joint solve of
state and kernel refines it, and a fixed-kernel integration confirms it: the power and frequency of
an interfacial oscillation are prescribed with the bulk rates unchanged. In the reaction--diffusion
family the cubic amplitude equations take the place of the response map and set pattern morphology
and amplitude, at leading order, with the whole linear operator unchanged. The construction requires
a constraint-preserving parametrization and a computable response map; programmable
motility\cite{palacci,frangipane,lavergne} and chemical pattern-forming systems\cite{ouyang} offer
experimental control of the corresponding microscopic rules. The resulting constructions separate
the responses that remain adjustable under a fixed description from those constrained by the
available budget or fixed by the retained dynamics, such as the decay of the matched angular moments
(Supplementary Note~6).

\FloatBarrier
\section*{Methods}
\paragraph{Kinetic model.}
For the one-dimensional band, the kinetic field of species $s$ satisfies
\[
 \partial_t F_s=-\partial_x(v_s\cos\theta\,F_s)+D_t\partial_x^2F_s
 +\alpha\int_0^{2\pi}q(\varphi)\,[F_s(x,\theta-\varphi)-F_s(x,\theta)]\,\dd\varphi,
 \qquad \rho_s=\int_0^{2\pi}F_s\,\frac{\dd\theta}{2\pi},
\]
with angular relaxation rates $\gamma_m=\alpha(1-\widehat q_m)$ and
$\widehat q_m=\int_0^{2\pi}q(\varphi)\cos m\varphi\,\dd\varphi$. The species speeds are
$v_s=\prod_u[1+0.7\tanh(\eta_{su}(\widetilde\rho_u-1)/0.7)]$ with
$\eta=\left(\begin{smallmatrix}-2&0.5\\-0.5&-2\end{smallmatrix}\right)$, where $\widetilde\rho_u$
is the density sensed through the radial kernel $K(r)=3(1-r)/\pi$ for $0\le r\le1$. The periodic
cell has length $L=40/3$, mean densities $(0.75,1.25)$ and $D_t=0.01$; all changes of turning law
keep these parameters fixed. The oscillation power $I$ is the density variance of the spatially
aligned orbit about its temporal mean.

\paragraph{Coarse-grained variables and interfacial response.}
In the scalar constant-speed hierarchy the leading difference between two kernels matched through
$M$ rates is
\begin{equation}
 \Delta\lambda(k)=\frac{2(-1)^{M+1}v^{2M+2}}{4^{M+1}\prod_{n=1}^{M}\gamma_n^2}\,
 \Delta(\gamma_{M+1}^{-1})\,k^{2M+2}+O(k^{2M+4}),
 \label{eq:bulkfirstdifference}
\end{equation}
which follows from the angular continued fraction\cite{risken} (Supplementary Note 1); the
coefficients of Fig.~\ref{fig:one}b are obtained by fitting the density branch of this hierarchy, computed
from the first 80 rates of each kernel, to a polynomial in $k^2$ on $k\le0.03$. For the two-species system the hierarchy becomes a matrix ladder and the
same order counting holds, and for the formal nonlinear long-wave closure\cite{levermore,maddu}
the coefficient of $\partial_T\rho_s$ at gradient order $2j$ depends only on
$\gamma_1,\ldots,\gamma_j$, with the first difference at order $2M+2$ given in closed form and
verified by an independent series computation through $M=4$ (Supplementary Note 2). These
asymptotic statements are not uniform approximations to a finite-width
interface\cite{peraud,zhaoyong}. The interfacial sensitivity
$W_m=\partial\,\mathrm{Re}\,\lambda/\partial\gamma_m$ was obtained by first-order perturbation of
the band's leading eigenvalue\cite{kato} including the change of the self-consistent band; at
64 angular modes its second-harmonic value $-10.6$ separates into $+0.46$ of direct damping and
$-11.0$ of shape feedback, and $|W_2|/|W_1|\simeq5$ at both 64 and 128 modes (Supplementary
Note 12). The kernels of Fig.~\ref{fig:one} share the normalization, $\gamma_1$ and $\gamma_2$ of
the base kernel $q_0$ to $10^{-18}$; a second witness pair whose kernels differ by a factor of 3.5
gives the same verdict at angular resolutions from 96 to 512 (Supplementary Note 13).

\paragraph{Active-matter simulations.}
Bands are computed pseudo-spectrally\cite{trefethen} in one Bloch cell, by time relaxation with
drift projection or by a bordered dense Newton solve\cite{govaerts} at fixed masses and phase.
Spectra of the travelling-wave linearisation\cite{sandstede,beyn} use block subspace iteration on
the exponentials of the linearised operator and of its adjoint. Periodic orbits are solved as
boundary-value problems in time by preconditioned Newton--Krylov iteration\cite{knoll} with
GMRES\cite{saad}; time stepping uses Lawson exponential Runge--Kutta integrators\cite{hochbruck},
and Floquet multipliers come from Arnoldi iteration on the period map. Particle simulations use
cloud-in-cell particle-mesh sensing\cite{hockney}. With orientation-independent speed and an even
kernel, angular reflection commutes with the one-dimensional evolution and the odd angular sector
contracts, so odd perturbations supply no additional unstable sector in one dimension
(Supplementary Note 7). The forward integrations of Fig.~\ref{fig:four}a,b use the same stepper with
1,024 steps per target period over 8,000 time units, started from the steady interface of the
reference kernel plus a small multiple of the critical eigenvector; frequency and power are read
from a six-harmonic fit over the last 16 periods.

\paragraph{Response capacity and prediction loss.}
Write the first-order response of the interfacial growth rate to a relative change $u$ of the
kernel as $\dd R[q_0u]=\int q_0\Psi u$; matching $M$ rates imposes $\int q_0uC=0$, and for
$\|u\|_\infty\le1$ the largest response is $D_M$ of Eq.~\eqref{eq:DM}. Any predictor using only
the matched rates has smallest worst-case error $\varepsilon D_M+O(\varepsilon^2)$ on kernels
$q_0(1+h)$ with $\|h\|_\infty\le\varepsilon$ (Supplementary Note 4). The values in
Fig.~\ref{fig:two}c use the response table at 128 angular modes truncated at harmonic $m=8$. For a
vector of outputs $\Phi$ the local response capacity is
$\Gamma_b=\operatorname{rank}(D\Phi|_{\ker D\mathcal C})$. The transect of Fig.~\ref{fig:two}a is
a line of kernels crossing the Hopf line of the linear map; full periodic solves at its five
oscillatory points return the power within 1.2\% of the linear response map and the frequency
within $10^{-5}$, with leading Floquet moduli between 0.934 and 0.994, and a local predictor built
on the Hopf normal form\cite{kuznetsov} reproduces full kinetic integrations to better than 1\% in
power at kernel budgets up to 8\% (Supplementary Note 13). The response surface of
Fig.~\ref{fig:one}a is this normal form evaluated on the linear response of the two hidden
directions; it reproduces the transect's linear predictions exactly.

\paragraph{Inverse design.}
The kernel family of Fig.~\ref{fig:four} is $q=q_0(1+h)$, $h=\sum_{j<33}c_j\cos j\varphi$, at
the critical Hopf reference, with the normalization, $\gamma_1$ and $\gamma_2$ of $q_0$ imposed
through the exact Bessel moments of $q_0\cos j\varphi$; the two hidden coordinates multiply the
minimum-budget direction of the linear design for the target and the direction of equal budget
perpendicular to it in response space. The joint problem solves the periodic orbit, represented by
temporal Fourier collocation on 17 nodes at $N=A=128$, together with the two coordinates, with the
phase condition, the target power and the target frequency as additional equations, by Newton
iteration with a Schur complement on the coordinates (Supplementary Note 9); all seven solves
converged in three or four iterations to residual norms below $4\times10^{-11}$ with the matched
moments held to $10^{-18}$. The constraint sweep of Fig.~\ref{fig:four}c--e constrains the hidden
subspace of the 33-harmonic family to preserve $\gamma_1,\ldots,\gamma_M$ for $M=2$ to 10 and
computes the capacity, the region reachable at unit budget by linear programming of its support
function, the growth half-range $D_M^{(33)}$ and the smallest budget for 400 random targets in a
fixed box; the sweep was repeated with 9 to 33 harmonics (Extended Data Fig.~3; Supplementary
Note 8). The 31 linear designs of Fig.~\ref{fig:four}b were tested by full periodic solves
with the kernel fixed, eight of them for targets drawn at random from a recorded seed.

\paragraph{Reaction--diffusion model.}
The family of Fig.~\ref{fig:three} has $J=\left(\begin{smallmatrix}0.8&-1\\1&-1\end{smallmatrix}\right)$
and $D=\operatorname{diag}(1,3.5)$, so that the uniform state is Turing-unstable at $k_c=0.5071$
with growth rate $\sigma=0.02278$, and $\max_k|\Delta\lambda|\le7\times10^{-13}$ across the
family. Near threshold the three-mode amplitude equations read the nonlinearity through
$(a,g,h)$, fixed by the symmetric bilinear and trilinear forms of its quadratic and cubic parts at
the critical eigenvector (Supplementary Note 10); the Jacobian of $(a,g,h)$ with respect to the
hidden coefficients was evaluated by central differences at 200 random reaction laws of four- and
sixteen-coefficient families and 500 of a 22-coefficient family. Twin reaction laws were drawn at random
in the sixteen-coefficient space and projected onto a target $(a,g,h)$ by minimum-norm
Gauss--Newton steps with every coefficient kept below two in magnitude. The four-coefficient family
$N_\theta=(\eta_2U^2+\eta_{11}UV-\beta_3U^3-\beta_5U^5,\,0)$ is used for the phase plane and the
amplitude designs of Extended Data Fig.~1; there $\Gamma_b=2$, with morphology set through $a$ and
amplitude through $g$, and the quintic coefficient is silent (Supplementary Note 11).

\paragraph{Numerical simulations.}
Equation~\eqref{eq:rd} is integrated pseudo-spectrally with fourth-order exponential time
differencing\cite{hochbruck} on $L_x=8\pi/k_c$, $L_y=16\pi/(\sqrt3k_c)$ with $64\times74$ modes,
two-thirds dealiasing and $\Delta t=0.5$; morphology is read from the number of Fourier pairs near
$k_c$ above one fifth of the strongest and polarity from the skewness of $U$. Every kernel of two
targets was rerun with two further seeds of the initial perturbation, which changed no pattern
from stripe or hexagon seeds; a run counts as the predicted pattern type when its final morphology is
the one predicted stable for the target and its seed. Discretization parameters,
residuals and the complete numerical record behind every quoted figure are given in Supplementary
Note 13.

\paragraph{Data availability.}
The kernels, periodic fields, forward trajectories, result tables and the scripts that generated
the figures are available from the authors upon reasonable request.

\clearpage
\section*{Extended Data}
\setcounter{figure}{0}
\renewcommand{\figurename}{Extended Data Fig.}
\renewcommand{\theHfigure}{ED\arabic{figure}}

\begin{figure}[h]
\centering
\includegraphics[width=.95\linewidth]{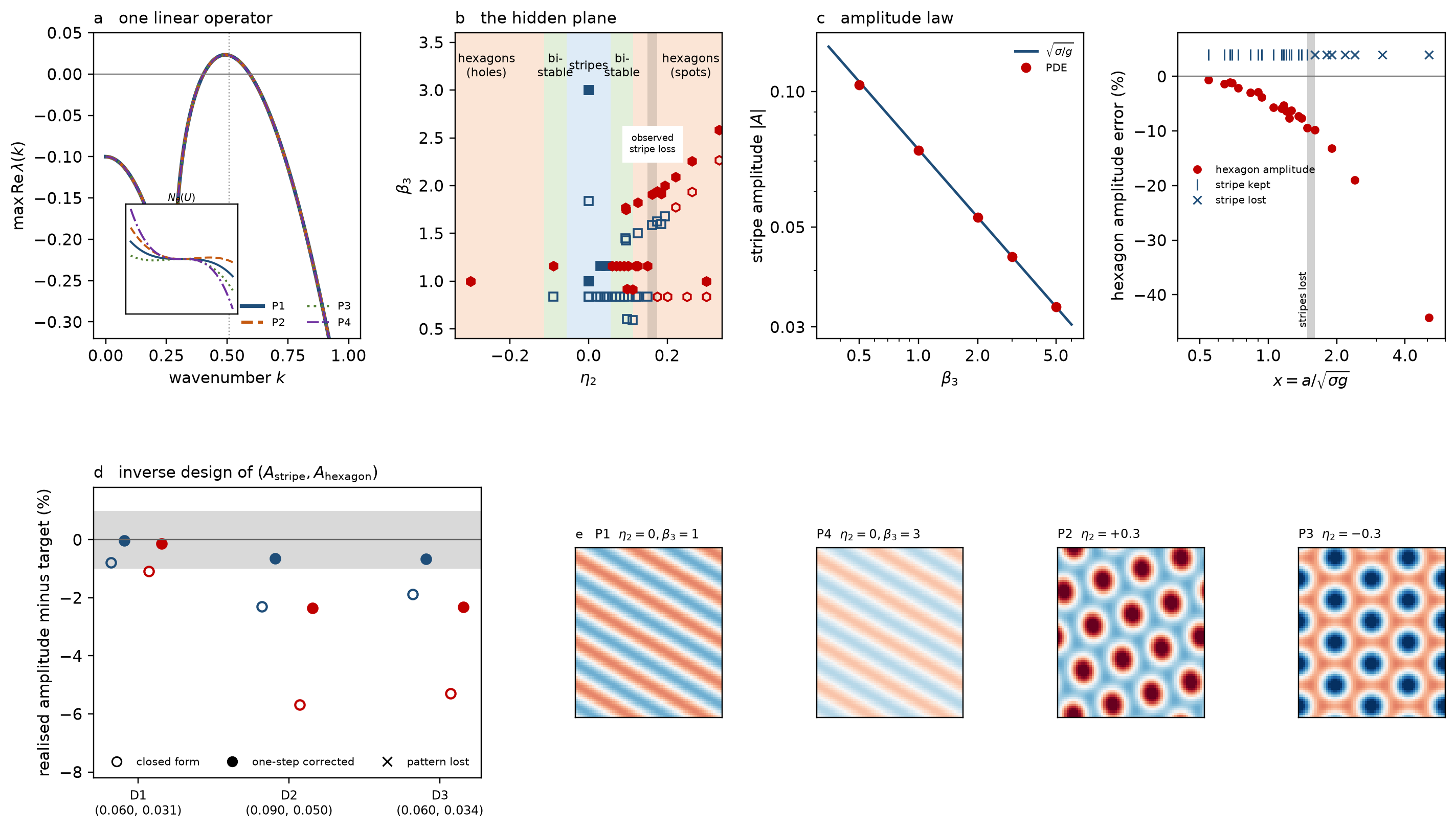}
\caption{\textbf{Hidden plane, amplitude law and inverse design in the
reaction--diffusion family.}
\textbf{a}, Four members P1--P4 share $J$ and $D$, hence the whole dispersion relation; the inset
shows their different nonlinearities.
\textbf{b}, Response-active plane of the quadratic coupling $\eta_2$ and the cubic saturation
$\beta_3$: shaded regions are amplitude-equation predictions, symbols simulated outcomes from
stripe-seeded (open) and hexagon- or noise-seeded (filled) initial conditions, and the grey band the
observed loss of stripes.
\textbf{c}, Stripe amplitude against $\beta_3$, theory and simulation (left); error of the hexagon
amplitude against $x=a/\sqrt{\sigma g}$, with stripes kept (bars) or lost (crosses) (right).
\textbf{d}, Inverse design of the amplitude pair $(A_{\rm stripe},A_{\rm hexagon})$: realised minus
target, from the closed-form solution and after one correction; the grey band is $\pm1\%$.
\textbf{e}, Final $U$ fields for P1--P4.}
\end{figure}
\begin{figure}[h]
\centering
\includegraphics[width=.95\linewidth]{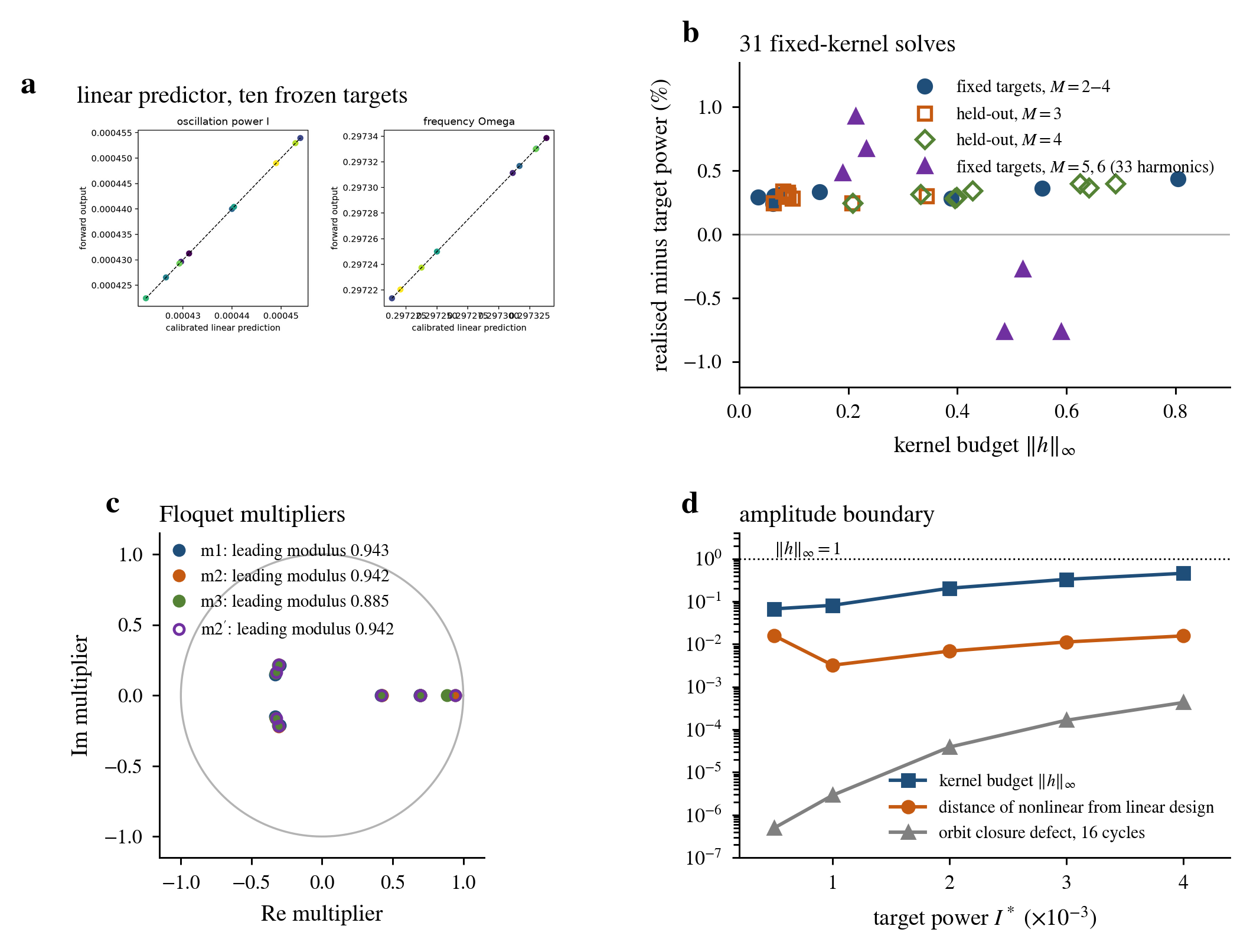}
\caption{\textbf{Validation and limits of the kinetic designs.}
\textbf{a}, Calibrated linear predictor against forward output for ten targets generated from a
recorded seed and frozen before solving, at $N=96$ and $A=256$.
\textbf{b}, Thirty-one full periodic kinetic solves with the kernel fixed: realised minus target
power against kernel budget, for fixed targets in the nine-harmonic family at $M=2$--$4$, held-out
targets at $M=3$ and 4, and fixed targets in the 33-harmonic family at $M=5$ and 6.
\textbf{c}, Eight leading Floquet multipliers of the four designed orbits of Fig.~4 (512 time
steps per period).
\textbf{d}, Kernel budget, distance of the nonlinear from the linear hidden coordinates, and
orbit-closure defect after sixteen cycles of forward integration, as the target power is raised at
fixed frequency; the dotted line is $\|h\|_\infty=1$, below which positivity is guaranteed.}
\end{figure}
\begin{figure}[h]
\centering
\includegraphics[width=\linewidth]{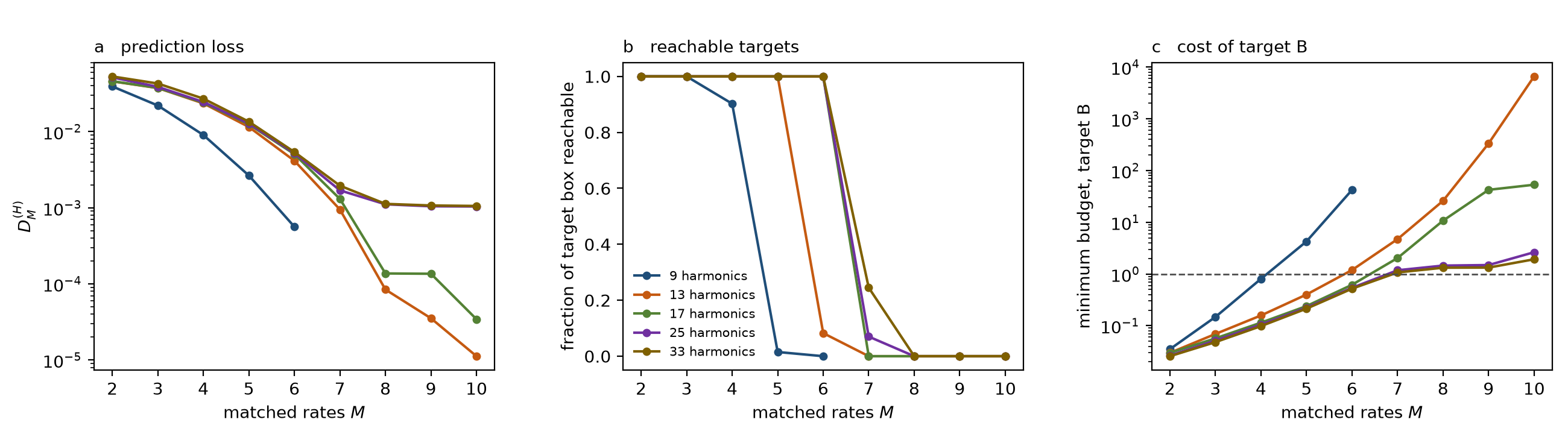}
\caption{\textbf{Truncation test of the constraint sweep.}
Kernel families with $H=9$, 13, 17, 25 and 33 cosine harmonics.
\textbf{a}, Restricted prediction loss $D_M^{(H)}$.
\textbf{b}, Fraction of the target box reachable within $\|h\|_\infty<1$.
\textbf{c}, Minimum budget for one fixed target; the dashed line is $\|h\|_\infty=1$.}
\end{figure}
\begin{figure}[h]
\centering
\includegraphics[width=\linewidth]{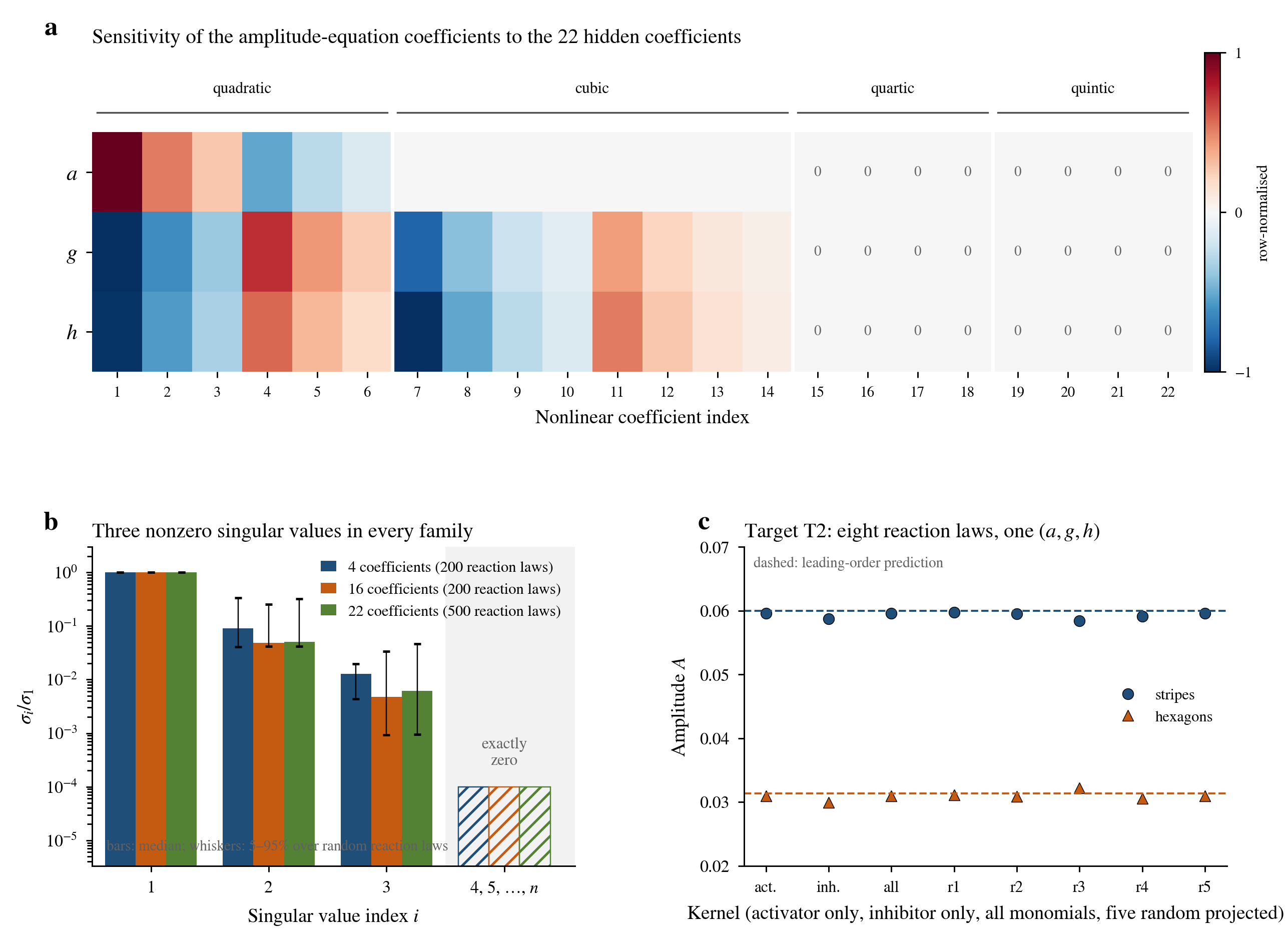}
\caption{\textbf{Evidence behind the three-dimensional active space.}
\textbf{a}, Jacobian of the amplitude-equation coefficients $(a,g,h)$ with respect to 22 hidden
nonlinear coefficients at a representative reaction law (rows normalised); the quartic and quintic
columns vanish identically.
\textbf{b}, Singular values of this Jacobian over random reaction laws of the 4-, 16- and 22-coefficient
families (bars, median; whiskers, 5--95\%): three are nonzero at every reaction law and the rest are
exactly zero.
\textbf{c}, The eight reaction laws of target T2 (activator-only, inhibitor-only, all-monomial and five
random projected nonlinearities) give stripes of amplitude 0.0584--0.0597 and hexagons of
0.0300--0.0322 against leading-order predictions of 0.0600 and 0.0314.}
\end{figure}
\begin{figure}[h]
\centering
\includegraphics[width=.95\linewidth]{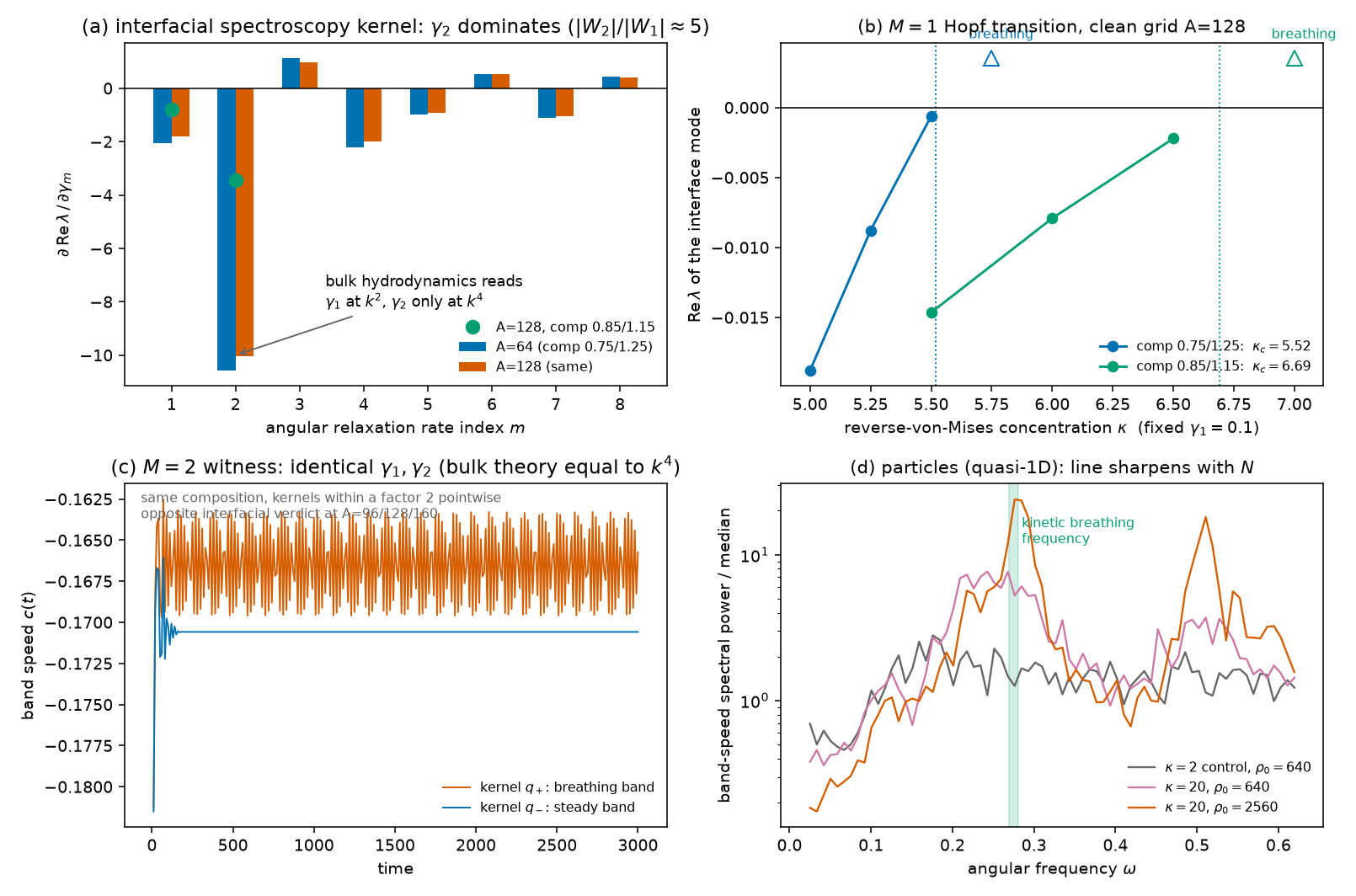}
\caption{\textbf{Interfacial sensitivity, Hopf transition and particle
line.}
\textbf{a}, Sensitivity $W_m=\partial\,\mathrm{Re}\,\lambda/\partial\gamma_m$ of the leading
eigenvalue of the travelling band to each angular rate, at two angular resolutions and two
compositions; the second rate dominates.
\textbf{b}, Growth rate of the interface mode against the concentration of the reverse von Mises
kernel at fixed $\gamma_1$, for two compositions.
\textbf{c}, Band speed against time for a second matched pair whose kernels differ by a factor of
3.5.
\textbf{d}, Band-speed spectral line of quasi-one-dimensional particle simulations at two particle
densities and a control, against the kinetic breathing frequency.}
\end{figure}
\clearpage
\section*{Supplementary Information}
\setcounter{figure}{0}\setcounter{table}{0}\setcounter{equation}{0}
\renewcommand{\figurename}{Supplementary Fig.}\renewcommand{\tablename}{Supplementary Table}
\renewcommand{\thefigure}{\arabic{figure}}\renewcommand{\theHfigure}{S\arabic{figure}}
\renewcommand{\thetable}{\arabic{table}}\renewcommand{\theHtable}{S\arabic{table}}
\renewcommand{\theequation}{S\arabic{equation}}\renewcommand{\theHequation}{S\arabic{equation}}
\renewcommand{\thetheorem}{S\arabic{theorem}}
\renewcommand{\thesubsection}{\thesuppnote.\arabic{subsection}}
\renewcommand{\thesubsubsection}{\thesubsection.\arabic{subsubsection}}
\setcounter{tocdepth}{1}\renewcommand{\contentsname}{Contents}

\noindent This part contains the proofs, derivations and data on which the statements of the main
text rest. Supplementary Notes 1--4 set out the mathematical framework, Notes 5--9 the analysis of the
active mixture, Notes 10 and 11 the reaction--diffusion fibre, and Notes 12 and 13 the data and the
numerical record behind the main text and the Methods. Equations and statements are numbered S1, S2,
\dots\ throughout.

\tableofcontents
\clearpage

\suppnote{Finite-order bulk blindness}

\begin{theorem}[Bulk dispersion blindness of order $2M+2$]\label{thm:bulk}
Assume a uniform angular gap, $\inf_m\gamma_m>0$ (Lemma~\ref{lem:gap} gives it for every kernel used
here). The uniform-state density branch $\lambda(k)=\sum_{j\ge1}l_jk^{2j}$ obeys the matrix continued fraction~\cite{risken} $\lambda+2u_1=0$, $u_n=(vk/2)^2/(\lambda+\gamma_n+u_{n+1})$, and $l_j$ depends only on $\gamma_1,\dots,\gamma_j$. Two generators matching $\gamma_1,\dots,\gamma_M$ therefore have identical bulk density dispersion through $k^{2M}$, with first possible difference at $k^{2M+2}$.
\end{theorem}

\subsection{The scalar ladder}
\label{app:bulk}

\subsubsection{The angular ladder}
Linearise the kinetic equation about the uniform isotropic state $f=\rho_0/2\pi$ and Fourier
transform in space; isotropy makes the branch a function of $|k|$, so we take $k$ along $x$.
Writing $f_m=\int f(\varphi)e^{-im\varphi}\,d\varphi$ and using the even sector $f_{-m}=f_m$, the
free-streaming term $v\cos\varphi\,\partial_x$ couples neighbouring harmonics while the turning
operator is diagonal with eigenvalues $-\gamma_m$ and $\gamma_0=0$ (mass conservation):
\begin{equation}\label{eq:ladder}
 \lambda f_0=-ivk\,f_1,\qquad
 (\lambda+\gamma_m)f_m=-\tfrac{ivk}{2}\,(f_{m-1}+f_{m+1}),\quad m\ge1 .
\end{equation}
Set
\begin{equation}\label{eq:udef}
 s:=\Big(\frac{vk}{2}\Big)^{2},\qquad u_m:=\frac{ivk}{2}\,\frac{f_m}{f_{m-1}}\quad(m\ge1).
\end{equation}
Dividing the $m$-th equation of~\eqref{eq:ladder} by $f_{m-1}$ turns the three-term recursion into
the continued fraction quoted in the main text,
\begin{equation}\label{eq:cf}
 u_m=\frac{s}{\lambda+\gamma_m+u_{m+1}},\qquad \lambda+2u_1=0 .
\end{equation}

\begin{lemma}[Analyticity]\label{lem:analytic}
Assume $\gamma_m\ge\gamma_{\min}>0$ for all $m\ge1$. There are $\rho,\delta>0$ such that for
$|s|<\rho$ and $|\lambda|<\delta$ the system~\eqref{eq:cf} has a unique solution with
$u_m=O(s)$ uniformly in $m$, the truncations obtained by setting $u_{N+1}=0$ converge
geometrically to it, and $u_1$ is analytic in $(\lambda,s)$ near the origin. The equation
$\lambda+2u_1(\lambda,s)=0$ then has a unique analytic solution $\lambda(s)$ with $\lambda(0)=0$.
\end{lemma}

\begin{proof}
On the ball $\mathcal B=\{(u_m)_{m\ge1}:\sup_m|u_m|\le\tfrac12\gamma_{\min}\}$ with the supremum
norm, the map $T[(u_m)]_m=s/(\lambda+\gamma_m+u_{m+1})$ satisfies
$|T[u]_m|\le|s|/(\gamma_{\min}-|\lambda|-\tfrac12\gamma_{\min})$, which is $\le\tfrac12\gamma_{\min}$
once $|\lambda|\le\tfrac14\gamma_{\min}$ and $|s|\le\tfrac18\gamma_{\min}^2$; and
$\|T[u]-T[\tilde u]\|\le 16|s|\gamma_{\min}^{-2}\|u-\tilde u\|$, a contraction for
$|s|<\gamma_{\min}^{2}/16$. Banach's fixed point theorem gives existence, uniqueness and the
geometric rate; the bound $|u_m|\le 4|s|/\gamma_{\min}$ follows by one further application of $T$,
so $u_m=O(s)$ uniformly. Each truncation $u_1^{(N)}$ is a rational function of $(\lambda,s)$ with
non-vanishing denominator on the polydisc, hence analytic, and the convergence is uniform there,
so $u_1$ is analytic. Finally $\partial_\lambda[\lambda+2u_1]=1+O(s)\ne0$ near the origin and
$\lambda+2u_1=0$ holds at $(\lambda,s)=(0,0)$, so the implicit function theorem applies.
\end{proof}

\subsubsection{The ladder lemma}
Write the branch and the ladder as convergent power series in $s$,
\begin{equation}
 \lambda(s)=\sum_{j\ge1}L_j\,s^{j},\qquad u_n(s)=\sum_{p\ge1}c_{n,p}\,s^{p}\quad(n\ge1),
\end{equation}
so that $L_j=-2c_{1,j}$ by~\eqref{eq:cf}. In terms of the wavenumber,
$\lambda(k)=\sum_j l_jk^{2j}$ with $l_j=(v^2/4)^{j}L_j$.

\begin{lemma}[Ladder lemma]\label{lem:ladder}
For all $n\ge1$ and $p\ge1$:
\begin{enumerate}
\item[(i)] $c_{n,p}$ is a polynomial in $\gamma_n^{-1},\dots,\gamma_{n+p-1}^{-1}$ and
 $L_1,\dots,L_{p-1}$;
\item[(ii)] $c_{n,p}$ has degree exactly one in $\gamma_{n+p-1}^{-1}$, with
\begin{equation}\label{eq:slope}
 \frac{\partial c_{n,p}}{\partial(\gamma_{n+p-1}^{-1})}=(-1)^{p-1}\prod_{i=n}^{n+p-2}\gamma_i^{-2},
\end{equation}
the empty product being $1$ when $p=1$;
\item[(iii)] $L_j$ depends only on $\gamma_1,\dots,\gamma_j$.
\end{enumerate}
\end{lemma}

\begin{proof}
We prove (i) and (ii) by induction on $p$, simultaneously for all $n$, and deduce (iii).

\emph{Base $p=1$.} Expanding~\eqref{eq:cf} with $\lambda=O(s)$ and $u_{n+1}=O(s)$ gives
$u_n=s\gamma_n^{-1}+O(s^2)$, so $c_{n,1}=\gamma_n^{-1}$: a polynomial of degree one in
$\gamma_n^{-1}$ with $\partial c_{n,1}/\partial(\gamma_n^{-1})=1$, matching~\eqref{eq:slope}.

\emph{Induction.} Let $p\ge2$ and assume (i), (ii) at all orders $<p$ and all $n$. Expand
\begin{equation}\label{eq:expand}
 u_n=\frac{s}{\gamma_n}\Big(1+\frac{\lambda+u_{n+1}}{\gamma_n}\Big)^{-1}
 =\frac{s}{\gamma_n}\sum_{i\ge0}(-1)^i\Big(\frac{\lambda+u_{n+1}}{\gamma_n}\Big)^{i},
\end{equation}
legitimate as a formal, and by Lemma~\ref{lem:analytic} convergent, series because
$\lambda+u_{n+1}=O(s)$. Extracting the coefficient of $s^p$,
\begin{equation}\label{eq:cnp}
 c_{n,p}=-\frac{1}{\gamma_n^{2}}\big(L_{p-1}+c_{n+1,p-1}\big)+R_{n,p},
\end{equation}
where $R_{n,p}$ collects the terms with $i\ge2$ in~\eqref{eq:expand}. Each such term is a product
of $i\ge2$ coefficients of $\lambda+u_{n+1}$, of orders $\ge1$ summing to $p-1$, so every factor
has order at most $p-2$; by the inductive hypothesis those factors involve at most
$\gamma_{n+1},\dots,\gamma_{n+p-2}$ and $L_1,\dots,L_{p-2}$, whence $R_{n,p}$ is free of
$\gamma_{n+p-1}$. The term $L_{p-1}$ in~\eqref{eq:cnp} involves only $\gamma_1,\dots,\gamma_{p-1}$
by (iii) at lower order, and $p-1<n+p-1$ because $n\ge1$, so it too is free of $\gamma_{n+p-1}$.
Therefore $\gamma_{n+p-1}$ enters $c_{n,p}$ only through $c_{n+1,p-1}$, in which by the inductive
hypothesis it appears to degree exactly one. This proves (i) and the degree claim, and
differentiating~\eqref{eq:cnp},
\[
 \frac{\partial c_{n,p}}{\partial(\gamma_{n+p-1}^{-1})}
 =-\frac{1}{\gamma_n^{2}}\,\frac{\partial c_{n+1,p-1}}{\partial(\gamma_{n+p-1}^{-1})}
 =-\frac{1}{\gamma_n^{2}}(-1)^{p-2}\prod_{i=n+1}^{n+p-2}\gamma_i^{-2}
 =(-1)^{p-1}\prod_{i=n}^{n+p-2}\gamma_i^{-2},
\]
which is~\eqref{eq:slope}. Finally (iii) follows by induction on $j$: $L_j=-2c_{1,j}$, and by (i)
$c_{1,j}$ is a polynomial in $\gamma_1^{-1},\dots,\gamma_j^{-1}$ and $L_1,\dots,L_{j-1}$, each of
which involves only $\gamma_1,\dots,\gamma_{j-1}$ by the inductive hypothesis.
\end{proof}

\subsubsection{Proof of Theorem~\ref{thm:bulk}, and its sharp form}
Part (iii) of Lemma~\ref{lem:ladder} is exactly Theorem~\ref{thm:bulk}: two generators with
$\gamma_m=\gamma'_m$ for $m\le M$ have $l_j=l'_j$ for $j\le M$, so their bulk density dispersions
agree through $k^{2M}$, and $\gamma_{M+1}$ first enters at $l_{M+1}$. Part (ii) upgrades
``first possible difference'' to an identity.

\begin{theorem}[Sharp form]\label{thm:sharp}
Under the hypotheses of Theorem~\ref{thm:bulk}, $L_{M+1}$ is an affine function of
$\gamma_{M+1}^{-1}$:
\begin{equation}
 L_{M+1}=A_M(\gamma_1,\dots,\gamma_M)+\frac{2(-1)^{M+1}}{\gamma_{M+1}\prod_{n=1}^{M}\gamma_n^{2}},
 \qquad
 \frac{\partial l_{M+1}}{\partial\gamma_{M+1}}
 =\frac{(-1)^{M}\,v^{2M+2}}{2^{2M+1}\,\gamma_{M+1}^{2}\prod_{n=1}^{M}\gamma_n^{2}}\ne0 .
\end{equation}
Consequently, for two generators matching $\gamma_1,\dots,\gamma_M$, the bulk density dispersions
differ at order $k^{2M+2}$ \emph{if and only if} $\gamma_{M+1}\ne\gamma'_{M+1}$, and the leading
difference is exact rather than linearised:
\begin{equation}\label{eq:exactdiff}
 \lambda(k)-\lambda'(k)
 =\frac{2(-1)^{M+1}v^{2M+2}}{4^{M+1}\prod_{n=1}^{M}\gamma_n^{2}}
 \Big(\frac{1}{\gamma_{M+1}}-\frac{1}{\gamma'_{M+1}}\Big)k^{2M+2}+O(k^{2M+4}).
\end{equation}
\end{theorem}

\begin{proof}
Apply Lemma~\ref{lem:ladder} with $n=1$ and $p=M+1$. It shows that the coefficient
$c_{1,M+1}$ is an affine function of $\gamma_{M+1}^{-1}$ with
slope $(-1)^{M}\prod_{n\le M}\gamma_n^{-2}$, so $L_{M+1}=-2c_{1,M+1}$ is affine, with
slope $2(-1)^{M+1}\prod_{n\le M}\gamma_n^{-2}$. The derivative in $\gamma_{M+1}$ follows from
$\partial_\gamma(\gamma^{-1})=-\gamma^{-2}$ and $l_{M+1}=(v^2/4)^{M+1}L_{M+1}$, using
$2(v^2/4)^{M+1}=v^{2M+2}/2^{2M+1}$. Since $\gamma_1,\dots,\gamma_M$ agree by assumption, so does
$A_M$, and~\eqref{eq:exactdiff} is the difference of the two affine expressions.
\end{proof}

Theorem~\ref{thm:sharp} is a statement about the scalar ladder~\eqref{eq:cf}, that is, about one species moving at constant speed; Section~\ref{app:twospecies} shows which part of it survives when the speed depends on the sensed densities and the species are coupled. Within that scope it explains why the log--log exponents measured in 80-digit arithmetic sit
stably on $2M+2$ rather than above it: the coefficient at that order cannot accidentally vanish
for any admissible rate table. The formula was also checked in exact rational arithmetic. Solving~\eqref{eq:cf} order by order on the non-degenerate table $\gamma_n=(1+n)/(3+2n)$, retaining orders through $s^{8}$ and truncating the
ladder at $n=12$, the coefficient $L_{M+1}$ is affine in $\gamma_{M+1}^{-1}$ to exact rational
equality at three distinct values of $\gamma_{M+1}$, with slopes $25/2$, $-1225/18$, $11025/32$ and
$-53361/32$ for $M=1,2,3,4$, each equal to the predicted
$2(-1)^{M+1}\prod_{n\le M}\gamma_n^{-2}$ exactly. The lowest coefficients are
$l_1=-v^2/(2\gamma_1)$, the standard active diffusivity, and
$L_2=2\gamma_1^{-2}(\gamma_2^{-1}-2\gamma_1^{-1})$.

\subsection{Two species with density-dependent speeds}
\label{app:twospecies}

The model of the main text has two species whose speeds depend on the sensed densities, so its
uniform-state hierarchy is a matrix ladder in place of the scalar one~\eqref{eq:cf}.

Let the two species share the angular rates $\gamma_m$, let the speed be independent of the
orientation, and let the sensing kernel be translation invariant and inversion symmetric with
Fourier symbol $\hat K_u(k)$ analytic near $k=0$. Write
\begin{equation}
 V=\diag(\bar v_1,\bar v_2),\qquad
 H_{su}(k)=\bar v_s\delta_{su}
 +\bar\rho_s\left.\frac{\partial v_s}{\partial\tilde\rho_u}\right|_{\bar\rho}\hat K_u(k),
 \qquad z=\lambda+D_tk^2 .
\end{equation}
Nonreciprocity is carried by $H$, which is not symmetric. Linearising about the uniform state and
keeping the even angular sector, the density $\rho$ and the harmonics $a_m$ obey
\begin{equation}
 z\rho=-ikVa_1,\quad
 (z+\gamma_1)a_1=-\tfrac{ik}{2}\big[H(k)\rho+Va_2\big],\quad
 (z+\gamma_m)a_m=-\tfrac{ik}{2}V(a_{m-1}+a_{m+1}),\ m\ge2 .
\end{equation}
Eliminating from the top gives a \emph{matrix} continued fraction and the density eigenvalue
condition
\begin{equation}\label{eq:matcf}
 Q_m=(z+\gamma_m)I+\frac{k^2}{4}VQ_{m+1}^{-1}V,\qquad
 \det\!\Big[zI+\frac{k^2}{2}VQ_1^{-1}H(k)\Big]=0 .
\end{equation}

\begin{proposition}[Blindness with species coupling]\label{prop:matrix}
Suppose the two eigenvalues of $A_0=VH_0/(2\gamma_1)$ are simple, with right and left vectors
$r_j,\ell_j$ normalised by $\ell_j^{\dagger}r_j=1$. Let two generators share $V$, $H(k)$, $D_t$ and
the rates $\gamma_1,\dots,\gamma_M$. Then the corresponding slow density branches satisfy
$\lambda^+_j(k)-\lambda^-_j(k)=O(k^{2M+2})$, with
\begin{equation}\label{eq:matcoef}
 \Delta\lambda_j(k)=
 \frac{(-1)^{M+1}\,\ell_j^{\dagger}V^{2M+1}H_0r_j}{2^{2M+1}\prod_{n=1}^{M}\gamma_n^{2}}
 \Big(\frac{1}{\gamma^+_{M+1}}-\frac{1}{\gamma^-_{M+1}}\Big)k^{2M+2}+O(k^{2M+4}).
\end{equation}
\end{proposition}

\begin{proof}[Proof sketch]
The mismatch enters only through $Q_{M+1}$. Propagating it down~\eqref{eq:matcf} costs one factor
$k^2/4$ per level and one inverse damping factor on each side; because the $\gamma_n$ are scalars
and $V$ is diagonal, those factors commute and collect into
$\Delta Q_1^{-1}=(-1)^M(k^2/4)^MV^{2M}\Delta(\gamma_{M+1}^{-1})\prod_{n\le M}\gamma_n^{-2}
+O(k^{2M+2})$. Multiplying by the $k^2V/2$ and $H_0$ of~\eqref{eq:matcf} and projecting onto the
simple branch with $\ell_j,r_j$ gives~\eqref{eq:matcoef}; analyticity in $k^2$ follows from the
reduction $z=k^2\zeta$ and the implicit function theorem at a simple root.
\end{proof}

We also evaluated~\eqref{eq:matcoef} at the parameters of the main text, using the speed law,
interaction matrix and sensing symbol of the Methods. Here
$\bar v=(1.60614,0.64123)$, so the asymptotic window is set by
$kv_{\max}/\gamma_{\min}\le0.04$; orders $M=3,4$ fall below double precision inside that window
and were recomputed in 60-digit arithmetic. Supplementary Table~\ref{tab:e1} reports the ratio of the measured
first-difference coefficient to~\eqref{eq:matcoef} at the finest wavenumber, together with the
bilinear factor, which is nonzero on both branches at this parameter point.

\begin{table}[htbp]\centering
\caption{\textbf{First-difference coefficient of the actual two-species model against
Eq.~\eqref{eq:matcoef}: ratio measured/predicted at $kv_{\max}/\gamma_{\min}=0.01$, and the
bilinear factor $\ell_j^{\dagger}V^{2M+1}H_0r_j$ on each slow branch.}}\label{tab:e1}
\footnotesize\setlength{\tabcolsep}{5pt}
\begin{tabular}{ccrrl}
\toprule
$M$ & branch & bilinear factor & ratio & precision\\
\midrule
1 & 1 & $-0.123927$ & $1.000004$ & double\\
1 & 2 & $2.129373$ & $0.999991$ & double\\
2 & 1 & $0.252537$ & $0.999483$ & double\\
2 & 2 & $5.556205$ & $0.999773$ & double\\
3 & 1 & $0.886758$ & $0.999960$ & $60$ digits\\
3 & 2 & $14.35926$ & $0.999975$ & $60$ digits\\
4 & 1 & $2.384314$ & $0.999953$ & $60$ digits\\
4 & 2 & $37.05321$ & $0.999971$ & $60$ digits\\
\bottomrule
\end{tabular}
\end{table}

The blindness statement of Theorem~\ref{thm:bulk} therefore carries over to the interacting
two-species model unchanged: matching $\gamma_1,\dots,\gamma_M$ still makes the bulk density
dispersion agree through $k^{2M}$. The sharpening of Theorem~\ref{thm:sharp} does \emph{not} carry
over. Its scalar slope is replaced by the bilinear factor $\ell_j^{\dagger}V^{2M+1}H_0r_j$, which
can vanish for particular $V,H_0$; the first difference then appears at higher order still, so the
``if and only if'' of Theorem~\ref{thm:sharp} is specific to one species at constant speed. Such cancellation can only make the bulk \emph{more} blind.

We verified~\eqref{eq:matcoef} numerically on a synthetic non-reciprocal system,
$V=\diag(1,1.25)$ and $H_0=\left(\begin{smallmatrix}1.6&0.3\\-0.1&0.7\end{smallmatrix}\right)$,
with the matrix ladder truncated at $60$ levels and each branch obtained by self-consistent
iteration. Supplementary Table~\ref{tab:matrix} gives the measured $\Delta\lambda_j/k^{2M+2}$ against the
prediction.

\begin{table}[htbp]\centering
\caption{\textbf{Measured first-difference coefficient of the two-species matrix ladder against
Eq.~\eqref{eq:matcoef}, for kernel pairs matched through $\gamma_M$.}}\label{tab:matrix}
\begin{tabular}{ccrrc}
\toprule
$M$ & branch & predicted & measured at $k=0.02$ & ratio\\
\midrule
1 & 1 & $0.1238475$ & $0.1238792$ & $1.00026$\\
1 & 2 & $0.1207010$ & $0.1208476$ & $1.00121$\\
2 & 1 & $-0.0493649$ & $-0.0493804$ & $1.00031$\\
2 & 2 & $-0.0276178$ & $-0.0276649$ & $1.00171$\\
3 & 1 & $0.0195299$ & $0.0195358$ & $1.00030$\\
3 & 2 & $0.0059055$ & $0.0059356$ & $1.00510$\\
\bottomrule
\end{tabular}
\end{table}

\suppnote{The nonlinear long-wave closure}

\begin{theorem}[Nonlinear closure blindness, with explicit first difference]\label{thm:nlclosure}
Set $X=\delta x$, $T=\delta^2t$, fix the density coordinate $f_{s,0}=\rho_s$ and expand the full
angular hierarchy, retaining cross-species speeds, the even gradient corrections of the sensing
symbol $\hat K$, translational diffusion, and the slow-time Fr\'echet feedback. Then
$f_{s,m}=\sum_{n\ge m,\;n-m\ \mathrm{even}}\delta^{n}A_{s,m,n}[\rho]$ with $A_{s,m,n}$ depending on
$\gamma_1,\dots,\gamma_{(n+m)/2}$ only, so that in
$\partial_T\rho_s=\sum_{j\ge1}\delta^{2j-2}\mathcal F_{s,2j}[\rho]$ the coefficient
$\mathcal F_{s,2j}$ depends on $\gamma_1,\dots,\gamma_j$ alone. Two generators matching
$\gamma_1,\dots,\gamma_M$ therefore have \emph{identical nonlinear closures} through order $2M$, and
the first difference is explicit:
\begin{equation}\label{eq:firstdiff}
 \mathcal F^{+}_{s,2M+2}-\mathcal F^{-}_{s,2M+2}
 =\frac{\gamma_{M+1,+}^{-1}-\gamma_{M+1,-}^{-1}}{2^{2M+1}\prod_{j=1}^{M}\gamma_j^{2}}\,
  \mathcal D_s^{\,2M+2}\rho_s,
 \qquad \mathcal D_sg:=\partial_X\big[v^{(0)}_s(\rho)\,g\big].
\end{equation}
\end{theorem}

\subsection{Proof of Theorem~\ref{thm:nlclosure}}
\label{app:nlclosure}

Sections~\ref{app:bulk} and~\ref{app:twospecies} concern the linearised dispersion of the uniform
state. This Note proves the corresponding statement for the full nonlinear long-wave closure,
on which the conclusion of the main text rests. Throughout, $X=\delta x$, $T=\delta^2t$, the
density coordinate is fixed as $f_{s,0}=\rho_s$, and all forbidden or negative indices are zero.

\subsubsection{An all-order formal recursion}

Write $v_s=\sum_{a\ge0}\delta^{2a}v_{s,2a}[\rho]$, whose coefficients are determined by the speed
law and the even gradient corrections of the sensing symbol and carry no dependence on the turning
kernel. In particular $v^{(0)}_s=v_s(\rho)$ and, from $\hat K(k)=1-\tfrac{3}{40}k^2+\tfrac1{448}k^4+O(k^6)$,
\begin{equation}\label{eq:v2}
 v^{(2)}_s=\frac{3}{40}\sum_u\partial_{\rho_u}v^{(0)}_s\,\partial_X^2\rho_u.
\end{equation}

\begin{theorem}[All-order formal nonlinear matching]\label{thm:allorder}
There is a unique recursive formal Chapman--Enskog expansion
\begin{equation}
 f_{s,m}=\sum_{\substack{n\ge m\\ n-m\ \mathrm{even}}}\delta^nA_{s,m,n}[\rho],
 \qquad
 \partial_T\rho_s=\sum_{j\ge1}\delta^{2j-2}\mathcal F_{s,2j}[\rho],
\end{equation}
in which $A_{s,m,n}$ depends on the turning generator only through
$\gamma_1,\dots,\gamma_{(n+m)/2}$. Consequently $\mathcal F_{s,2j}$ depends only on
$\gamma_1,\dots,\gamma_j$, and matching the first $M$ rates makes every nonlinear density
coefficient through gradient order $2M$ identical.
\end{theorem}

\begin{proof}
Substituting the expansions into the invariance equation for the fast moments gives
\begin{equation}\label{eq:recurrence}
 \gamma_mA_{s,m,n}=-\frac12\partial_X\!\!\sum_{2a+b=n-1}\!\!v_{s,2a}\big(A_{s,m-1,b}+A_{s,m+1,b}\big)
 +D_t\partial_X^2A_{s,m,n-2}-\sum_{j\ge1}D_\rho A_{s,m,n-2j}[\mathcal F_{2j}],
\end{equation}
the last sum being the slow-time Fr\'echet feedback, while the density equation gives
\begin{equation}\label{eq:densityrec}
 \mathcal F_{s,2j}=\mathbf 1_{j=1}D_t\partial_X^2\rho_s
 -\partial_X\!\!\sum_{2a+b=2j-1}\!\!v_{s,2a}A_{s,1,b}.
\end{equation}
These are triangular in $n$: in the feedback sum $n-2j\ge m$ forces $2j-1\le n-m-1<n$, so
$\mathcal F_{2j}$ is already known when it is needed, and division by the positive scalar
$\gamma_m$ determines each coefficient uniquely.

For the rate bound, argue by induction on $n$. In~\eqref{eq:recurrence} the $m{+}1$ term uses rates
through $\gamma_{(n+m-2a)/2}$, the $m{-}1$ term through $\gamma_{(n+m-2-2a)/2}$, and the diffusion
term through $\gamma_{(n+m-2)/2}$; a feedback term uses rates through
$\max\{(n-2j+m)/2,\,j\}$, which is at most $(n+m)/2$ wherever it occurs. Differentiating with
respect to density jets introduces no new rates, and multiplying by $\gamma_m^{-1}$ stays inside the
bound because $n\ge m$. Inserting this into~\eqref{eq:densityrec} gives $\gamma_j$ at physical order
$2j$. Parity and the lowest diagonal $n=m$ follow from the same recursion. The conclusion is an
identity in the formal algebra of differential expressions, not an estimate between solutions of the
two partial differential equations.
\end{proof}

\subsubsection{The first mismatch in closed form}

For each species and at a fixed density field put
\begin{equation}\label{eq:Dv}
 \mathcal D_sg:=\partial_X\big\{v^{(0)}_s(\rho)\,g\big\}.
\end{equation}
Powers of $\mathcal D_s$ mean repeated application, so they differentiate the density-dependent
speed as well; these are nonlinear differential functions of \emph{all} density components, not
powers of a constant-speed operator.

\begin{theorem}[Explicit first nonlinear difference]\label{thm:firstdiffapp}
Under Theorem~\ref{thm:allorder}, two generators matching $\gamma_1,\dots,\gamma_M$ satisfy
Eq.~\eqref{eq:firstdiff}. No rate beyond $\gamma_{M+1}$ occurs in that coefficient.
\end{theorem}

\begin{proof}
On the lowest diagonal the recursion gives
$A_{s,m,m}=(-1)^m\big(2^m\prod_{j\le m}\gamma_j\big)^{-1}\mathcal D_s^m\rho_s$, so the first changed
coefficient is $A_{s,M+1,M+1}$. For $m\le M$ the first index that can change is $n=2M+2-m$; there
the $m{-}1$ term, every positive-order sensing correction, the diffusion term and every feedback
term in~\eqref{eq:recurrence} involve matched rates only, so the difference can enter through the
$m{+}1$ term alone, giving
$\Delta A_{s,m,2M+2-m}=-(2\gamma_m)^{-1}\mathcal D_s\,\Delta A_{s,m+1,2M+1-m}$. Iterating from
$m=M+1$ down to $m=1$,
\begin{equation}
 \Delta A_{s,1,2M+1}
 =-\frac{\Delta(\gamma_{M+1}^{-1})}{2^{2M+1}\prod_{j=1}^{M}\gamma_j^{2}}\,\mathcal D_s^{2M+1}\rho_s,
\end{equation}
and in~\eqref{eq:densityrec} only the zeroth-order speed multiplies this changed coefficient, so
applying $-\mathcal D_s$ gives~\eqref{eq:firstdiff}. Every interspecies dependence is retained
inside $v^{(0)}_s(\rho)$.
\end{proof}

\noindent The formula follows the shortest path by which a mismatched rate can reach the density
equation, up the angular ladder to $m=M+1$ and back down, so no higher rate appears. Linearising~\eqref{eq:firstdiff} about the uniform state returns the
first-difference coefficient of Proposition~\ref{prop:matrix}; specialising to one species at
constant speed returns Theorem~\ref{thm:sharp}.

\subsubsection{Two independent numerical verifications}

The checks below are computed independently of the symbolic derivation. Each field is carried as a truncated
power series in $\delta$ on a fixed $X$-grid, so that $\partial_x=\delta\partial_X$ puts the small
parameter explicitly in the coefficients; the $\tanh$ speed law, the matrix $\eta$, the sensing
corrections~\eqref{eq:v2} and the slow-time feedback of~\eqref{eq:recurrence} are all retained, the
last by complex-step differentiation iterated to self-consistency. The coefficients are read
directly, with no time integration and no fixed-point iteration.

\begin{table}[htbp]\centering
\caption{\textbf{Relative difference of the two closures, coefficient by coefficient in $\delta$, for
kernel pairs matching $\gamma_1,\dots,\gamma_M$ exactly.} Entries shown as $0$ are exactly zero:
$\gamma_{M+1}$ does not enter the arithmetic that produces those orders.}\label{tab:nlblind}
\footnotesize
\begin{tabular}{lrrrrrcc}
\toprule
$M$ & $\delta^2$ & $\delta^4$ & $\delta^6$ & $\delta^8$ & $\delta^{10}$ & first seen & predicted\\
\midrule
$1$ & $0$ & $6.3\times10^{-1}$ & $7.6\times10^{-1}$ & $9.1\times10^{-1}$ & $9.6\times10^{-1}$ & $\delta^{4}$ & $\delta^{4}$\\
$2$ & $0$ & $0$ & $3.1\times10^{-1}$ & $4.4\times10^{-1}$ & $7.8\times10^{-1}$ & $\delta^{6}$ & $\delta^{6}$\\
$3$ & $0$ & $0$ & $0$ & $6.0\times10^{-2}$ & $2.1\times10^{-1}$ & $\delta^{8}$ & $\delta^{8}$\\
$4$ & $0$ & $0$ & $0$ & $0$ & $1.3\times10^{-2}$ & $\delta^{10}$ & $\delta^{10}$\\
\bottomrule
\end{tabular}
\end{table}

\noindent The $M=2$ row was repeated over $24$ further settings --- kernel contrast
$\epsilon\in\{0.3,0.9\}$, grid $N_X\in\{96,128,192\}$, profile amplitude $\{0.05,0.10\}$ and base
composition $\{(0.75,1.25),(1,1)\}$ --- with the same verdict in every case and identical coefficients, to rounding error, across $N_X$. Retaining the slow-time feedback carries the fourth-order coefficient from $49.62$ to
$57.17$ at the reference setting and leaves the difference between the two kernels unchanged, as
the shortest-path argument requires.

\begin{table}[htbp]\centering
\caption{\textbf{Direct test of the closed form~\eqref{eq:firstdiff}.} Listed is the relative residual
$\|\Delta\mathcal F_{2M+2}-\text{Eq.}~\eqref{eq:firstdiff}\|/\|\Delta\mathcal F_{2M+2}\|$ against
grid size, together with the projection ratio
$\langle\Delta\mathcal F,\text{pred}\rangle/\langle\text{pred},\text{pred}\rangle$ at the coarsest
grid. Evaluating $\mathcal D_s^{2M+2}$ amplifies grid-scale round-off by $k_{\max}^{2M+2}$, which
sets the residual level at each grid.}\label{tab:firstdiff}
\footnotesize
\begin{tabular}{lrrrrrrc}
\toprule
$M$ & $N_X{=}24$ & $32$ & $48$ & $64$ & $96$ & $128$ & ratio at $N_X{=}24$\\
\midrule
$1$ & $4.5\times10^{-14}$ & $2.3\times10^{-13}$ & $8.6\times10^{-13}$ & $1.1\times10^{-12}$ & $6.9\times10^{-12}$ & $8.2\times10^{-12}$ & $1.000000000$\\
$2$ & $1.3\times10^{-12}$ & $1.1\times10^{-11}$ & $1.1\times10^{-10}$ & $2.3\times10^{-10}$ & $3.6\times10^{-9}$ & $6.3\times10^{-9}$ & $1.000000000$\\
$3$ & $1.9\times10^{-11}$ & $3.2\times10^{-10}$ & $7.2\times10^{-9}$ & $2.6\times10^{-8}$ & $1.0\times10^{-6}$ & $2.8\times10^{-6}$ & $1.000000000$\\
$4$ & $1.7\times10^{-10}$ & $5.1\times10^{-9}$ & $2.7\times10^{-7}$ & $1.7\times10^{-6}$ & $1.6\times10^{-4}$ & $6.4\times10^{-4}$ & $1.000000000$\\
\bottomrule
\end{tabular}
\end{table}

\noindent From $N_X=128$ to $N_X=24$ the residual falls by factors $1.8\times10^2$,
$4.8\times10^3$, $1.4\times10^5$ and $3.8\times10^6$ for $M=1,\dots,4$, to be compared with the
round-off amplification ratios $(64/12)^{2M+2}$, namely $8.1\times10^2$, $2.3\times10^4$,
$6.5\times10^5$ and $1.9\times10^7$. The residual tracks the round-off amplification of the repeated
spectral differentiation.

Finally, the coefficient identities say nothing on their own about accuracy at a finite $\delta$.
Freezing a smooth long-wave profile and solving the slaved ladder \emph{exactly} as a dense linear
system --- no iteration, no time integration --- gives the kinetic tendency against which the
truncated closure can be measured directly.

\begin{table}[htbp]\centering
\caption{\textbf{Accuracy of the truncated closure at finite $\delta$, on the frozen profile, against the
exactly solved ladder.} The small parameter is $v\delta K/2\gamma\simeq8\delta$.}\label{tab:trunc}
\footnotesize
\begin{tabular}{rrrrr}
\toprule
$\delta$ & $8\delta$ & $\|T-\delta^2\mathcal F_2\|/\|T\|$ & $\|T-\delta^2\mathcal F_2-\delta^4\mathcal F_4\|/\|T\|$ & $\|T^{+}-T^{-}\|/\|T\|$\\
\midrule
$1.25\times10^{-2}$ & $0.100$ & $1.317\times10^{-2}$ & $1.313\times10^{-3}$ & $4.486\times10^{-4}$\\
$6.25\times10^{-3}$ & $0.050$ & $3.466\times10^{-3}$ & $9.449\times10^{-5}$ & $3.607\times10^{-5}$\\
$3.13\times10^{-3}$ & $0.025$ & $8.793\times10^{-4}$ & $6.167\times10^{-6}$ & $2.435\times10^{-6}$\\
$1.56\times10^{-3}$ & $0.013$ & $2.207\times10^{-4}$ & $3.899\times10^{-7}$ & $1.554\times10^{-7}$\\
$7.81\times10^{-4}$ & $0.006$ & $5.522\times10^{-5}$ & $2.444\times10^{-8}$ & $9.763\times10^{-9}$\\
$3.91\times10^{-4}$ & $0.003$ & $1.381\times10^{-5}$ & $1.558\times10^{-9}$ & $6.110\times10^{-10}$\\
\midrule
\multicolumn{2}{l}{finest fitted slope} & $1.9996$ & $3.9713$ & $3.9981$\\
\multicolumn{2}{l}{predicted} & $2$ & $4$ & $4$\\
\bottomrule
\end{tabular}
\end{table}

\noindent The first two columns measure the closure against the kinetic tendency itself, the third
measures the two kinetic models against each other, and all three follow the predicted order. These
are statements about the frozen slaved manifold; the self-organised interface of the following
sections is not an expansion in $\delta$.

\suppnote{Interfacial sensitivity}

\noindent The sensitivity of the band's leading eigenvalue to the angular rate $\gamma_m$ is

\begin{equation}\label{eq:kato}
 W_m=\frac{d\Real\lambda}{d\gamma_m}
 =\Real\frac{\langle l,\partial_{\gamma_m}L\,r\rangle+\langle l,(D_FL[F_m]+c_m\partial_cL)r\rangle}{\langle l,r\rangle},
\end{equation}

\subsection{Derivation of Eq.~\eqref{eq:kato}}
\label{app:kato}

\subsubsection{The band and its tangents}
In the comoving frame $\xi=x-ct$ a travelling band solves
\begin{equation}\label{eq:band}
 G(F,c;\gamma):=c\,\partial_\xi F+\mathcal N[F;\gamma]=0
\end{equation}
on one Bloch cell, together with the species mass constraints $\Phi_s(F)=\int F_s=\mu_s$ and a
phase condition $\vartheta(F-F_0)=0$ removing the translation zero mode. The turning operator is
diagonal in angular harmonics, so
\begin{equation}\label{eq:dG}
 \partial_{\gamma_m}G=-P_mF,
\end{equation}
where $P_m$ projects onto $\{\cos m\varphi,\sin m\varphi\}$.

Three scalar constraints are imposed but only the wave speed is added as an unknown, so the
naively bordered Jacobian is not square. Restore squareness by adjoining two constant mass-source
columns $e_s$ normalised by $\Phi_i(e_s)=\delta_{is}$ and the corresponding multipliers, i.e.\ solve
\begin{equation}\label{eq:waveaug}
 \mathcal B(Y;\gamma)=\big(\,G(F,c;\gamma)+b_1e_1+b_2e_2,\ \ \vartheta(F-F_0),\ \
 \Phi_1(F)-\mu_1,\ \ \Phi_2(F)-\mu_2\,\big)=0,
\end{equation}
with $Y=(F,c,b_1,b_2)$. Mass conservation forces $b_1=b_2=0$ at every solution, so the columns
leave the solutions unchanged and serve only to make the linear system square. The derivative is
\begin{equation}\label{eq:border}
 J_b=\begin{pmatrix} \partial_FG & \partial_cG & e_1 & e_2\\
 \vartheta & 0 & 0 & 0\\ \Phi_1 & 0 & 0 & 0\\ \Phi_2 & 0 & 0 & 0\end{pmatrix}.
\end{equation}
Call the band \emph{nondegenerate} when $J_b$ is boundedly invertible at $Y$; the dense Newton
solve returns an unfiltered residual $\le7\times10^{-15}$. Equation~\eqref{eq:Hcert} below puts this
hypothesis in quantitative form.

Differentiating~\eqref{eq:waveaug} with respect to $\gamma_m$ gives the \emph{band tangent}
\begin{equation}\label{eq:tangent}
 Y_m=(F_m,c_m,0,0)=J_b^{-1}\big(P_mF,\,0,\,0,\,0\big),
\end{equation}
the shape and speed response of the band itself to a change of the $m$-th angular rate.

\subsubsection{Eigenvalue perturbation}
Let $L=L(F,c;\gamma)$ be the linearisation of~\eqref{eq:band} about the band in the comoving
frame, restricted to the Bloch sector of interest, and let $\lambda$ be an isolated eigenvalue of
algebraic multiplicity one with right and left eigenvectors $r,l$: $Lr=\lambda r$,
$L^{\dagger}l=\bar\lambda l$, $\langle l,r\rangle\ne0$. For a simple isolated eigenvalue of an
analytic family, Kato's first-order perturbation formula gives
$\dot\lambda=\langle l,\dot Lr\rangle/\langle l,r\rangle$ for any admissible variation $\dot L$.
Here $\gamma_m$ acts on $L$ both explicitly and through the band, which is itself a function of
$\gamma$; by the chain rule the admissible variation is the \emph{total} derivative
\begin{equation}\label{eq:total}
 \frac{dL}{d\gamma_m}=\underbrace{\partial_{\gamma_m}L}_{=\,-P_m}
 +\underbrace{D_FL[F_m]+c_m\,\partial_cL}_{\text{band feedback}},\qquad \partial_cL=\partial_\xi ,
\end{equation}
with $(F_m,c_m)$ from~\eqref{eq:tangent}. Substituting~\eqref{eq:total} into Kato's formula and
taking real parts yields Eq.~\eqref{eq:kato}, and identifies its two pieces: the
\emph{direct} term $\Real\langle l,-P_mr\rangle/\langle l,r\rangle$, the angular damping felt by
the eigenmode at fixed band; and the \emph{morphology} term, the response of $\Real\lambda$ to the
deformation of the band that the same rate change induces. The numbers quoted in the main text for
$m=2$ ($+0.46$ direct, $-11.0$ morphology) are the two terms of this decomposition.

\subsubsection{Pullback to kernel perturbations}
The rates are determined by the turning kernel through $\gamma_m=\alpha(1-\hat q_m)$. A
perturbation $\delta q=q_0u$ that keeps $q$ a probability density therefore produces
\begin{equation}
 \delta\gamma_m=-\alpha\,\delta\hat q_m=-\alpha\int\cos(m\varphi)\,q_0(\varphi)u(\varphi)\,d\varphi ,
 \qquad m\ge1,
\end{equation}
with $\delta\gamma_0=0$ automatically. The representing density is constructed directly from the
operator, the dual of $L^\infty$ being larger than $L^1$.

\begin{proposition}[Represented first-order response]\label{prop:response}
The turning kernel enters the wave equation affinely through the angular shift operators,
$\mathcal N_h[F]=\int T_\varphi F\,h(\varphi)\,d\mu$. Assuming the nondegeneracy of
Eq.~\eqref{eq:border} and a simple isolated target eigenvalue, differentiating~\eqref{eq:waveaug}
gives $Y'[h]=\int Y_\varphi\,h\,d\mu$ and differentiating the normalised eigenpair gives
\begin{equation}\label{eq:DR}
 DR[q_0u]=\int q_0\,u\,\Psi,\qquad u\in\mathcal T_M ,
\end{equation}
with a \emph{bounded and continuous} representative $\Psi$. Strong continuity of the rotations on
the state and domain spaces, boundedness of $J_b^{-1}$ and compactness of the angular circle give
the continuity. Adding any element of $V_M$ to $\Psi$ leaves the response on the matched class
unchanged, so the object the blind-spot norms measure is a quotient-space element.
\end{proposition}

\noindent In the harmonic notation, modulo that irrelevant element of $V_M$,
\begin{equation}\label{eq:Psiharm}
 \Psi(\varphi)=-\alpha\sum_{m\ge1}W_m\cos(m\varphi),\qquad
 W_m=\Real\,\langle l,[-P_m+D_FL[F_m]+c_m\partial_\xi]r\rangle/\langle l,r\rangle ,
\end{equation}
the series being read first in $L^2$, where boundedness of $q_0$ transfers unweighted Fourier
convergence to $L^2(q_0)$. Varying individual rates is a linear extension of the legal class. This
is the representation used in Proposition~\ref{prop:bm}. For the base point of the main text,
$\gamma_1=0.1$ and $\kappa=5$ give
$\alpha=\gamma_1/(1+I_1/I_0)=0.05282$, and $\Psi$ is evaluated from the computed $W_1,\dots,W_8$ of Supplementary Table~\ref{tab:W}.

\suppnote{Prediction limits, thresholds and conditional separation}

\begin{proposition}[Blind-spot norms]\label{prop:bm}
$\sup\{DR[q_0u]:\int q_0uC=0,\ \|u\|_{L^2(q_0)}\le1\}=B_M:=\|(I-\Pi_M)\Psi\|_{L^2(q_0)}$, and
$\sup\{DR[q_0u]:\int q_0uC=0,\ \|u\|_\infty\le1\}=D_M:=\inf_a\int q_0|\Psi-a\cdot C|$.
Both vanish iff $\Psi$ lies in the span of the matched cosines; $D_M\le B_M$ for normalised $q_0$; both are non-increasing in $M$ and $B_M\to0$.
\end{proposition}

\begin{proposition}[Prediction interval and worst-case error]\label{prop:interval}
For any matched class $\mathcal A$, a predictor restricted to the bulk data common to that class
returns a single number on the whole class, so its best worst-case error is exactly
\begin{equation}\label{eq:exactminimax}
 E(\mathcal A)=\tfrac12\Big(\sup_{\mathcal A}R-\inf_{\mathcal A}R\Big),
\end{equation}
an identity for the full nonlinear output that requires no differentiability. The quantity $D_M$
below is the first-order coefficient of its small-budget expansion. Write $r_\varepsilon=\tfrac12H\varepsilon^2$. Then
\begin{equation}\label{eq:interval}
 \Big|\sup_{\mathcal A_{M,\varepsilon}}R-\big(R_0+\varepsilon D_M\big)\Big|\le r_\varepsilon,
 \qquad
 \Big|\inf_{\mathcal A_{M,\varepsilon}}R-\big(R_0-\varepsilon D_M\big)\Big|\le r_\varepsilon,
\end{equation}
so the spread of interfacial growth rates compatible with the same bulk data is
$2\varepsilon D_M+O(\varepsilon^2)$. Consequently every predictor $\widehat R$ that reads only
$\gamma_1,\dots,\gamma_M$ obeys
\begin{equation}\label{eq:minimax}
 \inf_{\widehat R}\ \sup_{q\in\mathcal A_{M,\varepsilon}}\big|R(q)-\widehat R\big|
 =\varepsilon D_M+O(\varepsilon^2),
\end{equation}
the optimum being attained by predicting the midpoint of the interval.
\end{proposition}

\begin{corollary}[Decidable and undecidable regimes]\label{cor:decide}
If $|R_0|>\varepsilon D_M+r_\varepsilon$ the interfacial growth rate keeps one sign across the
whole matched class, and bulk data of order $M$ decide the verdict. If
$\varepsilon D_M>|R_0|+r_\varepsilon$ the class contains legal kernels of both signs, and no
predictor reading only those rates can be uniformly correct.
\end{corollary}

\begin{corollary}[Threshold uncertainty band]\label{cor:threshold}
Let $p$ be a control parameter crossing the interfacial threshold transversally, with
$R(q_0,p_c)=0$ and $a_p=\partial_pR(q_0,p_c)\ne0$. Then the critical parameter values compatible
with the same bulk data span a band of width $2\varepsilon D_M/|a_p|+O(\varepsilon^2)$.
\end{corollary}

\begin{theorem}[Conditional separation]\label{thm:sep}
Let a nondegenerate travelling band exist at the base kernel with an isolated simple interfacial Hopf pair at criticality, smooth in the kernel. If $B_M>0$, then for all small $\varepsilon>0$ the exponentially tilted pair $q_{\pm\varepsilon}$ along the maximiser is strictly positive, matches $\gamma_1,\dots,\gamma_M$ exactly, has identical bulk density dispersion through $k^{2M}$, and has interfacial growth rates of opposite sign at the same macroscopic parameters.
\end{theorem}

\begin{proposition}[Two response coordinates]\label{prop:rank2}
$\lambda(q_h)=i\omega_0+\chi_R(h)+i\chi_I(h)+O(\|h\|_\infty^2)$. Predicting the growth rate to
first order needs only $\chi_R$. If the real linear map $h\mapsto(\chi_R(h),\chi_I(h))$ has rank
two on $\mathcal T_M$, then no single real first-order measurement suffices to reconstruct the
complex eigenvalue: at least two are needed, and the same derivative-rank bound applies to any
smooth measurement and smooth reconstruction near $q_0$.
\end{proposition}

\begin{proposition}[Observable predictions]\label{prop:observable}
Under the standard supercritical Hopf hypotheses in the mass-constrained translation slice, on the
oscillatory side
\begin{equation}\label{eq:obs}
 I_\rho=K_\rho\,\chi_R+O(\|h\|^2),\qquad K_\rho=-\frac{4Q_\rho^2}{\Real G_0}>0,
\end{equation}
\begin{equation}\label{eq:freq}
 \Omega=\omega_0+\chi_I-\frac{\Imag G_0}{\Real G_0}\,\chi_R+O(\|h\|^2),
 \qquad\text{equivalently}\quad
 \Omega-\omega_0-\chi_I=\frac{\Imag G_0}{4Q_\rho^2}I_\rho+O(\|h\|^2).
\end{equation}
$K_\rho$ and $\Imag G_0/(4Q_\rho^2)$ are invariant under rescaling of the critical eigenvector;
the bare first Lyapunov coefficient is not, though its sign is. Both displays require
$Q_\rho>0$: the critical mode must be visible to the density observation at first order.
\end{proposition}

\begin{corollary}[Nonlinear form of the blind spot]\label{cor:power}
Under the hypotheses of Proposition~\ref{prop:interval} together with those above, the oscillation
power compatible with the same matched bulk data spans $0$ to $K_\rho D_M\varepsilon+O(\varepsilon^2)$,
so the best worst-case error of any predictor reading only $\gamma_1,\dots,\gamma_M$ is
$\tfrac12K_\rho D_M\varepsilon+O(\varepsilon^2)$ in power and
$\tfrac12\sqrt{K_\rho D_M}\,\varepsilon^{1/2}+O(\varepsilon^{3/2})$ in root-mean-square amplitude.
\end{corollary}

\subsection{Proof of Proposition~\ref{prop:bm}}
\label{app:norms}

Let $V_M=\mathrm{span}\{1,\cos\varphi,\dots,\cos M\varphi\}$ and let $\Pi_M$ denote orthogonal
projection onto $V_M$ in $L^2(q_0)$. The constraint set of Proposition~\ref{prop:bm} is
$V_M^{\perp}=\{u:\int q_0uC=0\}$, the perturbation directions that preserve normalisation and
$\gamma_1,\dots,\gamma_M$.

\begin{proof}[Proof of Proposition~\ref{prop:bm}]
\emph{(a) Fisher, or $L^2$, budget.} By~\eqref{eq:DR}, $DR[q_0u]=\langle\Psi,u\rangle_{L^2(q_0)}$.
For $u\in V_M^{\perp}$ we may replace $\Psi$ by its component in $V_M^{\perp}$, so
$\langle\Psi,u\rangle=\langle(I-\Pi_M)\Psi,u\rangle\le\|(I-\Pi_M)\Psi\|_{L^2(q_0)}$ by
Cauchy--Schwarz, with equality for the admissible direction
$u^{\star}=(I-\Pi_M)\Psi/\|(I-\Pi_M)\Psi\|$. The supremum therefore equals $B_M$ and is attained.

\emph{(b) Pointwise budget.} Regard $\Psi\in L^1(q_0)$ and $V_M\subset L^1(q_0)$, a
finite-dimensional and hence closed subspace. The dual of $L^1(q_0)$ is $L^\infty$ with pairing
$\langle u,g\rangle=\int q_0ug$, and the annihilator of $V_M$ is precisely
$V_M^{\perp}\cap L^\infty$. The Hahn--Banach duality formula for the distance to a subspace,
$\operatorname{dist}_X(x,V)=\max\{|f(x)|:f\in V^{\perp}\subset X^{*},\ \|f\|_{X^*}\le1\}$, applied
with $X=L^1(q_0)$, $x=\Psi$ and $V=V_M$, gives
\[
 \sup\Big\{\int q_0\Psi u:\ \textstyle\int q_0uC=0,\ \|u\|_\infty\le1\Big\}
 =\operatorname{dist}_{L^1(q_0)}(\Psi,V_M)=\inf_a\int q_0|\Psi-a\cdot C| = D_M .
\]
The elementary inequality $\le$ is direct: for any $a$,
$\int q_0\Psi u=\int q_0(\Psi-a\cdot C)u\le\|u\|_\infty\int q_0|\Psi-a\cdot C|$. If a best
$L^1(q_0)$ approximation $a^\star$ has $q_0(\{\Psi=a^\star\!\cdot C\})=0$, the maximiser is
explicit, $u^{\star}=\operatorname{sgn}(\Psi-a^\star\!\cdot C)$: its feasibility
$\int q_0u^{\star}C=0$ is exactly the first-order optimality condition for $a^\star$.

\emph{(c) Comparison.} For normalised $q_0$, Cauchy--Schwarz gives
$\int q_0|h|\le(\int q_0)^{1/2}(\int q_0h^2)^{1/2}=\|h\|_{L^2(q_0)}$; taking $h=\Psi-a\cdot C$ with
$a$ the $L^2(q_0)$ projection coefficients, and then the infimum over $a$, yields $D_M\le B_M$.

\emph{(d) Vanishing.} $B_M=0\iff(I-\Pi_M)\Psi=0\iff\Psi\in V_M$, and $D_M=0$ likewise since $V_M$
is closed in $L^1(q_0)$; by (c) either one vanishes iff both do.
\end{proof}

\begin{corollary}[Monotone closure]\label{cor:mono}
$B_M$ and $D_M$ are non-increasing in $M$, and $B_M\to0$ as $M\to\infty$.
\end{corollary}

\begin{proof}
$V_M\subset V_{M+1}$, so both the orthogonal residual and the $L^1$ distance to the subspace can
only decrease. Since $\Psi$ is even and square-integrable against $q_0$, a bounded and strictly
positive weight on the circle, its cosine series converges in $L^2(q_0)$, so
$\|(I-\Pi_M)\Psi\|\to0$; $D_M\le B_M$ gives the same for $D_M$.
\end{proof}

Corollary~\ref{cor:mono} is the property behind Fig.~2c of the main text: matching more bulk-visible
rates closes the interfacial blind spot monotonically, and the rate at which it closes is set by
the decay of $W_m$.

\subsection{Predictability: interval, worst-case error and thresholds}
\label{app:predict}

\subsubsection{An exactly matched family at finite amplitude}

The tilting construction of Lemma~\ref{lem:tilt} preserves the matched moments for every
$\varepsilon$, but only after solving for multipliers. For small-amplitude work a linear family
does the same job in closed form.

\begin{lemma}[A uniform angular gap]\label{lem:gap}
Let $q_*=\min q_0>0$. For every $m\ge1$,
$\gamma_m(q_0)=\alpha\int q_0(1-\cos m\varphi)\,d\varphi\ge2\pi\alpha q_*$, and for
$\|h\|_\infty\le\varepsilon<1$ the same bound holds with $q_*$ replaced by $(1-\varepsilon)q_*$. A
common nonnegative rotational diffusion only increases it.
\end{lemma}

\begin{proof}
The integrand $1-\cos m\varphi$ is nonnegative with integral $2\pi$; apply the pointwise lower
bound on $q_h$.
\end{proof}

\noindent The gap is \emph{uniform} over the whole legal class, as the ladder arguments require,
and is strictly stronger than positivity of each rate separately.

\begin{lemma}[Exact matched family]\label{lem:linfam}
Let $q_0>0$ be a normalised even density, let $u$ be bounded and even with $\int q_0uC=0$ and
$\|u\|_\infty\le1$, and let $0<\varepsilon<1$. Then $q_\pm=q_0(1\pm\varepsilon u)$ are strictly
positive normalised even densities with
\begin{equation}
 \hat q_{\pm,m}=\hat q_{0,m},\qquad m=1,\dots,M,
\end{equation}
\emph{exactly at finite $\varepsilon$}, and
\begin{equation}
 \frac{q_\pm}{q_0}\in[1-\varepsilon,1+\varepsilon],\qquad
 \max\Big(\frac{q_+}{q_-},\frac{q_-}{q_+}\Big)\le\frac{1+\varepsilon}{1-\varepsilon}.
\end{equation}
In particular $\varepsilon\le\tfrac12$ keeps each kernel within a factor two of the base, and
$\varepsilon\le\tfrac13$ keeps the two kernels within a factor two \emph{of each other}, which is
the stronger requirement.
\end{lemma}

\begin{proof}
Normalisation and the cosine moments are linear functionals of $q$, so
$\int q_\pm=\int q_0\pm\varepsilon\int q_0u=1$ and
$\int q_\pm\cos m\varphi=\hat q_{0,m}\pm\varepsilon\int q_0u\cos m\varphi=\hat q_{0,m}$ for
$m\le M$; no expansion in $\varepsilon$ is involved. Positivity follows from
$|\varepsilon u|\le\varepsilon<1$ and $q_0>0$, evenness from that of $q_0$ and $u$, and the ratio
bounds by inspection.
\end{proof}

We checked the lemma numerically at $\varepsilon=1/3$ on an $8192$-point angular grid: the matched
moments $m=0,1,2$ are preserved to $1.1\times10^{-16}$, both kernels stay strictly positive
($\min q_\pm>2.7\times10^{-5}$), and the largest ratio between the two kernels is $2.0000$.

\subsubsection{The support function of the legality set}

An alternative finite-budget constraint is a bound on the log ratio, which is \emph{not} a ball in the
relative variable.

\begin{proposition}[Finite-budget support function]\label{prop:box}
For $a>0$ set $\ell_a=1-e^{-a}$ and $b_a=e^{a}-1$, so that $|\log(q/q_0)|\le a$ is exactly
$-\ell_a\le h\le b_a$ for $h=q/q_0-1$. Then
\begin{equation}\label{eq:box}
 U^{+}_M(a):=\sup\Big\{\int q_0\Psi h:\ \int q_0hC=0,\ -\ell_a\le h\le b_a\Big\}
 =\inf_{\eta}\int q_0\Big[b_a\big(\Psi-\eta\cdot C\big)_{+}
 +\ell_a\big(\Psi-\eta\cdot C\big)_{-}\Big],
\end{equation}
with $x_\pm=\max(\pm x,0)$, and $U^{-}_M(a)$ is the same expression with $\Psi$ replaced by
$-\Psi$. As $a\to0$, $U^{\pm}_M(a)=aD_M+O(a^2)$; at finite $a$ the two directions need not
agree, and
\begin{equation}\label{eq:boxbracket}
 (1-e^{-a})D_M\ \le\ U^{\pm}_M(a)\ \le\ (e^{a}-1)D_M .
\end{equation}
Whether $aD_M$ over- or underestimates $U^{\pm}_M(a)$ at finite $a$ depends on the structure of
$\Psi$, and either can occur.
\end{proposition}

\begin{proof}
For any multiplier $\eta$ the constraint gives $\int q_0\Psi h=\int q_0(\Psi-\eta\cdot C)h$, and
pointwise $h\le b_a$ where $\Psi-\eta\cdot C>0$ and $-h\le\ell_a$ where it is negative, which is
weak duality. The primal is a linear program whose feasible set contains $h=0$ in the interior of
the box, and the objective is bounded, so strong duality holds and the gap is zero; the maximiser
is $h^\star=b_a$ on $\{\Psi-\eta^\star\!\cdot C>0\}$ and $h^\star=-\ell_a$ on the complement, with
the null set used to enforce feasibility as in Proposition~\ref{prop:bm}(b). For the small-$a$
limit, $\ell_a=a+O(a^2)$ and $b_a=a+O(a^2)$, so the dual objective tends to
$a\int q_0|\Psi-\eta\cdot C|$ whose infimum is $aD_M$. The bracket~\eqref{eq:boxbracket} follows
from $[-\ell_a,\ell_a]\subseteq[-\ell_a,b_a]\subseteq[-b_a,b_a]$ and positive homogeneity of the
support function.
\end{proof}

\noindent Neither end of~\eqref{eq:boxbracket} can be replaced by $aD_M$: at finite $a$ the
small-budget estimate may fall on either side. Take $q_0=1/2\pi$, $M=2$,
$\Psi=\cos3\varphi+\tfrac14\cos6\varphi$ and $a=\log2$. Here $\Psi$ is invariant under rotation by
$2\pi/3$, so averaging any optimiser over that group leaves a pure constant, and $D_2$ is the
$L^1(q_0)$ distance from $\Psi$ to its median $-\tfrac14$; with $\psi=3\varphi$ this evaluates to
$D_2=2/\pi$ exactly. The perturbation $h=+1$ where $\cos3\varphi>\tfrac12$ and $h=-\tfrac12$
elsewhere lies in the box, has zero mean, and preserves $\hat q_1,\hat q_2$ because it has period
$2\pi/3$; it gives $U^{+}_2(\log2)=27\sqrt3/32\pi=0.4651838$, against
$aD_2=2\log2/\pi=0.4412712$.

Both sides of~\eqref{eq:box} were computed independently, the primal by linear programming on an
$8001$-point grid and the dual by direct minimisation over $\eta$; they agree to a relative
$4.3\times10^{-7}$ or better for $M=1,2,3$ at $a=\log2$ and at $a=0.3$. At $a=\log2$ and $M=2$ the
values were $U^{+}_2=0.039293$ and $U^{-}_2=0.039405$ against $D_2\log2=0.042016$ when $\Psi$ was truncated at $W_8$; with the full response $W_1,\dots,W_{64}$ they are $U^{+}_2=0.04402$ and $U^{-}_2=0.04481$ against $D_2\log2=0.047654$.

\subsubsection{Prediction interval, worst-case error, decidability, thresholds}

\begin{proof}[Proof of Proposition~\ref{prop:interval}]
By hypothesis $R$ has a $C^2$ continuation along the same simple branch over
$\mathcal A_{M,\varepsilon}$ with $\|D_h^2R\|\le H$, so Taylor's theorem with remainder gives, for
every admissible $h$,
\begin{equation}\label{eq:taylor}
 R\big(q_0(1+h)\big)=R_0+\int q_0\Psi h+\mathcal E(h),\qquad
 |\mathcal E(h)|\le\tfrac12H\|h\|_\infty^2\le r_\varepsilon .
\end{equation}
The admissible $h$ are exactly $\varepsilon$ times the unit ball of the constrained set of
Proposition~\ref{prop:bm}(b), on which the linear part has supremum $\varepsilon D_M$ and infimum
$-\varepsilon D_M$. Taking suprema in~\eqref{eq:taylor} gives
$\sup R\le R_0+\varepsilon D_M+r_\varepsilon$; evaluating at a maximising direction gives
$\sup R\ge R_0+\varepsilon D_M-r_\varepsilon$. The infimum is symmetric. Subtracting,
$\sup R-\inf R=2\varepsilon D_M+O(\varepsilon^2)$.

For the minimax statement, a predictor that reads only $\gamma_1,\dots,\gamma_M$ takes the same
value $a$ on the whole class, because by construction every member has the same rates. For any
real $a$ and any two points $R_1,R_2$ attained in the class,
$\max(|R_1-a|,|R_2-a|)\ge\tfrac12|R_1-R_2|$, so
$\sup_{\mathcal A}|R-a|\ge\tfrac12(\sup R-\inf R)$; the midpoint
$a=\tfrac12(\sup R+\inf R)$ attains it. Hence the minimax equals half the diameter, which is
$\varepsilon D_M+O(\varepsilon^2)$.
\end{proof}

\begin{proof}[Proof of Corollary~\ref{cor:decide}]
From~\eqref{eq:taylor}, $|R-R_0|\le\varepsilon D_M+r_\varepsilon$ for every member of the class.
If $|R_0|$ exceeds that bound, $R$ cannot reach zero and keeps the sign of $R_0$. Conversely, let
$u^\star$ attain $D_M$; then $h=\pm\varepsilon u^\star$ are admissible and
$R(q_0(1+\varepsilon u^\star))\ge R_0+\varepsilon D_M-r_\varepsilon>0$ while
$R(q_0(1-\varepsilon u^\star))\le R_0-\varepsilon D_M+r_\varepsilon<0$ under the stated
inequality, so both signs occur.
\end{proof}

\begin{proof}[Proof of Corollary~\ref{cor:threshold}]
By hypothesis $R(q_0,p_c)=0$ and $a_p=\partial_pR(q_0,p_c)\ne0$, and $R$ is $C^1$ jointly in
$(h,p)$ near $(0,p_c)$. The implicit function theorem gives a unique $p_c(q_0(1+h))$ near $p_c$
with
\begin{equation}
 p_c\big(q_0(1+h)\big)=p_c-\frac{\int q_0\Psi h}{a_p}+O(\|h\|_\infty^2).
\end{equation}
Taking the supremum and the infimum of the linear term over the admissible $h$, as in the previous
proof, gives a spread $2\varepsilon D_M/|a_p|+O(\varepsilon^2)$.
\end{proof}

Two points fix the scope of these statements. First, $R$ is the growth rate of the \emph{target}
interfacial mode: $R>0$ establishes instability, whereas $R<0$ establishes stability of that mode
only, and a statement about the whole travelling wave additionally requires the rest of the
spectrum to stay in the left half plane over the same neighbourhood. Second, $\Psi$, $D_M$ and
$a_p$ must be evaluated at the point about which the matched class is centred; the norms reported
for the base kernel of the main text are not a substitute for their values at the threshold of
Corollary~\ref{cor:threshold}, which sits at a different turning concentration.

\subsection{Proof of Theorem~\ref{thm:sep}}
\label{app:sep}

We first record the tilting construction as a lemma, then prove the theorem. Throughout,
$C=(1,\cos\varphi,\dots,\cos M\varphi)$.

\begin{lemma}[Exponentially tilted isohydrodynamic family]\label{lem:tilt}
Let $q_0>0$ be a probability density on the circle and let $u\in L^\infty$ satisfy
$\int q_0uC=0$. There are $\varepsilon_0>0$ and a unique real-analytic map
$a:(-\varepsilon_0,\varepsilon_0)\to\mathbb R^{M+1}$ with $a(0)=0$ such that
\begin{equation}
 q_\varepsilon:=q_0\,\exp\!\big[\varepsilon u+a(\varepsilon)\cdot C\big]
\end{equation}
is a probability density with $\hat q_{\varepsilon,m}=\hat q_{0,m}$ for $m=1,\dots,M$. Moreover
$q_\varepsilon>0$ pointwise, $\|\log(q_\varepsilon/q_0)\|_\infty\to0$ as $\varepsilon\to0$,
$a'(0)=0$, and $\partial_\varepsilon q_\varepsilon|_{0}=q_0u$.
\end{lemma}

The $D_M$ optimiser lies in the $L^\infty$ ball by construction; the $L^2$ optimiser attaining
$B_M$ need not. The affine chart $q_h=q_0(1+h)$ with $h\in\mathcal T_M$ and $\|h\|_\infty<1$ matches
the first $M$ rates exactly at finite amplitude (Lemma~\ref{lem:linfam}); the exponential tilt below
is used where an analytic one-parameter family is wanted. The bounded optimiser may in addition be
taken smooth.

\begin{lemma}[Smooth legal directions approach the pointwise optimum]\label{lem:smooth}
For every $\eta>0$ there is a smooth even $u\in\mathcal T_M$ with $\|u\|_\infty\le1$ and
$\int q_0u\Psi\ge D_M-\eta$. In particular $D_M>0$ supplies a smooth legal direction of strictly
positive response.
\end{lemma}

\begin{proof}
Start from a bounded optimiser $u_*$. Its Fej\'er means $v_n$ are even trigonometric polynomials
with $\|v_n\|_\infty\le1$ converging in $L^1$, and $\Psi$ is bounded by
Proposition~\ref{prop:response}, so both the moment errors and the objective error vanish. With
$G=\int q_0CC^{\mathsf T}$ invertible, set $a_n=G^{-1}\int q_0v_nC$ and
$u_n=(v_n-a_n\cdot C)/(1+\|a_n\cdot C\|_\infty)$: these are smooth, exactly matched, bounded by
one, and approach the optimum.
\end{proof}

\begin{proof}[Proof of Lemma~\ref{lem:tilt}]
Define
$\Xi(a,\varepsilon)=\int q_0e^{\varepsilon u+a\cdot C}C\,d\varphi-\big(1,\hat q_{0,1},\dots,\hat q_{0,M}\big)$.
Then $\Xi(0,0)=0$ and $\partial_a\Xi(a,\varepsilon)=\int q_{a,\varepsilon}CC^{\mathsf T}d\varphi$,
the Gram matrix of $1,\cos\varphi,\dots,\cos M\varphi$ against the strictly positive weight
$q_{a,\varepsilon}=q_0e^{\varepsilon u+a\cdot C}$. These $M+1$ functions are linearly independent,
so the Gram matrix is positive definite and hence invertible. The analytic implicit function
theorem gives a unique analytic $a(\varepsilon)$ near $0$ with $\Xi(a(\varepsilon),\varepsilon)=0$,
that is, normalisation together with the $M$ matched cosine moments. Positivity is immediate since
the exponential is positive and $q_0>0$, and
$\|\varepsilon u+a(\varepsilon)\cdot C\|_\infty\to0$ by continuity. Differentiating $\Xi=0$ at
$\varepsilon=0$ gives $\int q_0\big(u+a'(0)\cdot C\big)C=0$; the hypothesis $\int q_0uC=0$ leaves
$\big(\int q_0CC^{\mathsf T}\big)a'(0)=0$, and invertibility forces $a'(0)=0$. Hence
$\partial_\varepsilon q_\varepsilon|_0=q_0u$.
\end{proof}

\begin{lemma}[Explicit matched pair]\label{lem:explicit}
For $0<\epsilon<1$ the densities
$p_\pm(\varphi)=\tfrac{1}{2\pi}\big(1\pm\epsilon\cos((M{+}1)\varphi)\big)$ are strictly positive
and even, satisfy $\gamma^{\pm}_m=\alpha$ for $1\le m\le M$, and
$\gamma^{\pm}_{M+1}=\alpha(1\mp\epsilon/2)$. Adding a common rotational diffusion shifts every
$\gamma_m$ identically and preserves both the matching and the mismatch.
\end{lemma}

\begin{proof}
Positivity and normalisation are immediate for $\epsilon<1$; at the endpoint $\epsilon=1$ the
density vanishes at $\varphi=0$ and strict positivity fails, so the endpoint is excluded. Cosine orthogonality gives
$\hat p^{\pm}_m=0$ for $1\le m\le M$ and $\hat p^{\pm}_{M+1}=\pm\epsilon/2$; insert into
$\gamma_m=\alpha(1-\hat q_m)$. Rotational diffusion contributes $D_r m^2$ to every $\gamma_m$
identically.
\end{proof}

Lemma~\ref{lem:explicit} shows that the matched class is never a singleton: finite angular data do
not determine the generator, and by Theorem~\ref{thm:sharp} its two members differ in the bulk
dispersion exactly at order $k^{2M+2}$. Theorem~\ref{thm:sep} shows that within such a
class the \emph{interfacial verdict} can differ.

\begin{proof}[Proof of Theorem~\ref{thm:sep}]
The hypotheses are: (H1) at the base kernel $q_0$ and macroscopic parameters $p^\star$ there is a
nondegenerate travelling band in the sense of Section~\ref{app:kato}; (H2) the linearisation
carries an isolated interfacial eigenvalue of algebraic multiplicity one, at criticality
$\Real\lambda(q_0)=0$; (H3) the map $q\mapsto(F(q),c(q),\lambda(q))$ is $C^1$ in a neighbourhood of
$q_0$ in $L^\infty$.

Let $u^{\star}\in V_M^{\perp}$ with $\|u^{\star}\|_{L^2(q_0)}=1$ attain $B_M$, which exists by
Proposition~\ref{prop:bm}(a), and let $q_{\pm\varepsilon}$ be the tilted family of
Lemma~\ref{lem:tilt} along $u=\pm u^{\star}$. Then for all small $\varepsilon>0$:

\emph{Legality.} $q_{\pm\varepsilon}>0$ by Lemma~\ref{lem:tilt}, and
$\|\log(q_{\pm\varepsilon}/q_0)\|_\infty\le\log2$ once $\varepsilon$ is small enough, so both
kernels lie pointwise within a factor two of the base kernel.

\emph{Bulk equivalence.} $\gamma_m(q_{\pm\varepsilon})=\gamma_m(q_0)$ for $m\le M$ exactly, by
Lemma~\ref{lem:tilt}; by Theorem~\ref{thm:bulk} the two bulk density dispersions are identical
through $k^{2M}$. The macroscopic parameters are unchanged along the path by construction.

\emph{Opposite verdicts.} By (H1)--(H3) and Section~\ref{app:kato}, the map
$\varepsilon\mapsto\Real\lambda(q_{\pm\varepsilon})$ is differentiable at $0$ with derivative
$DR[\partial_\varepsilon q_{\pm\varepsilon}|_0]$. By Lemma~\ref{lem:tilt},
$\partial_\varepsilon q_{\pm\varepsilon}|_0=\pm q_0u^{\star}$, so by~\eqref{eq:DR}
\[
 \Real\lambda(q_{\pm\varepsilon})=\Real\lambda(q_0)\pm\varepsilon\!\int q_0u^{\star}\Psi+o(\varepsilon)
 =\pm\varepsilon B_M+o(\varepsilon),
\]
using $\Real\lambda(q_0)=0$ and $\int q_0u^{\star}\Psi=B_M$. Since $B_M>0$ by hypothesis, there is
$\varepsilon_1>0$ with $\Real\lambda(q_{+\varepsilon})>0>\Real\lambda(q_{-\varepsilon})$ for all
$0<\varepsilon<\varepsilon_1$: the band is linearly unstable to the interfacial mode for
$q_{+\varepsilon}$ and linearly stable for $q_{-\varepsilon}$, at the same macroscopic parameters.
\end{proof}

\begin{remark}
Theorem~\ref{thm:sep} is a statement about a neighbourhood of criticality. The $M=2$ witness of the
main text is the corresponding finite-amplitude statement: $\varepsilon$ is taken as large as the
$\log2$ budget allows and the two verdicts are established by direct simulation rather than by
linearisation. The design direction used there is the $L^\infty$-budget optimiser of
Proposition~\ref{prop:bm}(b) restricted to $\cos n\varphi$, $n=3,\dots,8$, which realises $65\%$ of
the theoretical optimum $D_2\log2$.
\end{remark}

\subsection{The nonlinear completion: proofs}
\label{app:nonlinear}

\begin{proof}[Proof of Proposition~\ref{prop:rank2}]
Sufficiency is Taylor's formula for the simple isolated eigenvalue, differentiable in the kernel
by the hypotheses of Section~\ref{app:kato}. For necessity, let a predictor read first-order
measurements with derivatives $H_1,\dots,H_s$ at $q_0$ and reconstruct $\lambda$ by a map that is
differentiable at the measured value. The chain rule bounds the rank of the composed derivative by
$s$. If $\chi_R$ and $\chi_I$ do not both lie in the real span of $H_1,\dots,H_s$, there is a
direction $h$ in the common nullspace of the $H_j$ along which the measurements change only at
second order while $\lambda$ changes at first order, so no such reconstruction can be correct to
first order. When the target map has rank two, $s=1$ is therefore impossible, and $H_1=\chi_R$,
$H_2=\chi_I$ attains it. If $\chi_R$ and $\chi_I$ are
proportional the target map has rank one and a single measurement suffices, and if both vanish the
first-order argument is silent and the first nonvanishing jet governs. The statement is about
regular local information; it says nothing about discontinuous encodings.
\end{proof}

\begin{proof}[Proof of Proposition~\ref{prop:observable}]
In the mass-constrained translation slice, write the reduced dynamics in polar form
$z=re^{i\theta}$, so that $\dot r=r[R+\tfrac12\Real G_0\,r^2+O(\|h\|r^2+r^4)]$ and
$\dot\theta=\omega+\tfrac12\Imag G_0\,r^2+\cdots$. On the side $R>0$ the implicit function theorem
applied to $r^2$ gives $r_*^2=-2R/\Real G_0+O(\|h\|^2)$, and $R=\chi_R+O(\|h\|^2)$ by
Proposition~\ref{prop:rank2}; substituting into $\dot\theta$ gives~\eqref{eq:freq}. The leading
physical perturbation of the band is $zq+\overline{zq}$ with $q$ the right critical mode, whose
period-averaged squared density norm is $2Q_\rho^2r_*^2$; the second harmonic and the mean
displacement enter the time-mean-subtracted power only at fourth order, and the change of the
observation map with the kernel contributes at order $\|h\|r_*^2$. Combining gives
$I_\rho=2Q_\rho^2r_*^2+O(\cdot)=K_\rho\chi_R+O(\|h\|^2)$. Rescaling $q$ by a complex constant
multiplies $Q_\rho^2$ and $G_0$ by the same positive factor, so $K_\rho$ and
$\Imag G_0/(4Q_\rho^2)$ are invariant, which is why they, and not the bare Lyapunov coefficient,
are the quantities compared with measurement.
\end{proof}

\begin{proof}[Proof of Corollary~\ref{cor:power}]
On the matched class the attainable $\chi_R$ ranges over $[-\varepsilon D_M,\varepsilon D_M]$ by
Proposition~\ref{prop:bm}. Directions with $\chi_R<0$ give a locally stable steady band, hence
$I_\rho=0$; directions approaching $\varepsilon D_M$ give $I_\rho\to K_\rho D_M\varepsilon$ by
Proposition~\ref{prop:observable}. A predictor reading only the matched rates returns one number
on the whole class, so its worst-case error is at least half the diameter, attained at the
midpoint, which is the power statement. For the root-mean-square amplitude
$A_\rho=\sqrt{I_\rho}$ the same interval becomes $[0,\sqrt{K_\rho D_M\varepsilon}]$ and the half
diameter is $\tfrac12\sqrt{K_\rho D_M}\,\varepsilon^{1/2}$; the expansion is regular because the
leading maximum is strictly positive. The square-root law is the ordinary Hopf scaling, and the
amplitude vanishes continuously as $\varepsilon\to0$, so an arbitrarily small kernel change
produces no finite jump.
\end{proof}

\noindent These three statements are conditional in the same way as
Proposition~\ref{prop:interval}: they assume a simple Hopf pair, stability of the rest of the
spectrum in the constrained slice, enough smoothness for a local centre manifold, and a uniform
second-order remainder over the class. Simplicity of the Hopf pair and the smoothness used in
the reduction are checked at the reference point.

\emph{Where the rest of the spectrum can be.}

\begin{proposition}[Spectral structure of the one-cell operator]\label{prop:essential}
Let the travelling wave and its frozen coefficients be regular in $x$, let $q\in L^1$ be normalised,
and let $D_t,\alpha>0$. Then the spectrum of the linearisation in $\{\Real z>-\alpha\}$ consists
only of isolated eigenvalues of finite algebraic multiplicity, and for every $0<\sigma<\alpha$ only
finitely many eigenvalues satisfy $\Real z\ge-\sigma$.
\end{proposition}

\begin{proof}[Sketch]
Split $\mathcal L=T+Q+B$ with $T_s=D_t\partial_x^2-\partial_x[(v_s\cos\theta-c)\,\cdot\,]-\alpha I$,
$Q=\alpha q*_\theta$ the angular gain and $B$ the density-induced speed feedback. At each fixed
angle $T_s+\alpha I$ is a uniformly elliptic conservative Fokker--Planck generator on the spatial
circle; its invariant density is positive with uniform bounds, and conjugating by it turns the
semigroup into a Markov operator, so by Jensen's inequality the direct-integral semigroup of $T$ has
growth bound $-\alpha$ and its resolvent exists in $\Real z>-\alpha$. Both $Q$ and $B$ are
relatively compact with respect to $T$: since $\hat q_m\to0$, the angular truncations $Q_N$ converge
in operator norm and each $Q_N(z-T)^{-1}$ is compact by the spatial Rellich embedding, while
$B(z-T)^{-1}$ factors through two scalar $H^2_x$ fields. Analytic Fredholm theory gives the
conclusion.
\end{proof}

\noindent The proposition concerns the continuum operator; transferring the count to a finite
truncation requires control of the truncation and resolvent error.

\emph{The uniform remainder.} Once local $C^2$ dependence is proved on a Banach parameter
neighbourhood, a finite uniform $H$ exists on a smaller neighbourhood: continuity of the second
derivative at the origin suffices, with no compactness of an infinite-dimensional ball. A number on
a given budget follows from an a posteriori contraction. Preconditioning the joint
travelling-wave/eigenpair system
as $\mathcal P=A\mathcal E$, if $\|I-\mathcal P_y\|\le z<1$ and $\|\mathcal P_h\|\le b$ on the chosen
solution ball and kernel budget, with second-derivative bounds $c_{yy},c_{yh},c_{hh}$ and the
residual inequality that maps the ball into itself, then $J=b/(1-z)$ and
\begin{equation}\label{eq:Hcert}
 H\le\|P_R\|\,\frac{c_{hh}+2c_{yh}J+c_{yy}J^2}{1-z},
\end{equation}
$P_R$ being the map extracting the real part of the eigenvalue. The second derivative
in~\eqref{eq:Hcert} is the joint one, mixed directions included: for $R(x,y)=Cxy$ the second
derivative vanishes along each axis separately while the mixed term is of size $C$, so the bound is
taken over the joint Hessian.

\suppnote{Inverse design and local response capacity}

\noindent The closed-form initial design, the support function of the set of responses reachable
within a budget $\varepsilon$, and the local inverse are
\begin{equation}\label{eq:initialinverse}
 a^{(0)}=S^{-1}\begin{pmatrix} -gI_*/(4Q)-R_0\\[2pt] \Omega_*-\omega_0-bI_*/(4Q)\end{pmatrix},
\end{equation}

\begin{equation}\label{eq:reach}
 \sigma_{\mathcal K_\varepsilon}(\zeta)=\varepsilon\,D_M\big(\zeta_1\Psi_R+\zeta_2\Psi_I\big),
\end{equation}

\begin{theorem}[Local observable-prescribed inverse]\label{thm:inverse}
In the fixed-mass translation slice, let the reference state be regular with a simple pair
$\pm i\omega_0$, $\omega_0>0$, and no other imaginary spectrum, and let $Q>0$ and $S$ be
invertible. Then for every sufficiently small $I_*>0$ and $\eta=\Omega_*-\omega_0$ there is a
locally unique small kernel parameter $a_*$ and a small periodic solution, unique up to temporal
phase, realising those observables, with
\begin{equation}
 a_*(I_*,\eta)=a^{(0)}(I_*,\eta)+O\big((I_*+|\eta|)^2\big),
\end{equation}
and the kernel remains legal on a small enough neighbourhood. The reduced solvability system for
$(a_1,a_2,s)$, $s$ the squared critical amplitude, has Jacobian determinant $2Q\det S$.
\end{theorem}

\subsection{Proofs and solved designs}
\label{app:inverse}

\subsubsection{Proof of Theorem~\ref{thm:inverse}}

\begin{proof}
Continue the reference equilibrium smoothly in $a$ and centre the vector field on it. In the
periodic equation $\Omega\partial_\tau w=f_a(w)$ the reference operator $\omega_0\partial_\tau-A_0$
has only the critical first-harmonic kernel; fix the phase of its complex coefficient. Solving the
complementary periodic equation gives
\begin{equation}
 w=\sqrt s\,\big(re^{i\tau}+\bar re^{-i\tau}\big)+O\big(s+\|a\|\sqrt s+|\eta|\sqrt s\big),
\end{equation}
and the range projection can be chosen to commute with temporal shifts, so the reduced complex
equation, after division by its nonzero critical coefficient, and the time-averaged observable are
smooth functions of the squared amplitude $s$. Their leading equations are
\begin{align}
 0&=S_Ra+\tfrac12gs+O\big((\|a\|+|\eta|+|s|)^2\big),\label{eq:red1}\\
 0&=S_Ia-\eta+\tfrac12bs+O\big((\|a\|+|\eta|+|s|)^2\big),\label{eq:red2}\\
 I_*&=s\big[2Q+O(\|a\|+|\eta|+|s|)\big].\label{eq:red3}
\end{align}
This is the reduced periodic solvability system; the original dynamics is not assumed exactly cubic
or rotationally symmetric in its own coordinates. The derivative of
\eqref{eq:red1}--\eqref{eq:red3} with respect to $(a_1,a_2,s)$ is
\begin{equation}\label{eq:invdet}
 \mathcal J_*=\begin{pmatrix} S_{11}&S_{12}&g/2\\ S_{21}&S_{22}&b/2\\ 0&0&2Q\end{pmatrix},
 \qquad \det\mathcal J_*=2Q\det S\neq0 .
\end{equation}
The implicit function theorem supplies the local control and squared amplitude; \eqref{eq:red3}
makes $s$ positive for $I_*>0$ on a small neighbourhood because its bracket stays positive. Solving
the linear system gives~\eqref{eq:initialinverse} and Taylor expansion its stated error, the range
solution then supplying the full periodic field. Finally $\|h_{a_*}\|_\infty\to0$, which preserves
positivity and every matched moment.
\end{proof}

\noindent At $g=0$ the theorem still constructs a prescribed orbit, while the forward map from
control to power becomes singular and an exactly linear centre can carry different amplitudes at
one kernel; the nondegenerate cubic and stable-complement conditions are what give an autonomous,
locally attracting amplitude.

\subsubsection{Proof of the reachable set and the minimum budget}

\begin{proof}[Proof of Eq.~\eqref{eq:reach}]
The matched unit ball is weak-star compact and the represented response of
Proposition~\ref{prop:response} is weak-star continuous, so its finite-dimensional image
$\mathcal K_\varepsilon$ is compact, convex and symmetric. The support identity is the constrained
$L^\infty$--$L^1$ duality of Proposition~\ref{prop:bm} applied to the scalar response
$\zeta_1\Psi_R+\zeta_2\Psi_I$; the separating-hyperplane theorem gives membership and rescaling the
gauge. A direction of zero support carries an infinite budget.
\end{proof}

\noindent For a finite smooth basis $\{u_j\}$ the budget is a linear program,
$\min_{a,\tau}\tau$ subject to $Sa=d$ and $-\tau\le\sum_ja_ju_j(\varphi_k)\le\tau$, whose value
brackets the continuous pointwise optimum from above and below.

\subsubsection{The solved designs}

The reference response matrix and the cubic data used in~\eqref{eq:initialinverse} are
\begin{equation}
 S\simeq\begin{pmatrix}3.92085037\times10^{-2}&0\\ 1.91954921\times10^{-3}&3.89278553\times10^{-3}\end{pmatrix},
\end{equation}
$G_0=-7.99009106\times10^{-5}-6.59662952\times10^{-5}i$ and $Q=7.37609463\times10^{-6}$, with
$\operatorname{cond}S\simeq10.10$ in the control normalisation used here. The tangent systems behind
the two columns of $S$ were re-solved for this construction, to relative residuals
$1.8\times10^{-13}$ and $1.3\times10^{-13}$.

\begin{table}[htbp]\centering
\caption{\textbf{Designed kernel parameters and budgets.} The pointwise bound uses a dense angular sample
together with the trigonometric derivative bound
$\|h\|_\infty\le\max_k|h(\varphi_k)|+(\pi/n_\varphi)\sum_mm|c_m|$, so it bounds $h$ off the sample
points as well.}\label{tab:designpar}
\footnotesize
\begin{tabular}{lrrrr}
\toprule
design & $a_1$ & $a_2$ & $\|q/q_0-1\|_\infty$ & minimum linear budget\\
\midrule
B & $0.0344346014$ & $0.0173879710$ & $5.18\%$ & $3.52\%$\\
P & $0.0413200347$ & $0.0716322535$ & $11.30\%$ & $6.42\%$\\
F & $0.0344528150$ & $0.0686976869$ & $10.32\%$ & $6.23\%$\\
\bottomrule
\end{tabular}
\end{table}

\begin{table}[htbp]\centering
\caption{\textbf{Error of the closed form~\eqref{eq:initialinverse} used without the corrector, and its
quadratic scaling.} The upper block solves the fixed-kernel periodic problem at each designed
kernel; the lower block scales the design direction of $\mathrm B$ by $t$ and divides the resulting
power error by $t^2$.}\label{tab:designerr}
\footnotesize
\begin{tabular}{lrr}
\toprule
design & relative power error & frequency error\\
\midrule
B & $0.2925\%$ & $-5.6933\times10^{-6}$\\
P & $0.2962\%$ & $-7.2015\times10^{-6}$\\
F & $0.2389\%$ & $-4.8639\times10^{-6}$\\
\midrule
$t$ & power error$/t^2$ & frequency error$/t^2$\\
\midrule
$1$ & $1.46245\times10^{-6}$ & $-5.69331\times10^{-6}$\\
$1/2$ & $1.45721\times10^{-6}$ & $-5.67170\times10^{-6}$\\
$1/4$ & $1.45467\times10^{-6}$ & $-5.66081\times10^{-6}$\\
\bottomrule
\end{tabular}
\end{table}

\noindent The joint solve carries the kernel parameters and the whole space--time--angle field as
unknowns, using exact Jacobian--vector products of the original speed, sensing and turning
operators and a preconditioner built from temporal harmonic blocks, the angular block
tridiagonal elimination and a low-rank density feedback; the two-dimensional Schur complement acts
on the design parameters alone. Increasing the temporal collocation from $9$ to $13$ to $17$ nodes
takes the kinetic defect measured \emph{off} the collocation grid from $1.70\times10^{-3}$ to
$5.33\times10^{-5}$ to $1.60\times10^{-6}$, alongside a collocation-point root-mean-square residual
of order $10^{-13}$. The batched vector field agrees with the original angular operator to
$8.4\times10^{-13}$ relative, and a central-difference check of the full periodic Jacobian
converges at second order.

Reference factors computed once are reused across targets: building them took $34.3$~s, after which
the Newton corrections for the three designs took $7.2$, $8.1$ and $7.5$~s, three Newton steps each,
with each target restarted from the fixed normal-form data rather than from the previous solution.

\suppnote{Restricted implementation and locked relaxation}

\begin{proposition}[Global moments relax at the matched rates]\label{prop:lock}
Every sufficiently regular solution of the kinetic equation on the periodic cell obeys
\begin{equation}\label{eq:lock}
 \dot A_{s,m}=-\gamma_mA_{s,m}
\end{equation}
exactly: transport and diffusion integrate to zero over the cell, and the turning generator is
diagonal on angular harmonics. Periodic and relative-periodic states therefore carry
$A_{s,m}=0$ for $m\ge1$, and the period map has exact left factors $e^{-\gamma_mP}$, unmoved by
phase alignment.
\end{proposition}

\begin{corollary}[Recovery ceiling]\label{cor:lock}
In any state norm that observes the moments $A_{s,m}$, $m\le M$, and over any admissible
perturbation class that moves them, a uniform asymptotic estimate $d(t)\le Ce^{-\rho t}d(0)$
requires $\rho\le\min_{m\le M}\gamma_m$.
\end{corollary}

\subsection{Proofs and data}
\label{app:locked}
\subsubsection{Proof of Proposition~\ref{prop:lock} and Corollary~\ref{cor:lock}}

\begin{proof}[Proof of Proposition~\ref{prop:lock}]
Integrate the kinetic equation against $\cos(m\theta)$ over the cell. The flux and diffusion terms
are exact spatial divergences on $\mathbb T_L$ and integrate to zero at every angle. For an even
kernel, $\int q(\phi)[\cos m(\theta-\phi)-\cos m\theta]\,d\phi=(\hat q_m-1)\cos m\theta$, so the
turning term contributes $-\alpha(1-\hat q_m)A_{s,m}=-\gamma_mA_{s,m}$, which
is~\eqref{eq:lock}. On a periodic or relative-periodic state each $A_{s,m}$ is periodic in time and
satisfies pure decay, hence vanishes for $m\ge1$. The functionals $A_{s,m}$ annihilate the spatial
and temporal tangent directions of the orbit, being translation averages, so the dual action of the
period map retains the factors $e^{-\gamma_mP}$ after the neutral phases are removed.
\end{proof}

\begin{proof}[Proof of Corollary~\ref{cor:lock}]
Take an admissible perturbation with $A_{s,m}(0)\ne0$ for some $m\le M$. Every phase shift of the
target orbit carries zero moments, so the orbital distance in the observing norm is bounded below
by a constant multiple of $|A_{s,m}(0)|e^{-\gamma_mt}$. An estimate $d(t)\le Ce^{-\rho t}d(0)$ with
$\rho>\gamma_m$ contradicts this bound as $t\to\infty$.
\end{proof}

\subsubsection{Data for the restricted designs}
\begin{table}[htbp]\centering
\caption{\textbf{The two five-primitive designs after $A=256$ recalibration.} Both share
$\gamma_1=0.1$ and $\gamma_2=0.01722734397666847$; the exact rational enclosure is
$\gamma_2^{-1}\in[58.04725332902919048,\,58.04725332902919049]$. The exact kernel specification
defines $w_4,w_5$ as printed rationals and $w_1,w_2,w_3$ by the exact moment equations; the
floating-point weights below differ from the exact ones by less than
$1.4\times10^{-14}$.}\label{tab:mixture}
\footnotesize
\begin{tabular}{lrr}
\toprule
primitive ($\delta_j$, rad) & $w_j$ (design B) & $w_j$ (design P)\\
\midrule
$0.00$ & $0.358550895672$ & $0.355324500975$\\
$0.35$ & $0.504082130368$ & $0.511706720682$\\
$0.70$ & $0.105673585767$ & $0.097835031522$\\
$1.05$ & $0.028259014132$ & $0.032760094958$\\
$1.40$ & $0.003434374062$ & $0.002373651862$\\
\bottomrule
\end{tabular}
\end{table}
\begin{table}[htbp]\centering
\caption{\textbf{Recovery and leading Floquet data for design B.} Distances are orbital distances after
spatial alignment and temporal-phase minimisation, relative to their initial values, after $140$
periods. Ritz values are the leading nontrivial multipliers of the even-sector period map with
mass and both neutral phases removed.}\label{tab:recovery}
\footnotesize
\begin{tabular}{lr}
\toprule
quantity & value\\
\midrule
\multicolumn{2}{l}{\emph{residual distance ratio after 140 periods}}\\
radial $+20\%$ & $3.46\times10^{-4}$\\
radial $-20\%$ & $2.04\times10^{-4}$\\
density $10\%$ & $1.27\times10^{-4}$\\
density $20\%$, second orbit phase & $2.93\times10^{-4}$\\
density $20\%$, $A=512$ fixed kernel & $1.05\times10^{-4}$\\
\midrule
\multicolumn{2}{l}{\emph{Floquet data}}\\
Ritz values, first grid & $0.943042;\ 0.694917\ (\times2);\ 0.420597;\ -0.302926\pm0.216935i$\\
Ritz values, $A=512$ & $0.944029;\ 0.694928$\\
dominant exponent, $A=512$ & $-0.0027264$\\
measured decay of the $20\%$ perturbation & $-0.0027100$\\
exact factor $e^{-\gamma_2P}$ & $0.69491691$\\
\bottomrule
\end{tabular}
\end{table}
\begin{table}[htbp]\centering
\caption{\textbf{The fixed-frequency probability budget at a cap of $0.002$ on any single probability.}
The linear power ceiling is $2.5380042\times10^{-5}$; targets were fixed at $70\%$, $98\%$ and
$120\%$ of it before the nonlinear solves.}\label{tab:budget}
\footnotesize
\begin{tabular}{lrrl}
\toprule
target & predicted max change & realised max change & against the cap\\
\midrule
$70\%$ & $0.001400000$ & $0.00140002495$ & within\\
$98\%$ & $0.001960000$ & $0.00196004894$ & within\\
$120\%$ & $0.002400000$ & $0.00240007341$ & beyond\\
\bottomrule
\end{tabular}
\end{table}

\suppnote{Recovery, phase response and microscopic noise}

\noindent For the turning operator, the carr\'e du champ, its action on angular harmonics, the
resulting covariance of the global moments and the phase diffusion of the orbit are
\begin{equation}\label{eq:carre}
 \Gamma_q(a,b)=2D_t\,\partial_xa\,\partial_xb
 +\alpha\!\int\! q(\varphi)\,\Delta_\varphi a\,\Delta_\varphi b\,d\varphi,
 \qquad \Delta_\varphi a=a(x,\theta{+}\varphi)-a(x,\theta),
\end{equation}

\begin{equation}\label{eq:harmnoise}
 \Gamma_q(c_m,c_n)=\tfrac12\big([\gamma_m{+}\gamma_n{-}\gamma_{m+n}]\,c_{m+n}
 +[\gamma_m{+}\gamma_n{-}\gamma_{|m-n|}]\,c_{|m-n|}\big).
\end{equation}

\begin{equation}\label{eq:noise}
 \mathrm{Cov}\big(\mathcal A^N_m(t{+}\tau),\mathcal A^N_m(t)\big)=\frac{e^{-\gamma_m|\tau|}}{2N_s},
 \qquad
 \mathrm{Var}\Big[\frac1T\!\int_0^T\!\mathcal A^N_m\,dt\Big]
 =\frac{\gamma_mT-1+e^{-\gamma_mT}}{N_s\gamma_m^2T^2},
\end{equation}

\begin{equation}\label{eq:phasevar}
  \frac{d\,\mathrm{Var}\,\phi}{dt}\simeq K_\phi,\qquad
 N_{\rm tot}K_\phi=8.562\times10^{4}
\end{equation}

\begin{proposition}[Noise reads twice as many rates]\label{prop:noise2M}
Matching $\gamma_1,\dots,\gamma_M$ fixes the angular drift on harmonics through $M$. Equality of
the covariance~\eqref{eq:harmnoise} on those harmonics for \emph{every} angular distribution holds
if and only if the rates match through $2M$. Under a uniform angular marginal only the diagonal
terms survive, which is why~\eqref{eq:noise} is class-invariant.
\end{proposition}

\subsection{Fluctuation law, phase response and the odd sector}
\label{app:noise}
\subsubsection{Proof of the fluctuation law~\eqref{eq:noise}}

\begin{proof}
The angles are autonomous and independent across particles, and the turning generator acts on
$\cos(m\theta)$ with eigenvalue $-\gamma_m$, so the stationary autocovariance of one particle's
$\cos(m\theta_i)$ is $\tfrac12e^{-\gamma_m|\tau|}$, the variance $\tfrac12$ coming from the uniform
marginal. Averaging $N_s$ independent copies divides by $N_s$, which is the covariance
in~\eqref{eq:noise}; integrating it twice over $[0,T]$ gives the time-average variance.
\end{proof}

\subsubsection{Proof of Proposition~\ref{prop:noise2M}, and the phase-noise normalisation}

\begin{proof}[Proof of Proposition~\ref{prop:noise2M}]
The identity $\Gamma(a,b)=G(ab)-aGb-bGa$ for the turning generator $G$ and the cosine product
formula give~\eqref{eq:harmnoise}, whose coefficients involve rates through $m+n\le2M$ only;
this is sufficiency. For necessity, let the first $M$ rates agree and the covariance agree for
every positive angular density. For $k=M+1,\dots,2M$ take $m=M$, $n=k-M$: since all lower rates
agree, the covariance difference is $-\tfrac12\Delta\gamma_k\cos k\theta$, and testing against
densities proportional to $1+\epsilon\cos k\theta$ forces $\Delta\gamma_k=0$. Sine identities and
linearity extend the statement to arbitrary trigonometric observables. Under a uniform marginal
the off-diagonal harmonics average to zero and only the diagonal terms
$\Gamma_q(c_m,c_m)$ contribute, which retains rates through $m$ alone.
\end{proof}

\noindent\emph{Normalisation of~\eqref{eq:phasevar}.} With particle mass $2L/N_{\rm tot}$ and
species ratio $3{:}5$, the time-displacement covariance rate is
\begin{equation}\label{eq:Kt}
 K_t=\frac{2L}{N_{\rm tot}P}\int_0^P\!\!dt\sum_s\int_0^L\!\!dx\int_0^{2\pi}\!\frac{d\theta}{2\pi}\,
 F_s\,\Gamma_q(p_s,p_s),
\end{equation}
$p_s$ the phase covector pulled back through the instantaneous spatial alignment; radian phase
multiplies by $\Omega^2$, and the convention is $\mathrm{Var}\,\phi\simeq K_\phi t$. The covector
is obtained by reverse accumulation through the Lawson Runge--Kutta step, with the transpose
checked against the action of the Jacobian to rounding error, the leading left multiplier from Arnoldi
iteration, and normalisation against the orbit tangent holding to $1.0\times10^{-7}$ over the
period; the angular generator is evaluated from the continuous primitive moments on the padded
quadrature grid.

\begin{table}[htbp]\centering
\caption{\textbf{Same-target designs and their phase-noise coefficients.} All four solve the periodic
problem at the target of design B on the reference grid; the six-primitive family adds a sixth
component and re-solves the moment constraints. The last column re-evaluates the coefficient
difference on the refined ($A=256$, $N_x=96$) orbit, where the baseline coefficient is
$9.778\times10^{4}$. ref., reference design; n.c., not computed.}\label{tab:phasenoise}
\footnotesize
\begin{tabular}{lrrrr}
\toprule
rule & smallest probability & $N_{\rm tot}K_\phi$ & vs.\ B & refined vs.\ B\\
\midrule
five-primitive B & $4.62\times10^{-3}$ & $85617.36$ & ref. & ref.\\
six-primitive $e_{0.0005}$ & $5.0\times10^{-4}$ & $85663.46$ & $+0.0539\%$ & n.c.\\
six-primitive $e_{0.001}$ & $9.15\times10^{-4}$ & $85709.59$ & $+0.1077\%$ & $+0.1298\%$\\
six-primitive sharp & $3.50\times10^{-3}$ & $86200.60$ & $+0.6812\%$ & $+0.5486\%$\\
\bottomrule
\end{tabular}
\end{table}

\subsubsection{The odd angular sector contracts}

\begin{lemma}[Odd-sector contraction]\label{lem:odd}
With orientation-independent speeds and an even kernel $q\ge q_{\min}>0$, angular reflection
commutes with the one-dimensional evolution; the even field evolves autonomously, and the odd
field obeys a linear Markov equation in the even field's velocity with
\begin{equation}
 \|F_o(t)\|_{L^1}\le e^{-\alpha\beta(t-s)}\|F_o(s)\|_{L^1},\qquad \beta=2\pi q_{\min}.
\end{equation}
\end{lemma}

\begin{proof}
The odd field carries zero density at every position, so the speeds are determined by the even
field alone. Write $q=\beta/(2\pi)+(1-\beta)\widetilde q$: uniform angular resetting annihilates
any odd function, leaving a mass-preserving, positivity-preserving time-dependent Markov generator
minus the constant $\alpha\beta$. The Markov part is an $L^1$ contraction on signed densities,
which gives the bound; distinct species use their own $\alpha\beta$.
\end{proof}

\noindent For even observables $a,b$ the operator $\Gamma_q(a,b)$ is again even, so the
conditional covariance~\eqref{eq:carre} of even observables is independent of the odd component.

\subsubsection{Particle pilot and fluctuation data}
\begin{table}[htbp]\centering
\caption{\textbf{Fluctuations at the two designs.} Upper block: time variance of the low-pass,
density-phase-aligned observable in the interacting-particle pilot, mean over three independent
seeds $\pm$ seed-to-seed standard deviation, with the deterministic value in the same gauge.
Lower block: independent-angle Monte Carlo against the exact law~\eqref{eq:noise}.}\label{tab:noise}
\footnotesize
\begin{tabular}{lrrr}
\toprule
kernel & particles & time variance & deterministic value\\
\midrule
B & $16{,}000$ & $0.04740\pm0.00379$ & $4.7124\times10^{-4}$\\
B & $64{,}000$ & $0.01885\pm0.00234$ & $4.7124\times10^{-4}$\\
P & $16{,}000$ & $0.04827\pm0.00757$ & $5.6565\times10^{-4}$\\
P & $64{,}000$ & $0.01999\pm0.00032$ & $5.6565\times10^{-4}$\\
\midrule
\multicolumn{4}{l}{independent-angle ensembles, both kernels, $N_s=1000$ and $4000$, lags $0,10,50$:}\\
\multicolumn{4}{l}{\quad largest deviation from Eq.~\eqref{eq:noise} $=1.62$ estimated standard errors}\\
\bottomrule
\end{tabular}
\end{table}

\suppnote{Constraint sweep of the kinetic design space}

\noindent In a 33-harmonic kernel family at the critical Hopf reference,
matching $M=2,\ldots,10$ angular relaxation rates leaves the local response capacity at
$\Gamma_b=2$ while the region of responses reachable to first order within unit budget
$\|h\|_\infty\le1$ shrinks by a factor of two to three per matched rate; a fixed box of targets lies
within that budget up to $M=6$ and outside it from $M=8$, and 31 full periodic kinetic solves of fixed-kernel designs return the
target power within 0.93\%.

\subsection{Setting}
At the critical Hopf reference of the main text ($\kappa_q=5.5122$, $N=A=128$) the kernel is
$q=q_0(1+h)$, $h(\varphi)=\sum_{j<H}c_j\cos j\varphi$. The complex eigenvalue response
$\delta\lambda$ along each direction of the subspace that preserves normalisation, $\gamma_1$ and
$\gamma_2$, including self-consistent reshaping of the interface, was computed with the linearised
kinetic solver; for $H=9$ it reproduces the delivered responses to $10^{-12}$. Matching
$\gamma_3,\ldots,\gamma_M$ adds linear constraints from the exact Bessel moments of
$q_0\cos j\varphi$. For each $M$ we computed the hidden dimension; the capacity
$\Gamma_b=\operatorname{rank}(\delta\operatorname{Re}\lambda,\delta\operatorname{Im}\lambda)$ on
the hidden subspace; the region of $(\delta\operatorname{Re}\lambda,\delta\operatorname{Im}\lambda)$
reachable with $\|h\|_\infty\le1$, by linear programming of its support function; its growth
half-range $D_M^{(H)}$, the prediction loss $D_M$ of the main text restricted to the $H$-harmonic
family and a lower bound on it; and, for targets of oscillation power $I$ and frequency $\Omega$,
the smallest budget $\|h\|_\infty$ of a first-order design. A budget below one guarantees a positive kernel; a larger budget does not by itself imply a
negative one. The map from $(\delta\operatorname{Re}\lambda,\delta\operatorname{Im}\lambda)$ to
$(I,\Omega)$ is the fixed linear map of the Hopf normal form (Methods). The $D_M$ of Fig.~2c of the
main text is computed at the working point from the response table truncated at $m=8$ and spans
$M=0,\ldots,4$; the $D_M^{(33)}$ here is the same functional restricted to the 33-harmonic family
at the Hopf reference and spans $M=2,\ldots,10$. The two agree in trend and are not expected to
agree in value.

\subsection{Results}
Fig.~4c--e of the main text and Supplementary Table~\ref{tab:nrqs} use $H=33$. Each matched rate removes one hidden
direction, but $\Gamma_b=2$ for every $M$ from 2 to 10. The reachable area falls from
$2.5\times10^{-3}$ at $M=2$ to $6.0\times10^{-4}$, $7.4\times10^{-5}$ and $1.2\times10^{-5}$ at
$M=4$, 6 and 8, and $D_M^{(33)}$ from $5.3\times10^{-2}$ to $1.1\times10^{-3}$ at $M=8$, after
which it levels off. The cost of fixed targets rises: for the targets B, P and F the smallest
budget grows from 0.026--0.034 at $M=2$ to 0.49--0.59 at $M=6$ and exceeds one at $M=7$. Of 400
random targets in the box $I\in[4,7]\times10^{-4}$, $\Omega\in[0.2972,0.2978]$, first-order designs within
$\|h\|_\infty<1$ reach all at $M\le6$, 24.5\% at $M=7$ and none at $M\ge8$.

\begin{table}[h]\centering\footnotesize
\caption{\textbf{Design space against the number $M$ of matched angular relaxation rates, 33-harmonic
family.} $D_M^{(33)}$ and area at unit budget; ``reachable'' is the fraction of 400 targets in the
box reached with budget below one.}\label{tab:nrqs}
\TabNRQS
\end{table}

\subsection{Truncation}
Repeating the analysis with $H=9$, 13, 17, 25 and 33 (Supplementary Table~\ref{tab:truncH}; Extended Data
Fig.~3 of the main text) shows that the limit found with nine harmonics, where 1.5\% of the
targets are reachable at $M=5$, came from the truncation. With more harmonics the unit-budget
boundary moves out and settles: $D_M^{(H)}$ and the budgets for $H=25$ and 33 agree closely up to
$M=8$, and for both families the whole box is reachable at $M=6$ and none of it at $M=8$.
$\Gamma_b=2$ holds for every $H\ge13$ and $M\le10$; the value $\Gamma_b=1$ seen with nine
harmonics at $M=7$ reflected a single remaining hidden direction.

\begin{table}[h]\centering\footnotesize
\caption{\textbf{Truncation test: fraction of the target box reachable within $\|h\|_\infty<1$ for kernel
families with $H$ harmonics.} A dash marks $\Gamma_b<2$.}\label{tab:truncH}
\TabTrunc
\end{table}

\subsection{Full kinetic validation and held-out design}
Each design was tested by the complete one-cell periodic kinetic boundary-value problem with the
kernel fixed, so that power and frequency are outputs (17 temporal nodes, $N=A=128$;
Supplementary Table~\ref{tab:kin}). Three groups were solved: the minimum-budget designs for B, P and F at
$M=2$, 3 and 4 in the nine-harmonic family; eight held-out targets, drawn at random in the box before any kinetic solve (a target was kept if its nine-harmonic $M=4$ design
has budget below 0.7), designed at $M=3$ and 4; and B, P and F at $M=5$ and 6 in the 33-harmonic
family, which the nine-harmonic family cannot reach. All 31 solves converged and match
$\gamma_1,\ldots,\gamma_M$ to $10^{-16}$. For the nine-harmonic designs the realised power exceeds
the target by 0.24--0.44\% and the frequency lies $5$--$9\times10^{-6}$ below it; the error grows
slowly with the budget, and at $M=2$ it equals that of the earlier B, P and F designs, so it is
the residual of the first-order formula. The 33-harmonic designs at $M=5$ and 6 return the power
within 0.93\% and the frequency within $3.7\times10^{-5}$; their larger error reflects the finer
angular structure of these kernels. Supplementary Note 9 removes this residual by solving the
orbit and the kernel coordinates together.

\begin{table}[h]\centering\footnotesize
\caption{\textbf{Full periodic kinetic solves of fixed-kernel designs.} First block: B, P, F, nine
harmonics; second block: B, P, F, 33 harmonics; third block: held-out targets drawn at random before any kinetic solve, nine harmonics.}\label{tab:kin}
\TabKinetic
\end{table}

\suppnote{Joint inverse design and independent forward integration}

\noindent With the normalisation, $\gamma_1$ and $\gamma_2$ of the reference
kernel imposed exactly, the periodic orbit of the full kinetic equations and two hidden kernel
coordinates can be solved together for a prescribed oscillation power and frequency; the resulting
kernels, held fixed in an independent forward integration started from the steady interface of the
reference kernel, lead the system to the prescribed state.

\subsection{Formulation}
The kernel is $q=q_0(1+h)$ with $h=\sum_{j<33}c_j\cos j\varphi$ and $c$ in the subspace that
preserves the normalisation and the first two angular rates, so that the first three entries of
$\gamma(q)-\gamma(q_0)$ vanish identically; for the design that also preserves $\gamma_3$ and
$\gamma_4$ the subspace is cut by two further Bessel-moment constraints. For a target
$(I^*,\Omega^*)$ the linear design is the kernel direction of smallest budget $\|h\|_\infty$
whose first-order response, mapped through the Hopf normal form (Methods), returns the
target; the perpendicular direction is the direction of equal budget whose response is orthogonal
to it. With $U=[c_{\rm lin},c_{\perp}]$ the hidden coordinates are $a=(a_1,a_2)$ and
$c=Ua$; the linear design is $a=(1,0)$.

The joint problem is the periodic boundary-value problem of the one-cell, angle-even kinetic
equations of the main text, represented by temporal Fourier collocation on 17 nodes at
$N=A=128$, augmented by three scalar equations: the phase condition, $I(\text{orbit})=I^*$ and
$\Omega=\Omega^*$, and by the two unknowns $a$. It is solved by Newton iteration with a Schur
complement on $a$; each step reuses the harmonic factorisation of the reference operator. The
Newton residual, the kinetic residual, the phase residual and the relative power residual are
recorded at each iteration.

\subsection{Results}
Seven targets were solved (Supplementary Table~\ref{tab:joint}). All converged from the linear design in three
or four iterations; the final residual norms are $5$--$6\times10^{-13}$ for the four designs of
the main text and $10^{-12}$--$4\times10^{-11}$ for the three boundary targets, and the matched
moments hold to $10^{-18}$. The nonlinear coordinates differ from the linear ones by
$|a-(1,0)|=0.016$, $0.003$, $0.003$ and $0.005$ for m1, m2, m3 and m2$'$; the correction is
along the perpendicular direction for m2, m3 and m2$'$ and mainly a shortening of the linear
direction for m1, whose target lies on the low-frequency side of the reference.

Four kernels were then held fixed in a forward integration of 8000 time units with 1024 steps
per target period, started from the steady interface of the reference kernel m0 plus
$0.1\sqrt{5\times10^{-4}/2Q_2}$ times the critical eigenvector; nothing about the target enters
the integration. The reference kernel m0, a steady-side design with the same moments and a budget
of 0.029, keeps the band steady: the critical-mode amplitude decays from 0.58 to 0.11 over 4000
time units and the temporal standard deviation of the density over the final window is
$1.3\times10^{-5}$. Under m1, m2, m3 and m2$'$ the amplitude grows and saturates within about
2000 time units. A six-harmonic fit over the last 16 periods returns the frequency and power in
Supplementary Table~\ref{tab:joint}; the direct density variance over the same window is within 0.03\% of the
target. Sixteen further cycles started on each solved orbit return relative orbit-closure defects
of $5\times10^{-7}$ (m1, m2, m2$'$) and $3\times10^{-6}$ (m3). Floquet multipliers from Arnoldi
iteration on the period map (512 steps per period, eight values) have leading moduli 0.9435,
0.9423, 0.8845 and 0.9422 with residuals below $5\times10^{-15}$ and neutral-mode relative errors
of $2\times10^{-6}$ to $1.5\times10^{-5}$.

\subsection{Towards the amplitude boundary}
Three further targets raise the power at $\Omega^*=0.2975$ to $2$, $3$ and $4\times10^{-3}$. The
joint solve converges for each, with the Newton safety cap on the budget raised from 0.25 to the unit budget, below which positivity
is guaranteed. The budget grows to 0.204, 0.330 and 0.457, so the kernel minimum falls to 0.80, 0.67 and
0.54 of $q_0$; the nonlinear correction grows to $|a-(1,0)|=0.007$, $0.011$ and $0.016$; and the
orbit-closure defect after sixteen cycles grows to $3.9\times10^{-5}$, $1.7\times10^{-4}$ and
$4.3\times10^{-4}$. For these three targets the independent integration from the perturbed
steady interface was not run. For the largest target the Floquet calculation returned a
neutral-mode error of $1.9\times10^{-3}$ at both 512 and 1024 steps per period, above the
$10^{-3}$ tolerance of the variational integrator, and was rejected; the stability of that orbit
is not established.

\begin{table}[h]\centering\footnotesize
\caption{\textbf{Joint inverse designs.} Budget is $\|h\|_\infty$; $|a-(1,0)|$ is the distance of the
nonlinear hidden coordinates from the linear design; ``forward'' columns are the harmonic fit over
the last 16 periods of the independent integration from the perturbed steady interface (m1--m2$'$)
or of sixteen cycles from the solved orbit (b2--b4); closure is the relative orbit-closure defect
after sixteen cycles; the last column is the leading nontrivial Floquet modulus (512 steps per period).
n.c., not computed; n.d., not determined: for b4 the variational integration failed its neutral-mode
check ($1.9\times10^{-3}$ at 512 and 1024 steps per period, tolerance $10^{-3}$), so no multiplier is
reported.}\label{tab:joint}
\footnotesize\setlength{\tabcolsep}{3pt}
\begin{tabular}{lrlllllllrr}
\toprule
design & $M$ & $I^*$ & $\Omega^*$ & budget & $\min q/q_0$ & $|a-(1,0)|$ & forward $\Omega$ & forward $I$ & closure & Floquet\\
\midrule
m1 & 2 & $5\times10^{-4}$ & 0.2965 & 0.067 & 0.933 & 0.016 & 0.29650000 & $5.0000\times10^{-4}$ & $5.0\times10^{-7}$ & 0.9435\\
m2 & 2 & $5\times10^{-4}$ & 0.2985 & 0.069 & 0.931 & 0.003 & 0.29850000 & $5.0000\times10^{-4}$ & $4.8\times10^{-7}$ & 0.9423\\
m3 & 2 & $1\times10^{-3}$ & 0.2975 & 0.081 & 0.919 & 0.003 & 0.29750000 & $1.0000\times10^{-3}$ & $2.9\times10^{-6}$ & 0.8845\\
m2$'$ & 4 & $5\times10^{-4}$ & 0.2985 & 0.086 & 0.914 & 0.005 & 0.29850000 & $5.0000\times10^{-4}$ & $4.8\times10^{-7}$ & 0.9422\\
b2 & 2 & $2\times10^{-3}$ & 0.2975 & 0.204 & 0.796 & 0.007 & 0.29749999 & $2.0000\times10^{-3}$ & $3.9\times10^{-5}$ & n.c.\\
b3 & 2 & $3\times10^{-3}$ & 0.2975 & 0.330 & 0.670 & 0.011 & 0.29749997 & $3.0000\times10^{-3}$ & $1.7\times10^{-4}$ & n.c.\\
b4 & 2 & $4\times10^{-3}$ & 0.2975 & 0.457 & 0.543 & 0.016 & 0.29749992 & $4.0000\times10^{-3}$ & $4.3\times10^{-4}$ & n.d.\\
\bottomrule
\end{tabular}
\end{table}

\suppnote{A reaction--diffusion fibre}

\subsection{The family and its coarse data}\label{si:rd}
For $w=(U,V)^{\mathsf T}$ on a periodic domain in the plane consider
\begin{equation}\label{eq:rdfamily}
 \partial_t w=Jw+D\nabla^2w+N_\theta(w),\qquad
 J=\begin{pmatrix}0.8&-1\\1&-1\end{pmatrix},\qquad D=\operatorname{diag}(1,3.5),
\end{equation}
\begin{equation}\label{eq:rdnonlin}
 N_\theta(U,V)=\big(\eta_2U^2+\eta_{11}UV-\beta_3U^3-\beta_5U^5,\ 0\big)^{\mathsf T},\qquad
 \theta=(\eta_2,\eta_{11},\beta_3,\beta_5).
\end{equation}
Every member satisfies $N_\theta(0)=0$ and $DN_\theta(0)=0$, so the uniform state, the
linearisation $Jw+D\nabla^2w$ and the dispersion relation
\begin{equation}
 \lambda_\pm(k)=\operatorname{eig}\big(J-k^2D\big),\qquad
 \det(J-k^2D)=3.5k^4-1.8k^2+0.2,
\end{equation}
are the same for every $\theta$ and at every wavenumber. We take the coarse map to be
$\mathcal C(\theta)=(J,D)$, the data a linear-stability experiment returns; its fibre is the whole
parameter space $\{\theta\}$. The uniform state is stable at $k=0$
($\operatorname{tr}J=-0.2$, $\det J=0.2$) and Turing-unstable in the band
$k^2\in(0.16238,0.35190)$, with $k_c^2=1.8/7$, $k_c=0.507093$ and $\sigma=\lambda_+(k_c)=0.022776$.

\subsection{Restricted response: the three-mode coefficients}
Let $r$ and $\ell$ be the right and left critical eigenvectors of $L(k_c)=J-k_c^2D$, normalised by
$r_U=1$, let $B_\theta(x,y)=(\eta_2x_Uy_U+\tfrac12\eta_{11}(x_Uy_V+x_Vy_U),0)^{\mathsf T}$ be the
symmetric quadratic form and $C_\theta(x,y,z)=(-\beta_3x_Uy_Uz_U,0)^{\mathsf T}$ the cubic form of
$N_\theta$, and write $L(k)=J-k^2D$. For three critical modes at $120^\circ$,
$w=\sum_jA_jre^{\mathrm ik_j\cdot x}+\mathrm{c.c.}+\cdots$, the slaved second-order fields are
\begin{align}
 w_0&=-L(0)^{-1}\,2B_\theta(r,r), &
 w_2&=-\big(L(2k_c)-2\sigma\big)^{-1}B_\theta(r,r),\\
 w_-&=-\big(L(\sqrt3k_c)-2\sigma\big)^{-1}2B_\theta(r,r), &
 w_+&=-\big(L(k_c)-2\sigma\big)^{-1}\big(I-P\big)2B_\theta(r,r),
\end{align}
The mean mode is evaluated at zero shift and the harmonics at $2\sigma$; this convention is used for
every coefficient and for the construction of the twins. It is not a consistent expansion in
$\sigma$. The two consistent choices, dropping $2\sigma$ everywhere (the threshold limit) or keeping it
everywhere, change the coefficients of individual reaction laws at $\sigma=0.0228$. Twins that share
$(a,g,h)$ exactly in the convention used here differ within each target by at most $3.4\%$ in $g$ and
$4.7\%$ in $h$ in the threshold limit, and by up to $25.5\%$ in $g$ and $10.6\%$ in $h$ when $2\sigma$
is kept in the mean mode as well (target T3 in both cases). The leading-order equivalence of the
twins is therefore defined up to these convention-dependent differences.

with $P=r\ell^{\mathsf T}/\ell^{\mathsf T}r$ the critical projection, and the amplitude equations read
\begin{equation}\label{eq:rdamp}
 \dot A_1=\sigma A_1+a\,\bar A_2\bar A_3-g|A_1|^2A_1-h\big(|A_2|^2+|A_3|^2\big)A_1
\end{equation}
and cyclically, with
\begin{align}
 a&=\frac{2\,\ell^{\mathsf T}B_\theta(r,r)}{\ell^{\mathsf T}r},\qquad
 g=-\frac{\ell^{\mathsf T}\big[2B_\theta(r,w_0)+2B_\theta(r,w_2)+3C_\theta(r,r,r)\big]}{\ell^{\mathsf T}r},\\
 h&=-\frac{\ell^{\mathsf T}\big[2B_\theta(r,w_0)+2B_\theta(r,w_-)+2B_\theta(r,w_+)+6C_\theta(r,r,r)\big]}{\ell^{\mathsf T}r}.
\end{align}
The quadratic coefficient $a$ is linear in $(\eta_2,\eta_{11})$ and independent of $\beta_3$; $g$ and
$h$ are affine in $\beta_3$ and quadratic in $(\eta_2,\eta_{11})$. At $\eta_2=\eta_{11}=0$ they
reduce to $g=3\beta_3\ell_Ur_U^3/\ell^{\mathsf T}r$ and $h=2g$ ($g=4.1123$ at $\beta_3=1$).

\subsection{Outputs and capacity}
Stripes have amplitude $|A|=\sqrt{\sigma/g}$; hexagons have
$|A|=\big(|a|+\sqrt{a^2+4\sigma(g+2h)}\big)/\big(2(g+2h)\big)$, with spots or holes according to the
sign of $a\,r_U$. For $h>g$, stripes are linearly stable to the oblique pair when
$\sigma(1-h/g)+|a|\sqrt{\sigma/g}<0$, and hexagons exist stably when $\sigma<a^2(2g+h)/(h-g)^2$.
The restricted response of $(a,g)$ to the hidden coordinates $(\eta_2,\beta_3)$ at
$\eta_2=\eta_{11}=0$, $\beta_3=1$ is
\begin{equation}
 \frac{\partial(a,g)}{\partial(\eta_2,\beta_3)}=\begin{pmatrix}2.7415&0\\0&4.1123\end{pmatrix},
\end{equation}
so the local response capacity for the pair (morphology, amplitude) is $\Gamma_b=2$.

\subsection{The response-active quotient is finite-dimensional}
The fibre of $(J,D)$ is infinite-dimensional: it contains every nonlinearity $N$ with $N(0)=0$ and
$DN(0)=0$. Through third order in the amplitude, however, the pattern-forming dynamics read $N$
only through the three numbers $(a,g,h)$ of Eq.~\eqref{eq:rdamp}, so the response-active quotient of
the fibre at this order has dimension at most three, whatever the number of hidden parameters. The
Jacobian of $(a,g,h)$ with respect to $(\eta_2,\eta_{11},\beta_3,\beta_5)$ is
\begin{equation}
 \frac{\partial(a,g,h)}{\partial(\eta_2,\eta_{11},\beta_3,\beta_5)}\Big|_{\eta_2=0.09}=
 \begin{pmatrix}2.742&1.426&0&0\\-5.943&-3.606&4.112&0\\-8.154&-4.523&8.225&0\end{pmatrix},
\end{equation}
with singular values $15.0$, $2.28$ and $0.25$. The column of $\beta_5$ vanishes identically: the
quintic saturation is a silent direction at this order, entering only at relative order $A^2$,
$\delta A/A=-\tfrac12\,c_5A^2/g$ with $c_5=10\,\beta_5\,\ell_U/\ell^{\mathsf T}r=13.7\beta_5$.
At $\eta_2=\eta_{11}=0$ the rank drops to two, because $h\equiv2g$ there and $g,h$ respond to
$\beta_3$ in fixed proportion. The third singular value is sixty times smaller than the first, so
the effective capacity is two. Within the bistable window $0.06\le\eta_2\le0.12$ the image of
$(\eta_2,\beta_3)\mapsto(A_{\rm stripe},A_{\rm hexagon})$ is a narrow wedge, $A_{\rm hexagon}/A_{\rm stripe}\in[0.475,0.646]$
(at $A_{\rm stripe}=0.060$: $[0.491,0.536]$), because $h/g$ stays close to two: the two amplitudes are
independently prescribable only inside the wedge.

\subsection{Inverse design}
Given target amplitudes $(A^*_{\rm stripe},A^*_{\rm hexagon})$ in the wedge, the closed-form step
solves $\sqrt{\sigma/g}=A^*_{\rm stripe}$ and the hexagon relation for $(\eta_2,\beta_3)$. A
one-step correction replaces the targets by $A^*+\big(A^{\rm theory}(\theta_0)-A^{\rm PDE}(\theta_0)\big)$
and solves again, which absorbs the systematic difference between three-mode theory and simulation at
the first design. Supplementary Table~\ref{tab:rdinv} lists the outcomes.

\subsection{Numerical method}
Equation~\eqref{eq:rdfamily} is integrated pseudo-spectrally with the exponential time-differencing
fourth-order Runge--Kutta scheme\cite{hochbruck}, applying the exponential to diffusion alone and
treating the whole reaction explicitly; placing the diagonal of $J$ in the exponential and its
off-diagonal part in the explicit stage produces a spurious low-wavenumber growth at $\Delta t=0.5$.
The domain $L_x=8\pi/k_c$, $L_y=16\pi/(\sqrt3k_c)$ carries the hexagonal triad at $k_c$ exactly;
$64\times74$ modes with two-thirds dealiasing, $\Delta t=0.5$. The linear growth rate of the
lattice mode at $k_c$ is reproduced as $0.022776$.

\subsection{Predictions against simulation}
\begin{table}[htbp]\centering
\caption{\textbf{Reaction--diffusion family~\eqref{eq:rdfamily}: three-mode predictions against full
two-dimensional simulations ($\beta_5=0.5$, $\eta_{11}=0$).} Amplitudes are single-mode $|A|$ of $U$.
Oblique growth rates are measured on a converged stripe.}\label{tab:rd}
\footnotesize\setlength{\tabcolsep}{3pt}
\resizebox{\linewidth}{!}{%
\begin{tabular}{llll}
\toprule
quantity & setting & prediction & simulation\\
\midrule
dispersion difference & all members, $0\le k\le1.2$ & $0$ & $\le6.6\times10^{-13}$\\
linear growth at $k_c$ & lattice mode & $0.022776$ & $0.022776$\\
stripe amplitude & $\eta_2=0,\ \beta_3=1$ & $0.0744$ & $0.0740$\\
stripe amplitude & $\eta_2=0,\ \beta_3=3$ & $0.0430$ & $0.0429$\\
morphology, polarity & $\eta_2=\pm0.3$ & hexagons, spots/holes & hexagons, skew $+0.86/-0.85$\\
hexagon amplitude & $\eta_2=0.07,\ 0.09,\ 0.15,\ 0.30$ & $0.0389,\ 0.0410,\ 0.0492,\ 0.1093$ & $0.0384,\ 0.0398,\ 0.0445,\ 0.0597$\\
hexagons stable & along $\eta_2$ & $|\eta_2|>0.0559$ & decay at $0.045$, persist at $0.07$\\
stripes stable & along $\eta_2$ & $|\eta_2|<0.1134$ & persist at $0.150$, lost at $0.175$\\
oblique growth rate & $\eta_2=0,\ 0.02,\ 0.04$ & $-0.0228,\ -0.0187,\ -0.0148$ & $-0.0215,\ -0.0186,\ -0.0157$\\
oblique growth rate & $\eta_2=0.07,\ 0.15$ & $-0.0088,\ +0.0079$ & $-0.0114,\ -0.0004$\\
\bottomrule
\end{tabular}}
\end{table}

\begin{table}[htbp]\centering
\caption{\textbf{Inverse design of $(A_{\rm stripe},A_{\rm hexagon})$ in the reaction--diffusion fibre
($\eta_{11}=0$, $\beta_5=0.5$).} ``Closed'' is the solution of the three-mode relations; ``corrected'' shifts
the targets by the measured theory--simulation difference at the closed-form design and solves again.
Stripe- and hexagon-seeded runs at each kernel both keep their pattern.}\label{tab:rdinv}
\footnotesize
\begin{tabular}{llrrrrrr}
\toprule
design & target & stage & $\eta_2$ & $\beta_3$ & $x$ & realised & error\\
\midrule
D1 & $(0.060,\,0.031)$ & closed & $0.09383$ & $1.6091$ & $0.68$ & $(0.0595,\,0.0307)$ & $(-0.8,\,-1.1)\%$\\
   &                    & corrected & $0.09519$ & $1.5874$ & $0.69$ & $(0.0600,\,0.0310)$ & $(0.0,\,0.0)\%$\\
D2 & $(0.090,\,0.050)$ & closed & $0.09744$ & $0.7600$ & $1.06$ & $(0.0879,\,0.0472)$ & $(-2.3,\,-5.7)\%$\\
   &                    & corrected & $0.11170$ & $0.7536$ & $1.24$ & $(0.0894,\,0.0488)$ & $(-0.7,\,-2.4)\%$\\
D3 & $(0.060,\,0.034)$ & closed & $0.16135$ & $1.7475$ & $1.17$ & $(0.0589,\,0.0322)$ & $(-1.9,\,-5.3)\%$\\
   &                    & corrected & $0.18441$ & $1.7555$ & $1.36$ & $(0.0596,\,0.0332)$ & $(-0.7,\,-2.4)\%$\\
\bottomrule
\end{tabular}
\end{table}

\begin{table}[htbp]\centering
\caption{\textbf{Stripe loss and hexagon-amplitude error along the closed-form design curve
$A_{\rm stripe}=0.060$, parametrised by $x=a/\sqrt{\sigma g}$.} n.a., stripes lost, so no stripe
amplitude.}\label{tab:rdx}
\footnotesize
\begin{tabular}{rrrlrlr}
\toprule
$x$ & $\eta_2$ & $\beta_3$ & stripe-seeded & error & hexagon-seeded & error\\
\midrule
0.90 & 0.1246 & 1.663 & stripe & $-1.2\%$ & hexagon & $-2.9\%$\\
1.26 & 0.1745 & 1.783 & stripe & $-2.2\%$ & hexagon & $-6.2\%$\\
1.40 & 0.1938 & 1.840 & stripe & $-2.6\%$ & hexagon & $-7.7\%$\\
1.60 & 0.2215 & 1.933 & hexagon (stripes lost) & n.a. & hexagon & $-9.8\%$\\
1.90 & 0.2631 & 2.094 & hexagon (stripes lost) & n.a. & hexagon & $-13.2\%$\\
2.40 & 0.3323 & 2.425 & hexagon (stripes lost) & n.a. & hexagon & $-19.0\%$\\
\bottomrule
\end{tabular}
\end{table}

\begin{figure}[htbp]\centering
\includegraphics[width=.55\linewidth]{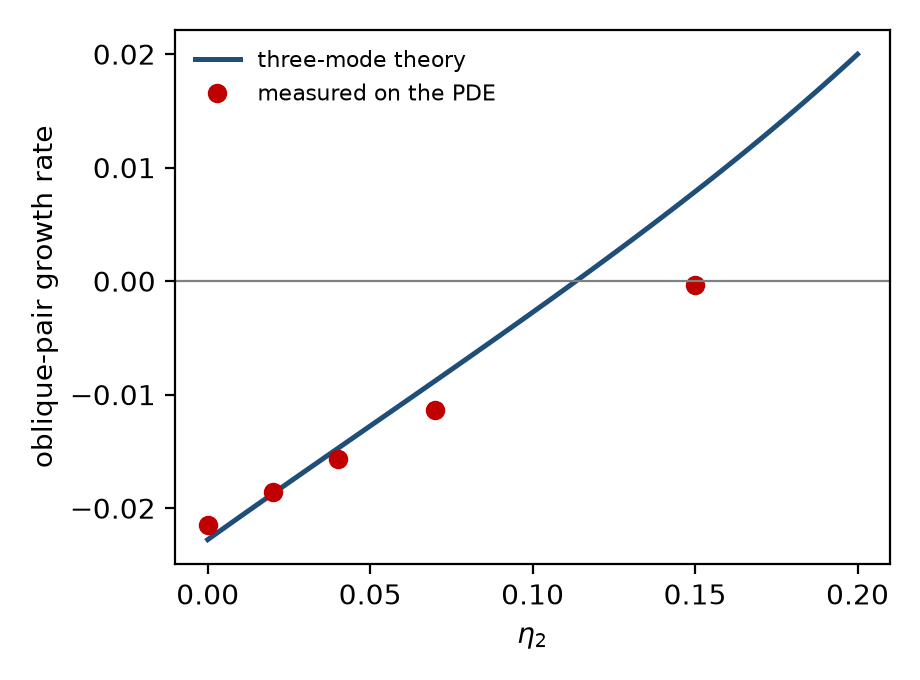}
\caption{Growth rate of the oblique mode pair on a converged stripe against $\eta_2$: three-mode theory
and direct measurement on the PDE. At $\eta_2=0$ the measured $-0.0215$ compares with the leading-order
$-0.0228$; the theory locates the stripe-stability boundary at $x\approx1.0$ against the observed
$1.48$--$1.60$.}\label{fig:rdoblique}
\end{figure}

\suppnote{The response-active quotient of the reaction--diffusion fibre}

\noindent In the family of Eq.~\eqref{eq:rdfamily} with an arbitrary polynomial
nonlinearity $N_\theta$ vanishing to second order at the origin, the map $\theta\mapsto(a,g,h)$
from the hidden coefficients to the three-mode amplitude-equation coefficients has rank three at
every sampled kernel, all monomials of degree four and higher lie in its null space, and kernels
of unrelated structure with equal $(a,g,h)$ reach the same pattern with amplitudes that differ
among themselves by about one per cent.

\subsection{Rank of the map to the quotient}
The nonlinearity may contain any monomial $U^iV^j$ with $i+j\ge2$ in either equation, so
$N_\theta(0)=0$, $DN_\theta(0)=0$, and $(J,D)$ and the dispersion relation are identical for every
member (Fig.~3a of the main text). Near the Turing threshold the three-mode amplitude equations
read $\theta$ only through $\Phi(\theta)=(a,g,h)$, the resonant quadratic coefficient and the two
cubic saturation coefficients, which are fixed by the symmetric bilinear and trilinear forms of the
quadratic and cubic parts of $N_\theta$ at the critical eigenvector (Supplementary Note 10).
Monomials of degree four and higher do not enter $\Phi$.

The Jacobian $\partial\Phi/\partial\theta$, evaluated by central differences at random kernels of
three families, has rank three at every point (Supplementary Table~\ref{tab:rank}; Fig.~3b and Extended Data Fig.~4a,b of the main text).
For the family with 22 coefficients, all quadratic and cubic monomials and $U^4$, $V^4$, $U^5$,
$V^5$ in both equations, the null space is 19-dimensional: the eight quartic and quintic columns
vanish identically and eleven combinations of quadratic and cubic coefficients leave $(a,g,h)$
unchanged. The third singular value is about two orders of magnitude below the first; its direction
changes $h/g$, which the four-coefficient family of Extended Data Fig.~1 barely moves.

\begin{table}[h]\centering\footnotesize
\caption{\textbf{Rank of the map from hidden nonlinear coefficients to $(a,g,h)$.} Singular-value ratios
are medians over the sampled kernels.}\label{tab:rank}
\TabRank
\end{table}

\subsection{A silent direction}
The quintic coefficient $\beta_5$ changes the kernel but not $(a,g,h)$. In full two-dimensional
simulations ($64\times74$ modes, $T=6000$, domain commensurate with the hexagonal triad) at two
points of the main-text family, raising $\beta_5$ from 0 to 12 keeps the pattern of every seed; in
the bistable window ($\eta_2=0.09$) stripes and hexagons both persist at every $\beta_5$
(Supplementary Table~\ref{tab:silentq}). The amplitude drifts by at most 14\%, the size
of the $O(A^2)$ correction a quintic term produces, and at $\beta_5=0$ the simulated stripe
amplitude equals the leading-order value $\sqrt{\sigma/g}=0.0744$. The active coefficient
$\beta_3$, varied over a factor of ten, moves the stripe amplitude from 0.103 to 0.033 along
$\beta_3^{-1/2}$.

\begin{table}[h]\centering\footnotesize
\caption{\textbf{Final pattern and amplitude against the silent coefficient $\beta_5$ ($\eta_{11}=0$,
$\beta_3=1$); seed 1.} Seeds 2 and 3 give the same patterns and amplitudes to a relative
$2\times10^{-11}$.}\label{tab:silentq}
\TabSilent
\end{table}

\subsection{Twin reaction laws}
If the quotient controls the collective state, kernels with equal $(a,g,h)$ must reach the same
state however different their coefficients are. Six target triples, specified by the stripe
amplitude $A_{\rm s}=\sqrt{\sigma/g}$, the ratio $x=a/\sqrt{\sigma g}$ and $h/g$, cover the
stripe-only, bistable and hexagon-only regions. For each, five kernels were drawn at random in the
16-coefficient space and projected onto $\Phi^{-1}(a,g,h)$ by minimum-norm Gauss--Newton steps,
keeping every coefficient below two in magnitude, and run from stripe, hexagon and noise initial
conditions (90 runs). For T1--T4, three further kernels have nonlinearity only in the activator
equation, only in the inhibitor equation, or in every monomial of both.

In 88 of the 90 runs the final pattern type is the predicted one. The two exceptions are
hexagon-seeded runs of one reaction law each in targets T1 and T5, where only stripes are predicted to
be stable; they remain hexagonal. Supplementary Table~\ref{tab:twins} separates the spread
among twins from the offset of their mean from the leading-order amplitude. In ten of twelve
target--seed groups the twins differ by 0.1--1.8\% (standard deviation), while the group mean can
lie up to 15.5\% below the leading-order value (T4, $x=2$). The error of the amplitude equations is common to a group, and kernels are separated by
$(a,g,h)$. For target T2, eight kernels of four structures give stripes of amplitude
0.0584--0.0597 (prediction 0.0600) and hexagons of amplitude 0.0300--0.0322 (prediction 0.0314)
(Fig.~3c and Extended Data Fig.~4c of the main text).

\begin{table}[h]\centering\footnotesize
\caption{\textbf{Twins with coefficients bounded by two: spread among the kernels of each target and
offset of their mean from the leading-order amplitude.} ``In main state'' counts kernels within 20\%
of the group median; the remaining two T3 hexagon runs are the larger branch discussed
below.}\label{tab:twins}
\TabTwins
\end{table}

\subsection{Seed robustness}
Every T2 and T4 kernel (16 kernels) and every point of the $\beta_5$ scan were rerun with two
further random seeds for the $10^{-3}$ perturbation added to the initial condition (160 runs;
Supplementary Table~\ref{tab:seeds}). With stripe or hexagon seeds, all 32 kernel--condition pairs and all 32
$\beta_5$ runs reach the same pattern for all three seeds, and the amplitudes agree to a relative
$10^{-5}$. From noise, five of the eight bistable T2 kernels reach stripes for two seeds and
hexagons for the third, which is the expected behaviour in a bistable window; the T4 kernels, for
which only hexagons are stable, reach hexagons for every seed.

\begin{table}[h]\centering\footnotesize
\caption{\textbf{Seed robustness of the twins (T2 and T4, 16 kernels) under three random
seeds.}}\label{tab:seeds}
\TabSeeds
\end{table}

\begin{figure}[h]\centering
\includegraphics[width=\linewidth]{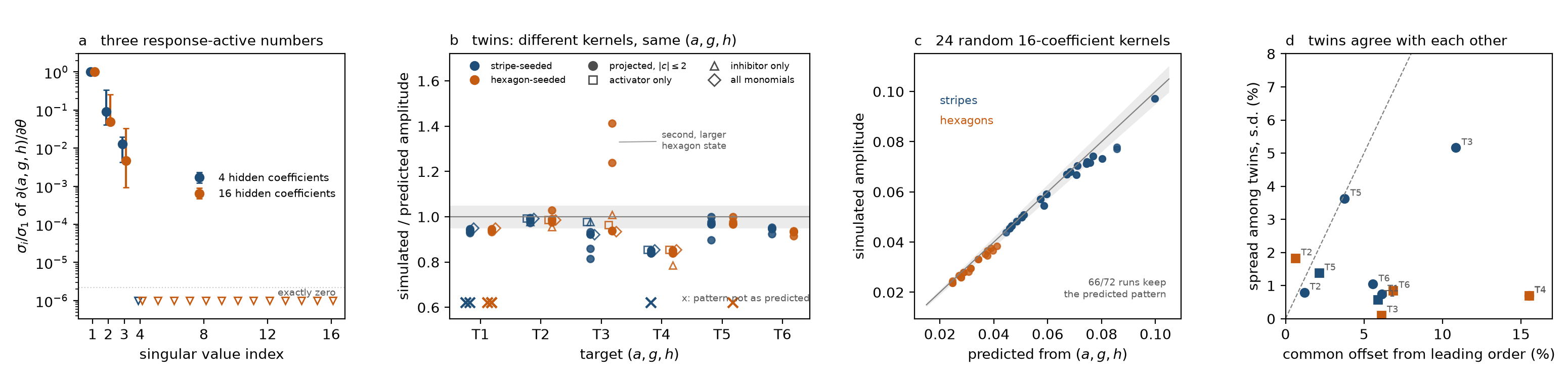}
\caption{\textbf{All reaction--diffusion quotient runs.} \textbf{a}, Singular values of the Jacobian of $(a,g,h)$. \textbf{b}, Simulated over predicted amplitude for every twin kernel, by target and seed; crosses mark runs whose pattern differs from the prediction. \textbf{c}, 24 random sixteen-coefficient kernels without projection: simulated against predicted amplitude. \textbf{d}, Spread among twins against the common offset from leading order.}\label{fig:rdsi}
\end{figure}
\subsection{Where the quotient stops}
Two of the five T3 twins, with largest coefficients 1.67 and 1.30, also possess a hexagonal state
of larger amplitude ($+24\%$ and $+41\%$, skewness above one), which captures their noise-seeded
runs. The cubic amplitude equations do not describe this branch; it is set by the silent
directions. Kernels whose construction forces large coefficients leave the regime of the quotient:
at T1 the activator-only kernel (largest coefficient 5.9) settles into hexagons and the
inhibitor-only kernel (8.0) falls to a different uniform state or diverges; at T4 the
inhibitor-only kernel (4.1) keeps stripes at $x=2$. Among 24 kernels drawn at random in the
16-coefficient space without projection, 66 of 72 runs give the predicted pattern, with a median
amplitude error of 2.8\%; the exceptions lie near the stripe-loss threshold ($x=1.0$--$1.3$, where
the leading-order threshold $x\simeq1.0$ underestimates the observed $1.48$--$1.60$) or at the
largest amplitudes (Supplementary Fig.~\ref{fig:rdsi}c).

\suppnote{Data behind the numbers quoted in the main text}

\label{app:data}
\begin{table}[htbp]\centering
\caption{\textbf{Interfacial sensitivity $W_m=\partial\Real\lambda/\partial\gamma_m$.} Values at the base
point ($\gamma_1=0.1$, reverse von Mises $\kappa=5$), at two angular resolutions and, for
$m\le2$, at a second composition; n.c., not computed. Extended Data Fig.~5a of the main text.}\label{tab:W}
\footnotesize\setlength{\tabcolsep}{4pt}
\begin{tabular}{lrrrrrrrr}
\toprule
$m$ & 1 & 2 & 3 & 4 & 5 & 6 & 7 & 8\\
\midrule
$A=64$, comp $0.75/1.25$ & $-2.0623$ & $-10.5762$ & $1.1174$ & $-2.2075$ & $-0.9826$ & $0.5244$ & $-1.1011$ & $0.4322$\\
$A=128$, comp $0.75/1.25$ & $-1.8192$ & $-10.0406$ & $0.9785$ & $-1.9877$ & $-0.9232$ & $0.5183$ & $-1.0448$ & $0.4025$\\
$A=128$, comp $0.85/1.15$ & $-0.7916$ & $-3.4523$ & n.c. & n.c. & n.c. & n.c. & n.c. & n.c.\\
\bottomrule
\end{tabular}
\end{table}
\begin{table}[htbp]\centering
\caption{\textbf{Bulk-matched witness.} The two kernels share $\gamma_1$ and $\gamma_2$ to $2.4\times10^{-17}$
and the same composition. Each sits within $2.0000013$ of the common base pointwise; between
themselves the largest ratio is $3.5412$. Extended Data Fig.~5c of the main text.}\label{tab:witness}
\footnotesize\setlength{\tabcolsep}{5pt}
\begin{tabular}{lll}
\toprule
grid & kernel $q_+$ & kernel $q_-$\\
\midrule
$A=64$ (dense Newton, exact spectrum) & unstable, $\Real\lambda=+0.009408$ & stable, $\Real\lambda\approx-0.05$\\
$A=96$ & breathing orbit & stable, $\Real\lambda=-0.047$\\
$A=128$ & breathing orbit & stable\\
$A=160$ & breathing orbit & stable, $\Real\lambda=-0.053$\\
\midrule
drift of $\Real\lambda$, measured & $+0.0256$ & $-0.034$\\
drift of $\Real\lambda$, linear prediction & $+0.0284$ & $-0.0300$\\
\bottomrule
\end{tabular}
\end{table}
\begin{table}[htbp]\centering
\caption{\textbf{Quasi-one-dimensional particle runs ($L_y=2$, $16$ replicas): band-speed spectral line
against the kinetic breathing frequency $0.269$--$0.281$.}}\label{tab:particles}
\begin{tabular}{lccc}
\toprule
run & line centre $\omega$ & peak-to-background & band speed vs kinetic\\
\midrule
$\kappa=20$, $\rho_0=640$ & $0.243$ & $2.0$ & $0.5$--$1.5\%$\\
$\kappa=20$, $\rho_0=2560$ & $0.2763$ & $7.9$ & $0.5$--$1.5\%$\\
$\kappa=2$ control, $\rho_0=640$ & no line & $1.3$ & n.a.\\
\bottomrule
\end{tabular}
\end{table}
\begin{table}[htbp]\centering
\caption{\textbf{One fixed physical kernel pair under grid refinement.} The kernel definition, a
$\varepsilon=0.02$ step in the linear family at the threshold, is the same at every resolution; it
is not re-optimised per grid. Convergence is non-monotone at the coarsest angular grids, but the
difference of the two rates is nearly grid independent and the signs never coincide.}
\label{tab:grid}
\footnotesize\setlength{\tabcolsep}{6pt}
\begin{tabular}{ccrrr}
\toprule
$N$ & $A$ & $\Real\lambda_+$ & $\Real\lambda_-$ & $\Real\lambda_+-\Real\lambda_-$\\
\midrule
$192$ & $96$ & $+9.2866\times10^{-5}$ & $-1.4724\times10^{-3}$ & $1.5653\times10^{-3}$\\
$96$ & $128$ & $+7.8267\times10^{-4}$ & $-7.8565\times10^{-4}$ & $1.5683\times10^{-3}$\\
$128$ & $128$ & $+7.8238\times10^{-4}$ & $-7.8598\times10^{-4}$ & $1.5684\times10^{-3}$\\
$128$ & $256$ & $+6.0263\times10^{-4}$ & $-9.6268\times10^{-4}$ & $1.5653\times10^{-3}$\\
$128$ & $384$ & $+5.8640\times10^{-4}$ & $-9.7982\times10^{-4}$ & $1.5662\times10^{-3}$\\
$128$ & $512$ & $+5.8763\times10^{-4}$ & $-9.7856\times10^{-4}$ & $1.5662\times10^{-3}$\\
\bottomrule
\end{tabular}
\end{table}
\begin{table}[htbp]\centering
\caption{\textbf{Direct nonlinear integration of the full angle-even kinetic model against
Eqs.~\eqref{eq:obs}--\eqref{eq:freq}, with $K_\rho$ and $G_0$ fixed in advance at the reference
point and nothing refitted per trajectory.} Frequency errors are on the shift from $\omega_0$, not
on the carrier. The last column is the power residual divided by $\varepsilon^2$, which
Eq.~\eqref{eq:obs} requires to be asymptotically constant.}\label{tab:dns}
\footnotesize\setlength{\tabcolsep}{5pt}
\begin{tabular}{crrrrr}
\toprule
$\varepsilon$ & $I_\rho$ predicted & $I_\rho$ measured & power err. & freq.\ shift err. &
 residual$/\varepsilon^2$\\
\midrule
$0.02$ & $2.89563\times10^{-4}$ & $2.90096\times10^{-4}$ & $0.184\%$ & $0.322\%$ & $1.332\times10^{-3}$\\
$0.04$ & $5.79128\times10^{-4}$ & $5.81360\times10^{-4}$ & $0.385\%$ & $0.661\%$ & $1.395\times10^{-3}$\\
$0.08$ & $1.15826\times10^{-3}$ & $1.16728\times10^{-3}$ & $0.779\%$ & $1.339\%$ & $1.410\times10^{-3}$\\
\bottomrule
\end{tabular}
\end{table}
\begin{table}[htbp]\centering
\caption{\textbf{A matched direction designed to be growth-silent and frequency-active,
$\chi_R\simeq0$ and $\chi_I=3.892786\times10^{-3}$ per unit amplitude.} The real drift is even in
$\varepsilon$ and scales as $\varepsilon^2$; the imaginary shift is odd and scales as
$\varepsilon$.}\label{tab:silent}
\footnotesize\setlength{\tabcolsep}{6pt}
\begin{tabular}{rrrr}
\toprule
$\varepsilon$ & $\Delta\Real\lambda$ & $\Delta\Imag\lambda$ & predicted $\Delta\Imag\lambda$\\
\midrule
$-0.04$ & $+9.1504\times10^{-8}$ & $-1.557039\times10^{-4}$ & $-1.557114\times10^{-4}$\\
$+0.04$ & $+9.1314\times10^{-8}$ & $+1.557184\times10^{-4}$ & $+1.557114\times10^{-4}$\\
$-0.08$ & $+3.6638\times10^{-7}$ & $-3.113913\times10^{-4}$ & $-3.114228\times10^{-4}$\\
$+0.08$ & $+3.6487\times10^{-7}$ & $+3.114493\times10^{-4}$ & $+3.114228\times10^{-4}$\\
\bottomrule
\end{tabular}
\end{table}
\begin{table}[htbp]\centering
\caption{\textbf{The transect through the Hopf line of the response map, at fixed $a_2=0.017388$.} Nonlinear
entries are full periodic solves at $N=128$, $A=128$ with $17$ temporal nodes, warm-started from design~B,
with the kernel fixed and neither power nor frequency imposed; linear entries are the closed-form response
map. The two points on the steady side return no orbit on the tracked branch (final residuals of about
$1.6\times10^{-2}$ and $9\times10^{-3}$).}\label{tab:transect}
\footnotesize
\begin{tabular}{rrrrrr}
\toprule
$a_1$ & $I_\rho$ (solve) & $I_\rho$ (linear) & difference & $\Omega_{\rm solve}-\Omega_{\rm lin}$ & off-grid defect\\
\midrule
$-0.0201$ & \multicolumn{5}{l}{no orbit found on the tracked branch}\\
$-0.0081$ & \multicolumn{5}{l}{no orbit found on the tracked branch}\\
$+0.0039$ & $5.736\times10^{-5}$ & $5.791\times10^{-5}$ & $-0.96\%$ & $+9.5\times10^{-6}$ & $4.3\times10^{-7}$\\
$+0.0079$ & $1.1444\times10^{-4}$ & $1.1583\times10^{-4}$ & $-1.20\%$ & $+5.4\times10^{-6}$ & $1.8\times10^{-10}$\\
$+0.0149$ & $2.1597\times10^{-4}$ & $2.1717\times10^{-4}$ & $-0.55\%$ & $+4.7\times10^{-6}$ & $1.1\times10^{-15}$\\
$+0.0249$ & $3.6124\times10^{-4}$ & $3.6196\times10^{-4}$ & $-0.20\%$ & $+2.8\times10^{-6}$ & $1.6\times10^{-13}$\\
$+0.0399$ & $5.7965\times10^{-4}$ & $5.7913\times10^{-4}$ & $+0.09\%$ & $-2.0\times10^{-6}$ & $4.2\times10^{-14}$\\
\bottomrule
\end{tabular}
\end{table}

\suppnote{Numerical record behind the Methods}

\noindent This note collects the discretizations, residuals and intermediate numbers that the Methods summarise. Figure numbers refer to the main text.

\subsection{Witness kernels and interfacial sensitivity}
The kernels of Fig.~1a,c belong to the design family described below and share the normalization, $\gamma_1$ and $\gamma_2$ of $q_0$ to $10^{-18}$. The two kernels of Extended Data Fig.~5c share $\gamma_1$ and $\gamma_2$ to $2.4\times10^{-17}$;
each lies within a factor $2.0000013$ of the common base pointwise, and their largest mutual ratio
is $3.5412$. The verdict survives angular refinement from $A=96$ to $512$ modes. A separate
near-threshold pair differs pointwise by at most $1.0408$ and has interfacial growth rates
$+7.8239\times10^{-4}$ and $-7.8598\times10^{-4}$ at the reference discretization; with the
kernels fixed, the signs survive refinement from $A=96$ to $512$, and the difference of the rates
stays within $0.2\%$ of $1.566\times10^{-3}$. The sensitivity $W_m$ was obtained by first-order
perturbation of the band's leading eigenvalue, including the change of the
self-consistent band; at $A=64$ the second-harmonic value $-10.58$ separates into $+0.46$ of
direct damping and $-11.0$ of shape feedback. The values at 64 and 128 modes are tabulated in
Supplementary Note 12.

\subsection{Prediction loss and response landscape}
$D_M$ in Fig.~2c uses the response table at 128 angular modes truncated at harmonic
$m=8$; its values are $0.1670$, $0.0611$, $0.0602$, $0.0481$ and $0.0361$ for $M=0,\ldots,4$. The
minimax statement is that the smallest worst-case error of any predictor using only the matched
rates equals $\varepsilon D_M+O(\varepsilon^2)$ on kernels $q_0(1+h)$ with
$\|h\|_\infty\le\varepsilon$ (Supplementary Note 4). A $9\times9$ grid of fixed five-primitive kernels at $N=32$ spatial and $A=64$ angular modes on the
periodic branch gave 81 converged solves, with power between $1.23\times10^{-3}$ and
$1.36\times10^{-3}$ and frequency between $0.2958$ and $0.2962$ (Supplementary Note 12). The transect of
Fig.~2a is a line of kernels at fixed
second coordinate $a_2=0.01739$ with $a_1$ from $-0.0201$ to $+0.0399$, crossing the Hopf line of
the linear map, the full periodic solve at the five points on the oscillatory side returns the
power within $0.96\%$, $1.20\%$, $0.55\%$, $0.20\%$ and $0.09\%$ of the linear response map and the
frequency within $10^{-5}$; at the two points on the steady side the solver returns no orbit on
the tracked branch, and the linear Hopf line at $a_1=-1\times10^{-4}$ lies inside the nonlinear
bracket $(-0.0081,+0.0039)$. The Floquet moduli of the five orbits at 512 steps per period are
$0.99356$, $0.98694$, $0.97535$, $0.95878$ and $0.93386$, with neutral-mode relative errors below
$2.8\times10^{-6}$ and leading Ritz residuals below $4.5\times10^{-15}$.

\subsection{Local predictor}
Near the simple supercritical Hopf bifurcation of the band, the complex eigenvalue
response and the computed cubic coefficient predict oscillation power and frequency shift. The
coefficients $K_\rho$ and $G_0$ of the reference state were fixed in advance and no per-trajectory
fitting was used. Along the power-maximizing matched direction at kernel budgets
$\varepsilon=0.02$, $0.04$ and $0.08$, predicted and integrated powers agree to $0.18\%$, $0.39\%$
and $0.78\%$, and the nonlinear frequency shifts, measured from $\omega_0$, to $0.32\%$, $0.66\%$
and $1.34\%$; the power residual divided by $\varepsilon^2$ is $1.33$, $1.40$ and
$1.41\times10^{-3}$. On the stable side the band returns to a steady travelling wave at rate
$-0.0031668$ against the linear $-0.0031661$. A matched direction with $\chi_R\simeq0$ and
$\chi_I=3.893\times10^{-3}$ per unit amplitude shows growth drift scaling by $4.0$ and frequency
drift by $2.0$ when its amplitude doubles, and a third kernel designed to share the growth response
of the $\varepsilon=0.04$ kernel but not its $\chi_I$ gives equal powers (difference $0.054\%$) and
a frequency difference $1.5656\times10^{-4}$ against the predicted $1.5571\times10^{-4}$. As a
separate test of the calibrated linear predictor, ten targets drawn at random were solved at $N=96$, $A=256$; all converged, and the largest differences between
prediction and forward output were $2.50\times10^{-9}$ in $I$ and $9.65\times10^{-9}$ in $\Omega$,
with matched moments within $2.3\times10^{-16}$ (Extended Data Fig.~2a).

\subsection{Restricted five-primitive family}
A restricted realization uses five symmetric von Mises turning primitives of concentration $20$
centred at $\pi\pm\delta_j$, $\delta_j=0,0.35,0.70,1.05,1.40$, of which only the selection
probabilities vary; turning-angle statistics of this kind are measurable for run-and-tumble
bacteria. Normalization and two moment constraints leave two coordinates,
and the measured response realizes $\Gamma_b=2$. Exact rational enclosures certify positivity and
moment matching for the delivered kernels. For two fixed-frequency designs the same kernels at
$A=512$ reproduce the requested power increment to $0.025\%$ and the frequency match to
$1.9\times10^{-8}$, while their absolute powers retain approximately $1\%$ finite-grid bias.
Targets at $70\%$, $98\%$ and $120\%$ of the linear fixed-frequency power limit require
probability-change budgets within $0.0031\%$ of the local prediction; the exterior target is solved
but exceeds the imposed budget. For the designed orbit a perturbation of $20\%$ of the oscillatory
density component leaves $1.05\times10^{-4}$ of its initial orbital distance after $140$ cycles;
the leading nontrivial Floquet multiplier in the mass-constrained, angle-even calculation is
$0.94403$, whose decay exponent $-0.0027264$ agrees with the measured $-0.0027100$. A second
design in the same family has a converged multiplier of $0.93160$ at $N=A=128$ with residuals
below $5.7\times10^{-8}$. Integrating over the periodic cell removes spatial fluxes, so each global
angular moment obeys $\dot A_{s,m}=-\gamma_m A_{s,m}$; any uniform asymptotic bound
$d(t)\le Ce^{-\rho t}d(0)$ on a state distance that observes these moments has
$\rho\le\min_{m\le M}\gamma_m$, and in the restricted family $\gamma_2^{-1}\simeq58.05$
(Supplementary Note 6).

\subsection{Joint inverse design and independent integration}
The kernel family of Fig.~4 is $q=q_0(1+h)$, $h=\sum_{j<33}c_j\cos j\varphi$, at
the critical Hopf reference ($\kappa_q=5.5122$, $N=A=128$), with the normalization, $\gamma_1$ and
$\gamma_2$ of $q_0$ imposed through the exact Bessel moments of $q_0\cos j\varphi$. The two hidden
coordinates $a=(a_1,a_2)$ multiply the minimum-budget direction of the linear design for the target
and the direction of equal budget perpendicular to it in response space; the linear design is
$a=(1,0)$. The complex eigenvalue response of each direction, including the self-consistent
reshaping of the interface, was computed with the linearised kinetic solver. The joint problem
solves the periodic orbit, represented by temporal Fourier collocation on 17 nodes, together with
$a$, with the phase condition, the target power and the target frequency as additional equations,
by Newton iteration with a Schur complement on the two kernel coordinates (Supplementary Note
9). The targets were $I^*=5\times10^{-4}$ with $\Omega^*=0.2965$ and $0.2985$, and
$I^*=10^{-3}$ with $\Omega^*=0.2975$, against a reference frequency near $0.2974$. All seven
solves, including three boundary targets at $I^*=2$, $3$ and $4\times10^{-3}$, converged from the
linear design in three or four iterations to residual norms between $5\times10^{-13}$ and
$4\times10^{-11}$; the matched moments hold to $10^{-18}$. The distance of the nonlinear
coordinates from $(1,0)$ is $0.016$, $0.003$, $0.003$ and $0.005$ for m1, m2, m3 and m2$'$ and
$0.007$, $0.011$ and $0.016$ for the boundary targets. The kernel budget is the supremum of $|h|$
on $65{,}536$ angles: $0.067$, $0.069$, $0.081$ and $0.086$ for the four designs of the main text,
and $0.204$, $0.330$ and $0.457$ for the boundary targets, whose kernel minima fall to $0.80$,
$0.67$ and $0.54$ of $q_0$. The forward integrations use the same exponential time stepper as the
periodic solver with 1,024 steps per target period over 8,000 time units. The initial state is the
steady interface of the reference kernel m0, a steady-side design with the same moments and a
budget of $0.029$, plus $0.1\sqrt{5\times10^{-4}/2Q_2}$ times the critical eigenvector, where
$Q_2$ is the power per unit squared amplitude; the drift is removed at each output by projection on
the translation mode, and the critical-mode amplitude $b$ is the projection of the deviation from
the steady interface on the left eigenvector. Under m0 the amplitude decays from $0.58$ to $0.11$
over 4,000 time units and the temporal standard deviation of the density over the final window is
$1.3\times10^{-5}$. Frequency and power over the last 16 periods come from a harmonic fit with six
harmonics; the direct power is the density variance about the temporal mean over the same window.
The fitted frequencies are $0.29650000$, $0.29850000$, $0.29750000$ and $0.29850000$ and the
fitted powers $5.000\times10^{-4}$, $5.000\times10^{-4}$, $1.0000\times10^{-3}$ and
$5.000\times10^{-4}$. Orbit validation runs sixteen cycles of the same stepper from the solved
orbit and reports the relative orbit closure: $5\times10^{-7}$ for m1, m2 and m2$'$,
$3\times10^{-6}$ for m3, and $3.9\times10^{-5}$, $1.7\times10^{-4}$ and $4.3\times10^{-4}$ for the
boundary targets, for which the integration from the perturbed steady interface was not run.
Floquet multipliers come from Arnoldi iteration on the period map at 512 steps per period with
eight requested values; the leading moduli are $0.9435$, $0.9423$, $0.8845$ and $0.9422$, with
residuals below $5\times10^{-15}$ and neutral-mode relative errors below $1.6\times10^{-5}$. For
the boundary target with $I^*=4\times10^{-3}$ the neutral-mode check returned $1.9\times10^{-3}$
at both 512 and 1,024 steps, above the $10^{-3}$ tolerance, and the calculation was rejected; the
stability of that orbit is not established.

\subsection{Constraint sweep}
For each $M$ from 2 to 10 the hidden subspace of the 33-harmonic family was constrained to preserve
$\gamma_1,\ldots,\gamma_M$, and the following were computed: the capacity
$\Gamma_b=\operatorname{rank}(\delta\operatorname{Re}\lambda,\delta\operatorname{Im}\lambda)$ on
the hidden subspace; the region of $(\delta\operatorname{Re}\lambda,\delta\operatorname{Im}\lambda)$
reachable with $\|h\|_\infty\le1$, by linear programming of its support function; its growth
half-range $D_M^{(33)}$; and, for 400 random targets in the box $I\in[4,7]\times10^{-4}$,
$\Omega\in[0.2972,0.2978]$, the smallest budget of a first-order design, with the fixed linear map
from $(\delta\operatorname{Re}\lambda,\delta\operatorname{Im}\lambda)$ to $(I,\Omega)$ of the Hopf
normal form. The reachable area falls from $2.5\times10^{-3}$ at $M=2$ to $6.0\times10^{-4}$,
$7.4\times10^{-5}$ and $1.2\times10^{-5}$ at $M=4$, 6 and 8, and $D_M^{(33)}$ from
$5.3\times10^{-2}$ to $1.1\times10^{-3}$ at $M=8$, after which it levels off; first-order designs within $\|h\|_\infty<1$
reach all 400 targets at $M\le6$, $24.5\%$ at $M=7$ and none at $M\ge8$. The sweep was repeated
with 9, 13, 17, 25 and 33 harmonics (Extended Data Fig.~3). Thirty-one fixed-kernel designs, three
fixed targets at $M=2$--$4$ in the nine-harmonic family, eight held-out targets drawn at random before any kinetic solve at $M=3$ and 4, and the three fixed targets at
$M=5$ and 6 in the 33-harmonic family, were tested by the complete one-cell periodic kinetic
boundary-value problem with the kernel fixed (17 temporal nodes, $N=A=128$); all converged, and
the realised power lies within $0.44\%$ of the target for the nine-harmonic designs and within
$0.93\%$ for the 33-harmonic designs (Extended Data Fig.~2b; Supplementary Note 8).

\subsection{Reaction--diffusion family and simulations}
The family of Fig.~3 has $J=\left(\begin{smallmatrix}0.8&-1\\1&-1\end{smallmatrix}\right)$
and $D=\operatorname{diag}(1,3.5)$, so that the uniform state is Turing-unstable at $k_c=0.5071$
with growth rate $\sigma=0.02278$, and $\max_k|\Delta\lambda|\le7\times10^{-13}$ across the
family. The four-coefficient family used for the phase plane and the designs is
$N_\theta=(\eta_2U^2+\eta_{11}UV-\beta_3U^3-\beta_5U^5,\,0)$, in which $\eta_2$ and $\eta_{11}$
are the quadratic feedbacks and $\beta_3$ the saturation of the autocatalytic step; at
$\eta_2=\eta_{11}=0$ one has $h=2g$ and $\Gamma_b=2$. The reaction--diffusion equation is integrated
pseudo-spectrally with fourth-order exponential time differencing, the exponential
applied to diffusion and the whole reaction treated explicitly, on $L_x=8\pi/k_c$,
$L_y=16\pi/(\sqrt3k_c)$ with $64\times74$ modes, two-thirds dealiasing and $\Delta t=0.5$; the
linear growth rate of the lattice mode at $k_c$ is reproduced as $0.022776$. Morphology is read
from the number of Fourier pairs near $k_c$ above one fifth of the strongest, polarity from the
skewness of $U$ ($+0.86$ for spots at $\eta_2=+0.3$, $-0.85$ for holes at $\eta_2=-0.3$), and the
oblique growth rate from an exponential fit to a $10^{-3}$ perturbation of a converged stripe.
With $\eta_2=0$ the stripe amplitude follows $\sqrt{\sigma/g}\propto\beta_3^{-1/2}$ to within
$1.8\%$ across a factor of ten in $\beta_3$; the amplitude equations place hexagon stability at
$|\eta_2|=0.0559$ and stripe loss at $0.1134$, and simulation confirms the first (hexagons seeded
at $0.045$ decay to stripes, those at $0.06$ persist) and locates the second between $0.150$ and
$0.175$. Raising $\beta_5$ from 0 to 12 at two points of the family leaves every pattern in place
and moves the amplitude by at most $14\%$, the size of the $O(A^2)$ correction a quintic term
produces (Supplementary Note 11). The Jacobian of $(a,g,h)$ was evaluated by central differences
at 200 random kernels of the four- and sixteen-coefficient families and 500 of the 22-coefficient
family. Twin kernels were drawn at random in the sixteen-coefficient space and projected onto
$\Phi^{-1}(a,g,h)$ by minimum-norm Gauss--Newton steps with every coefficient kept below two in
magnitude; six target triples, five kernels each, were run from stripe, hexagon and noise initial
conditions, and within a group the amplitudes differ by $0.1$--$1.8\%$ while the group can sit up
to $15\%$ below the leading-order value. Every kernel of two targets was rerun with two further
seeds, which changed no pattern from stripe or hexagon seeds and the amplitudes by less than
$10^{-5}$. Stripes persist up to $x=a/\sqrt{\sigma g}=1.484$ and are lost by $1.60$, and the
hexagon amplitude error grows from $-3\%$ at $x=0.9$ to $-19\%$ at $x=2.4$. For the amplitude
designs, the closed-form solution for $(\eta_2,\beta_3)$ at the targets
$(A_{\rm stripe},A_{\rm hexagon})=(0.060,0.031)$ and $(0.090,0.050)$ gives, in simulation,
$-0.8\%$ and $-1.1\%$, and $-2.3\%$ and $-5.7\%$; one correction that shifts the targets by the
measured offset brings these to $0.0\%$ and $0.0\%$, and $-0.7\%$ and $-2.4\%$.


\begin{thebibliography}{99}
\bibitem{wilson} Wilson, K. G. The renormalization group: critical phenomena and the Kondo problem. \emph{Rev. Mod. Phys.} \textbf{47}, 773--840 (1975).
\bibitem{marchetti} Marchetti, M. C. et al. Hydrodynamics of soft active matter. \emph{Rev. Mod. Phys.} \textbf{85}, 1143--1189 (2013).
\bibitem{bechinger} Bechinger, C. et al. Active particles in complex and crowded environments. \emph{Rev. Mod. Phys.} \textbf{88}, 045006 (2016).
\bibitem{noid} Noid, W. G. Perspective: Coarse-grained models for biomolecular systems. \emph{J. Chem. Phys.} \textbf{139}, 090901 (2013).
\bibitem{fruchart} Fruchart, M., Hanai, R., Littlewood, P. B. \& Vitelli, V. Non-reciprocal phase transitions. \emph{Nature} \textbf{592}, 363--369 (2021).
\bibitem{you} You, Z., Baskaran, A. \& Marchetti, M. C. Nonreciprocity as a generic route to traveling states. \emph{Proc. Natl Acad. Sci. USA} \textbf{117}, 19767--19772 (2020).
\bibitem{duan} Duan, Y., Agudo-Canalejo, J., Golestanian, R. \& Mahault, B. Phase coexistence in nonreciprocal quorum-sensing active matter. \emph{Phys. Rev. Research} \textbf{7}, 013234 (2025).
\bibitem{turing} Turing, A. M. The chemical basis of morphogenesis. \emph{Phil. Trans. R. Soc. Lond. B} \textbf{237}, 37--72 (1952).
\bibitem{crosshohenberg} Cross, M. C. \& Hohenberg, P. C. Pattern formation outside of equilibrium. \emph{Rev. Mod. Phys.} \textbf{65}, 851--1112 (1993).
\bibitem{machta} Machta, B. B., Chachra, R., Transtrum, M. K. \& Sethna, J. P. Parameter space compression underlies emergent theories and predictive models. \emph{Science} \textbf{342}, 604--607 (2013).
\bibitem{transtrum} Transtrum, M. K. et al. Perspective: Sloppiness and emergent theories in physics, biology, and beyond. \emph{J. Chem. Phys.} \textbf{143}, 010901 (2015).
\bibitem{zwanzig} Zwanzig, R. Memory effects in irreversible thermodynamics. \emph{Phys. Rev.} \textbf{124}, 983--992 (1961).
\bibitem{mori} Mori, H. Transport, collective motion, and Brownian motion. \emph{Prog. Theor. Phys.} \textbf{33}, 423--455 (1965).
\bibitem{torquato} Torquato, S. Inverse optimization techniques for targeted self-assembly. \emph{Soft Matter} \textbf{5}, 1157 (2009).
\bibitem{sherman} Sherman, Z. M., Howard, M. P., Lindquist, B. A., Jadrich, R. B. \& Truskett, T. M. Inverse methods for design of soft materials. \emph{J. Chem. Phys.} \textbf{152}, 140902 (2020).
\bibitem{molesky} Molesky, S. et al. Inverse design in nanophotonics. \emph{Nat. Photon.} \textbf{12}, 659--670 (2018).
\bibitem{saha} Saha, S., Agudo-Canalejo, J. \& Golestanian, R. Scalar active mixtures: the non-reciprocal Cahn--Hilliard model. \emph{Phys. Rev. X} \textbf{10}, 041009 (2020).
\bibitem{brauns} Brauns, F. \& Marchetti, M. C. Nonreciprocal pattern formation of conserved fields. \emph{Phys. Rev. X} \textbf{14}, 021014 (2024).
\bibitem{duanprl} Duan, Y., Agudo-Canalejo, J., Golestanian, R. \& Mahault, B. Dynamical pattern formation without self-attraction in quorum-sensing active matter: the interplay between nonreciprocity and motility. \emph{Phys. Rev. Lett.} \textbf{131}, 148301 (2023).
\bibitem{dinelliNR} Dinelli, A. et al. Non-reciprocity across scales in active mixtures. \emph{Nat. Commun.} \textbf{14}, 7035 (2023).
\bibitem{tailleur08} Tailleur, J. \& Cates, M. E. Statistical mechanics of interacting run-and-tumble bacteria. \emph{Phys. Rev. Lett.} \textbf{100}, 218103 (2008).
\bibitem{baeuerle} B\"auerle, T., Fischer, A., Speck, T. \& Bechinger, C. Self-organization of active particles by quorum sensing rules. \emph{Nat. Commun.} \textbf{9}, 3232 (2018).
\bibitem{lefranc} Lefranc, T. et al. Synthetic quorum sensing and absorbing phase transitions in colloidal active matter. \emph{Phys. Rev. X} \textbf{15}, 031050 (2025).
\bibitem{cates} Cates, M. E. \& Tailleur, J. When are active Brownian particles and run-and-tumble particles equivalent? Consequences for motility-induced phase separation. \emph{EPL} \textbf{101}, 20010 (2013).
\bibitem{solon} Solon, A. P., Cates, M. E. \& Tailleur, J. Active Brownian particles and run-and-tumble particles: a comparative study. \emph{Eur. Phys. J. Spec. Top.} \textbf{224}, 1231--1262 (2015).
\bibitem{dinelli} Dinelli, A., O'Byrne, J. \& Tailleur, J. Fluctuating hydrodynamics of active particles interacting via taxis and quorum sensing: static and dynamics. \emph{J. Phys. A: Math. Theor.} \textbf{57}, 395002 (2024).
\bibitem{micchelli} Micchelli, C. A. \& Rivlin, T. J. A survey of optimal recovery. In \emph{Optimal Estimation in Approximation Theory} 1--54 (Springer, 1977).
\bibitem{grayscott} Gray, P. \& Scott, S. K. Autocatalytic reactions in the isothermal, continuous stirred tank reactor: oscillations and instabilities in the system $A+2B\to3B$, $B\to C$. \emph{Chem. Eng. Sci.} \textbf{39}, 1087--1097 (1984).
\bibitem{epstein} Epstein, I. R. \& Pojman, J. A. \emph{An Introduction to Nonlinear Chemical Dynamics: Oscillations, Waves, Patterns, and Chaos} (Oxford Univ. Press, 1998).
\bibitem{hoyle} Hoyle, R. \emph{Pattern Formation: An Introduction to Methods} (Cambridge Univ. Press, 2006).
\bibitem{ouyang} Ouyang, Q. \& Swinney, H. L. Transition from a uniform state to hexagonal and striped Turing patterns. \emph{Nature} \textbf{352}, 610--612 (1991).
\bibitem{kondo} Kondo, S. \& Miura, T. Reaction-diffusion model as a framework for understanding biological pattern formation. \emph{Science} \textbf{329}, 1616--1620 (2010).
\bibitem{palacci} Palacci, J., Sacanna, S., Steinberg, A. P., Pine, D. J. \& Chaikin, P. M. Living crystals of light-activated colloidal surfers. \emph{Science} \textbf{339}, 936--940 (2013).
\bibitem{frangipane} Frangipane, G. et al. Dynamic density shaping of photokinetic \emph{E. coli}. \emph{eLife} \textbf{7}, e36608 (2018).
\bibitem{lavergne} Lavergne, F. A., Wendehenne, H., B\"auerle, T. \& Bechinger, C. Group formation and cohesion of active particles with visual perception-dependent motility. \emph{Science} \textbf{364}, 70--74 (2019).
\bibitem{risken} Risken, H. \emph{The Fokker--Planck Equation: Methods of Solution and Applications} 2nd edn (Springer, 1996).
\bibitem{levermore} Levermore, C. D. Moment closure hierarchies for kinetic theories. \emph{J. Stat. Phys.} \textbf{83}, 1021--1065 (1996).
\bibitem{maddu} Maddu, S., Weady, S. \& Shelley, M. J. Learning fast, accurate, and stable closures of a kinetic theory of an active fluid. \emph{J. Comput. Phys.} \textbf{504}, 112869 (2024).
\bibitem{peraud} P\'eraud, J.-P. M. \& Hadjiconstantinou, N. G. Extending the range of validity of Fourier's law into the kinetic transport regime via asymptotic solution of the phonon Boltzmann transport equation. \emph{Phys. Rev. B} \textbf{93}, 045424 (2016).
\bibitem{zhaoyong} Zhao, W. \& Yong, W.-A. Boundary conditions for kinetic theory-based models II: a linearized moment system. \emph{Math. Methods Appl. Sci.} \textbf{44}, 14148--14172 (2021).
\bibitem{kato} Kato, T. \emph{Perturbation Theory for Linear Operators} 2nd edn (Springer, 1995).
\bibitem{trefethen} Trefethen, L. N. \emph{Spectral Methods in MATLAB} (SIAM, 2000).
\bibitem{govaerts} Govaerts, W. J. F. \emph{Numerical Methods for Bifurcations of Dynamical Equilibria} (SIAM, 2000).
\bibitem{sandstede} Sandstede, B. Stability of travelling waves. In \emph{Handbook of Dynamical Systems} Vol.~2 (ed. Fiedler, B.) 983--1055 (Elsevier, 2002).
\bibitem{beyn} Beyn, W.-J. An integral method for solving nonlinear eigenvalue problems. \emph{Linear Algebra Appl.} \textbf{436}, 3839--3863 (2012).
\bibitem{knoll} Knoll, D. \& Keyes, D. Jacobian-free Newton--Krylov methods: a survey of approaches and applications. \emph{J. Comput. Phys.} \textbf{193}, 357--397 (2004).
\bibitem{saad} Saad, Y. \& Schultz, M. H. GMRES: a generalized minimal residual algorithm for solving nonsymmetric linear systems. \emph{SIAM J. Sci. Stat. Comput.} \textbf{7}, 856--869 (1986).
\bibitem{hochbruck} Hochbruck, M. \& Ostermann, A. Exponential integrators. \emph{Acta Numer.} \textbf{19}, 209--286 (2010).
\bibitem{hockney} Hockney, R. W. \& Eastwood, J. W. \emph{Computer Simulation Using Particles} (CRC Press, 1988).
\bibitem{kuznetsov} Kuznetsov, Y. A. \emph{Elements of Applied Bifurcation Theory} (Springer, 2023).
\end{thebibliography}
\end{document}